\documentclass[a4paper,11pt]{article}

\usepackage{amsmath,amsthm}
\usepackage{amssymb}
\usepackage{breqn}
\usepackage{enumerate}
\usepackage{graphicx}
\usepackage{bm}
\usepackage{bbm}
\usepackage[affil-it]{authblk}
\usepackage{tabu}
\usepackage{bold-extra}
\usepackage[hang,flushmargin]{footmisc}
\usepackage[driverfallback=dvipdfm]{hyperref}
\usepackage{mathtools}
\usepackage{relsize}
\usepackage{scalerel}
\usepackage{xcolor}
\usepackage{pagecolor}
\usepackage{titlesec}
\usepackage{apptools}
\usepackage{appendix}
\usepackage[
    backref=true,
    isbn=false,
    maxbibnames=99,
    sorting=nyt
]{biblatex}
\renewbibmacro{in:}{}
\renewbibmacro*{doi+eprint+url}{%
  \printfield{doi}%
  \newunit\newblock%
  \iftoggle{bbx:eprint}{%
    \usebibmacro{eprint}%
  }{}%
  \newunit\newblock%
  \iffieldundef{doi}{%
    \iffieldequalstr{eprinttype}{arXiv}
      {}
      {\usebibmacro{url+urldate}}%
  }{}%
}
\AtEveryBibitem{\clearlist{language}}
\DefineBibliographyStrings{english}{
    backrefpage = {page},
    backrefpages = {pages}
}

\usepackage{xpatch}
\DeclareFieldFormat{backrefparens}{\addperiod\raisebox{4pt}{\scriptsize{#1}}}
\xpatchbibmacro{pageref}{parens}{backrefparens}{}{}

\newtheorem{theorem}{Theorem}
\newtheorem{corollary}[theorem]{Corollary}
\newtheorem{lemma}[theorem]{Lemma}
\newtheorem{proposition}[theorem]{Proposition}

\newtheorem{claim}[theorem]{Claim}
\newtheorem*{claim*}{Claim}

\theoremstyle{definition}
\newtheorem{defin}[theorem]{Definition}

\newtheorem*{rem*}{Remark}

\titleformat{\section}[hang]{\scshape\large\bfseries\filcenter}{\S\thesection}{4pt}{}
\titleformat{\subsection}[hang]{\scshape\bfseries}{\thesubsection.}{4pt}{}
\titleformat{\subsubsection}[hang]{\scshape\bfseries}{\thesubsubsection.}{4pt}{}

\allowdisplaybreaks

\newcommand\id{\mathbbm{1}}	
\newcommand{\tss}[1]{\textsuperscript{#1}}
\newcommand{\on}[1]{
	\operatorname{#1}
}

\def \ls#1#2 {^{#1}\!#2}

\newcommand{\tdt}{\times\cdots\times}

\newcommand{\tightoverset}[2]{
  \mathop{#2}\limits^{\vbox to -.5ex{\kern-1.15ex\hbox{$#1$}\vss}}}

\newcommand\blfootnote[1]{%
  \begingroup
  \renewcommand\thefootnote{}\footnote{#1}%
  \addtocounter{footnote}{-1}%
  \endgroup
}

\renewenvironment{thebibliography}[1]
{
  \begin{oldthebibliography}{#1}
    \setlength{\itemsep}{0em  plus 0.3ex}
    \setlength{\parskip}{0em}
}
{
  \end{oldthebibliography}
}

\newcommand\eplog[1]{\exp(-\log^{O(1)} (2#1^{-1}))}

\newcommand\ssk[1]{
	\substack{#1}
}

\newcommand\ex{\mathop{\mathbb{E}}}

\newcommand{\exx}{
  \mathop{
    \mathchoice{\vcenter{\hbox{\larger[4]$\mathbb{E}$}}}
               {\kern0pt\mathbb{E}}
               {\kern0pt\mathbb{E}}
               {\kern0pt\mathbb{E}}
  }\displaylimits
}

\makeatletter
\newcommand*\bcdot{\mathpalette\bigcdot@{0.5}}
\newcommand*\bigcdot@[2]{\mathbin{\vcenter{\hbox{\scalebox{#2}{$\m@th#1\bullet$}}}}}
\makeatother

\makeatletter
\def\blfootnote{\gdef\@thefnmark{}\@footnotetext}
\makeatother

\newcommand{\blc}{\bm{\mathsf{C}}}

\newcommand{\bsc}{\bm{\mathsf{c}}}

\newcommand{\upd}[1]{\overset{\bcdot}{#1}}

\begin{document}

\begin{center}\Large\noindent{\bfseries{\scshape A quasipolynomial inverse theorem for the $\mathsf{U}^k(\mathbb{F}_p^n)$ norm in the high characteristic}}\\[24pt]\normalsize\noindent{\scshape Luka Mili\'cevi\'c\tss{\dag}}
\end{center}
\blfootnote{\noindent\dag\ Mathematical Institute of the Serbian Academy of Sciences and Arts\\\phantom{\dag\ }Email: luka.milicevic@turing.mi.sanu.ac.rs\\
\phantom{\dag\ }2020 Mathematics Subject Classification. Primary 11P70; Secondary 05C25, 05E16.}

\footnotesize
\begin{changemargin}{1in}{1in}
\centerline{\sc{\textbf{Abstract}}}
\phantom{a}\hspace{12pt}~We prove an inverse theorem for the Gowers uniformity norm $\mathsf{U}^k(\mathbb{F}_p^n)$ with quasipolynomial bounds in the case when $p \geq k$. The inverse theorem follows from a quasipolynomial structure theorem for Freiman multihomomorphisms, which are a natural generalization of Freiman homomorphisms to maps of several variables.\\
\phantom{a}\hspace{12pt}~The proof of the structure theorem for Freiman multihomomorphisms is the central result of the paper and rests on three main ingredients:\ algebraic regularity method, abstract Balog-Szemer\'edi-Gowers theorem and the theory of multilinear maps defined on multilinear varieties. The last ingredient originates from an earlier work of Gowers and the author, and is significantly expanded in this paper. In particular, once the theory of such maps is in place, the proof of the structure theorem for Freiman multihomomorphisms is relatively short, especially compared to the previous quantitative results in the inverse theory of Gowers norms.
\end{changemargin}
\normalsize

\section{Introduction}

We begin by recalling the definition of the Gowers uniformity norms.

\begin{defin}
    Let $f : G \to \mathbb{C}$ be a function on a finite abelian group $G$. The \textit{discrete multiplicative derivative} with \textit{shift} $a$ is the operator that maps $f$ to the function $\partial_a f$, given by the formula $\partial_a f(x) = f(x + a)\overline{f(x)}$. With this notation, the uniformity norm $\|f\|_{\mathsf{U}^k(G)}$ is defined as 
    \[\Big(|G|^{-k - 1}\sum_{x, a_1, \dots, a_k \in G} \partial_{a_1} \dots \partial_{a_k} f(x)\Big)^{2^{-k}}.\]
\end{defin}

Although it is not obvious from the definition, $\|\cdot\|_{\mathsf{U}^k(G)}$ is a norm for $k \geq 2$.

In the present paper, we prove the quasipolynomial inverse theorem for uniformity norms $\|\cdot\|_{\mathsf{U}^k(G)}$ in the case when the ambient group $G$ is a finite-dimensional vector space over a prime field $\mathbb{F}_p$, in the case of the high characteristic, meaning that $p \geq k$. In the rest of the paper, we write $\|\cdot\|_{\mathsf{U}^k}$ for the norm without the group in the subscript.

\begin{theorem}\label{inverseUniformityMain} Let $p\geq k$ and let $f \colon \mathbb{F}_p^n \to \mathbb{D}$ be a function such that $\|f\|_{\mathsf{U}^k} \geq c > 0$ (where $\mathbb{D}$ is the unit disc in $\mathbb C$). Then there is a polynomial $g \colon \mathbb{F}_p^n \to \mathbb{F}_p$ of degree at most $k-1$ such that 
    \[\Big|\exx_{x \in \mathbb{F}_p^n} f(x) \exp\Big(\frac{2 \pi i}{p} g(x)\Big)\Big| \geq \exp(-\log^{O(1)}(2c^{-1})).\]
    \end{theorem}

The central result of the paper is the quasipolynomial structure theorem for Freiman multihomomorphisms, introduced in~\cite{FreimanMultihom}, whose definition we now recall. Recall first that for abelian groups $G$ and $H$ and a subset $A \subset G$, a map $\phi \colon A \to H$ is a \textit{Freiman homomorphism} if whenever $a_1, a_2, a_3, a_4 \in A$ satisfy $a_1 + a_2 = a_3 + a_4$, then $\phi(a_1) + \phi(a_2) = \phi(a_3) + \phi(a_4)$. When this equality holds for a particular additive quadruple $(a_1, a_2, a_3, a_4)$, we say that $\phi$ \textit{respects} $(a_1, a_2, a_3, a_4)$. Informally, Freiman multihomomorphisms are multivariate maps that respect additive quadruples in all principal directions.

\begin{defin} Let $G_1, \dots, G_k$ and $H$ be finite-dimensional vector spaces over $\mathbb{F}_p$, and let $A$ be a subset of $G_1 \tdt G_k$. A function $\phi\colon A \to H$ is a \emph{Freiman multihomomorphism} if for every $d\in\{1,2,\dots,k\}$ and every $(a_1,\dots,a_{d-1},a_{d+1},\dots,a_k)\in G_1\tdt G_{d-1}\times G_{d+1}\tdt G_k$, the map from $\{x_d \in G_d \colon (a_1,\dots,a_{d-1},x_d, a_{d+1}, \dots, a_k) \in A\}$ to $H$ defined by the formula $x_d\mapsto\phi(a_1,\dots,a_{d-1},x_d,a_{d+1},\dots,a_k)$ is a Freiman homomorphism.\end{defin}

When the domain is the full product $G_1 \tdt G_k$, the notion of  a Freiman multihomomorphism coincides with that of a multiaffine map, which are maps that are affine in each variable. We may now state the central result of this paper, saying that multiaffine maps are essentially the only source of Freiman multihomomorphisms.

\begin{theorem}[Quasipolynomial structure theorem for Freiman multihomomorphisms]\label{strFmult}
    Let $G_1, \dots, G_k, H$ be finite-dimensional vector spaces over $\mathbb{F}_p$. Let $A \subseteq G_1 \tdt G_k$ be a set of density $c > 0$ and let $\phi : A \to H$ be a Freiman multihomomorphism. Then there exists a global multiaffine map $\Phi : G_1\tdt G_k \to H$ such that $\phi(x_1, \dots, x_k) = \Phi(x_1, \dots, x_k)$ holds for at least $\eplog{c} |G_1 \tdt G_k|$ points $(x_1, \dots, x_k) \in A$.
\end{theorem}

\begin{rem*} Theorem~\ref{strFmult} was first proved by Gowers and the author in~\cite{FreimanMultihom}, with a significantly more complicated argument and worse bounds, involving iterated exponential of height greater than $k!$.\end{rem*}

From Theorem~\ref{strFmult} one may reasonably straightforwardly deduce the quasipolynomial inverse theorem for the uniformity norms, i.e. Theorem~\ref{inverseUniformityMain}. The proof is standard and the full details can be found in Subsection 12.3 of~\cite{FreimanMultihom}.

In the rest of introduction, we briefly discuss the history of the uniformity norms and their inverse theory, before moving on to a discussion of the proof of Theorem~\ref{strFmult}. The next subsection is similar to the introduction of~\cite{QPU4}, except for the novel results in the inverse theory that appeared in the meantime.

\subsection{History of the problem}

In his pioneering  work~\cite{GowU4, GowerskAP}, in which he proved an effective version of Szemer\'edi's theorem on arithmetic progressions, Gowers introduced the uniformity norms as means to quantifying algebraic structure in functions on an abelian group $G$. As observed by Gowers, if $A \subseteq G$ is a set of density $\delta = \frac{|A|}{|G|}$ and if the norm $\|\id_A - \delta\|_{\mathsf{U}^{k}}$ is small enough, then $A$ must contain essentially the same number of arithmetic progressions of length $k + 1$ as a randomly chosen subset of $G$ of the same density $\delta$. To prove Szemer\'edi's theorem from this observation, we must understand the structure of functions $f : G \to \mathbb{D} = \{z \in \mathbb{C} : |z| \leq 1\}$ with large uniformity norm. This is the \textit{inverse problem for uniformity norms}, and has been a key question of additive combinatorics giving rise to the field of higher order Fourier analysis.

In his work, Gowers obtained a partial answer to the inverse problem by showing a \text{local} inverse theorem. For any function $f : \mathbb{Z}/N\mathbb{Z} \to \mathbb{D}$ with $\|f\|_{\mathsf{U}^k} \geq c$, he showed that the group $\mathbb{Z}/N\mathbb{Z}$ can be partitioned into arithmetic progressions $P_1, \dots, P_m$ of lengths at least $N^{\delta_{k,c}}$, where $\delta_{k,c} \in (0,1)$ is a constant depending on $k$ and $c$ only, and that there exits polynomials $p_1, \dots, p_m$ of degree $k - 1$, such that $f$ locally correlates with phases of these polynomials, namely
\[\sum_{i \in [m]}\Big|\sum_{x \in P_i} f(x) \on{e}(p_i(x))\Big| \geq \Omega_c(N).\]
From the perspective of the inverse problem, this result is incomplete as not all such piece-wise phase polynomials have large uniformity norm.

As a solution to the inverse problem, we want that $f$ correlates \textit{globally} with a function with rich algebraic structure. In other words, we look for a function $g : G \to \mathbb{D}$, typically with a polynomial-like behaviour such that $|\ex_{x \in G} f(x) \overline{g(x)}| \geq c'$, where  $c'$ is a parameter that depends on $c = \|f\|_{\mathsf{U}^k}$, which we call the \textit{correlation bound} and $\ex_{x \in G}$ is the standard shorthand for the average $\frac{1}{|G|}\sum_{x \in G}$. Thus, solution to the inverse problem includes a definition of a family of functions from which we may pick $g$, whose members we call \textit{obstructions to uniformity}. Additionally, the definition of obstruction function family must be simple enough, so that the polynomial-like behaviour can be efficiently used. The obstruction functions and thus the answer to the inverse question typically depend on the ambient group. We now briefly mention some previous results.

\noindent\textbf{Previous results.} The starting point of the inverse theory is the case of the $\mathsf{U}^3(G)$ norm. Inverse theorems for this norm were obtained by Green and Tao~\cite{GreenTaoU3}, when $G$ is of odd order, and by Samorodnitsky~\cite{SamorU3}, when $G= \mathbb{F}_2^{n}$. Jamneshan and Tao~\cite{JamTao} obtained a proof for all finite abelian groups, giving a unified theory for the $\mathsf{U}^3(G)$ norm.

\indent For $k \geq 4$, the difficulty of the inverse problem for uniformity norms $\mathsf{U}^k$ increases significantly. The inverse theorem for $\mathsf{U}^k(\mathbb{F}_p^n)$ norm, in the high characteristic case, $k \leq p$, where the obstruction functions can be taken to be polynomial phases, is a remarkable result of Bergelson, Tao and Ziegler~\cite{BTZ, TaoZieglerCorr}, later extended by Tao and Ziegler~\cite{TaoZiegler} to include the low characteristic case ($k > p$). Low-characteristic case also requires a more general family of obstructions, known as the \textit{non-classical polynomials}. On the other hand, when the group is cyclic, Green, Tao and Ziegler~\cite{GTZU4, GTZ} proved the inverse theorem for the $\mathsf{U}^k(\mathbb{Z}/N\mathbb{Z})$ norm. This breakthrough is a key part of the Green--Tao programme for obtaining asymptotic estimates for the counts of linear configurations in primes~\cite{GTprimes1, GTprimes2}. For cyclic groups, the family of obstructions can be taken to be nilsequences.

\indent Another approach to the inverse theory for uniformity norms was developped by Szegedy~\cite{Szeg} and Camarena and Szegedy~\cite{CamSzeg}, known as the nilspace theory. Since the foundational papers, there have been many developmments in nilspace theory, in particular, Candela, Gonz\'alez-S\'anchez and Szegedy~\cite{nilspacesCharp} obtained a new proof the Tao-Ziegler inverse theorem, and made contributions to the inverse problem in cases of the groups of bounded torsion~\cite{nilspacesBoundedTorsion} and the general abelian groups~\cite{nilspacesGeneralAbelian}.

\noindent\textbf{Quantitative inverse theorems.} Given the scope of their applications, it is highly desireable to obtain inverse theorems with good bounds on the corelation with the obstruction functions. The results  mentioned so far in the case $k \geq 4$ only provided qualitative bounds. For example, the proof of Bergelson, Tao and Ziegler, in which the ambient group is a finite vector space, relied on ergodic theory and the nilspace theory is infinitary in its nature. In fact, the question of bounds is one of the key questions of additive combinatorics (see the surveys of Wolf~\cite{WolfSurvey} and of Peluse~\cite{PeluseSurvey}).

\indent The situation with the inverse theorems for $\mathsf{U}^3(G)$ norm is again significantly better. The first results of Green and Tao, and of Samorodnitsky, already gave bounds involving a single exponential (more precisely, $c' \geq \exp(-c^{-O(1)})$). Sanders~\cite{Sanders}, relying on almost-periodicity result of Croot and Sisask~\cite{CrootSisaskPaper}, obtained a quasipolynomial bound, namely of the shape $c' \geq \exp(-\log^{O(1)}(2c^{-1}))$. Finally, Gowers, Green, Manners and Tao~\cite{Marton1, Marton2} proved the polynomial bounds ($c' \geq c^{O(1)}$) in the case of the finite vector spaces.

\indent When it comes to higher norms, when $G = \mathbb{F}_p^n$, quantitative bounds were first obtained by Gowers and the author~\cite{U4paper} for the norm $\mathsf{U}^4$ for $p \geq 5$. Those bounds were roughly doubly exponential (more precisely, $c' \geq \exp^{(3)}(\log^{O(1)}(2c^{-1}))$). Kim, Li and Tidor~\cite{KimLiTidor}, and independently Lovett, improved the bounds by a single exponential by improving a step of the argument which relied on an inefficent use of Inclusion-Exclusion principle. A quantitative version of the inverse theorem for $\mathsf{U}^k$ norms in the high characteristic was proved by Gowers and the author~\cite{FreimanMultihom}, with a bound involving a bounded number of exponentials, depending on $k$ only. There is also progress in the low characteristic~\cite{LukaU56, Tidor}.

\indent Finally, the author found a new proof~\cite{QPU4} of the inverse theorem for the $\mathsf{U}^4$ norm with a quasipolynomial corelation bound. That proof is the starting point for the present paper, and will be discussed in detail later.

\indent When the ambient group is cyclic, Manners proved doubly exponential bounds in the inverse theorem~\cite{MannersUk} for all $k$, giving the first quantitative proof in that setting. Furthermore,  quasipolynomial bounds in the cyclic groups case were obtained by Leng~\cite{LengNil2} for the $\mathsf{U}^4$ norm and by Leng, Sah and Sawhney~\cite{LengSahSawhney} for the higher norms, building on the work of Green, Tao and Ziegler, and relying on Leng's improved equidistribution theory for nilsequences~\cite{LengNil1}.

\indent Furthermore, it is worth noting that, more recently, some works went beyond the cases of cyclic groups and finite vector spaces. Using ergodic theory, Jamneshan, Shalom and Tao~\cite{JamShaTao} remarkably proved an inverse theorem for all bounded torsion groups and the author developed~\cite{LukaGenU4} a quantitative general inverse theory for the $\|\cdot\|_{\mathsf{U}^4}$ norm.

\subsection{New $\mathsf{U}^4(\mathbb{F}_p^n)$ proof}

As mentioned above, in~\cite{QPU4}, the author obtained quasipolynomial bounds for the inverse theorem for the $\mathsf{U}^4(\mathbb{F}_p^n)$ norm. Let us briefly outline the differences between that proof and the other quantitative proofs; more details can be found in the introduction of~\cite{QPU4} which contains a detailed comparison with previous results.

Several proofs of quantitative inverse theorems proceed by studying approximate variants of polynomials~\cite{U4paper, FreimanMultihom, MannersUk}; see also a very recent work of Peluse~\cite{PeluseU4} following a similar strategy to that in~\cite{U4paper}. Broadly speaking, the strategy in those works is to begin with a map $\phi$ that satisfying $1\%$ of certain cocycle identities in the given group and to first strengthen its structure to a map that satisfies $99\%$ of related cocycle identities, which is then related to another map satisfying all identities in its domain via a different argument. In the finite vector spaces, the first phase requires approximation theorems for generalizations of convolutions, which have an exponential cost that cannot be avoided. Additionally, the pass from a function that respects $99\%$ of cocycles to a genuine polynomial is surprisingly difficult in the higher order case. In the case of cyclic groups, the work of Manners depends on the low rank and no small subgroups assumptions.

On the other hand, the strategy of Green, Tao and Ziegler~\cite{GTZ} begins by applying inductive hypothesis and obtaining an approximately linear system of obstructions in the following sense. If $\|f\|_{\mathsf{U}^k} \geq c$ holds for a function $f : G= \mathbb{Z}/N\mathbb{Z} \to \mathbb{D}$, then there are degree-$(k-2)$ nilsequences $\phi_a$, indexed by $\Omega_c(N)$ elements $a \in G$, such that
\begin{equation}\label{gtzcondition}\text{for at least }\Omega_c(N^3)\text{ additive quadruples }a_1 + a_4 = a_2 + a_3,\,\overline{\phi_{a_1}}T_{a_1 - a_4} \phi_{a_2} \phi_{a_3} \overline{T_{a_1 - a_4}\phi_{a_4}}\text{ is biased,}\end{equation}
where $T$ is the translation operator. In their case, the low rank structure of the cyclic group is essential. It implies that nilsequnces have bounded dimension, which gives a finite list of parameters associated to a nilsequence, playing the role of its coefficients. Ultimately, after important regularization step called the Sunflower Lemma, condition~\eqref{gtzcondition} is used to obtain a Freiman homomorphism.

\indent Similar application of inductive hypothesis is used in~\cite{QPU4} and an analogue of~\eqref{gtzcondition} becomes having a collection of linear maps $\phi_a : G \to G$ such that 
\begin{equation}\label{matrixcondition}\text{for at least }\Omega_c(|G|^3)\text{ additive quadruples }a_1 + a_4 = a_2 + a_3,\,\,\phi_{a_1} - \phi_{a_2}  - \phi_{a_3} + \phi_{a_4}\text{ has rank }O_c(1).\end{equation}

We note that a similar system of linear maps was studied by Kazhdan and Ziegler in~\cite{KazhZiegApproximateCohom}. However, they used the inverse theorem for the $\mathsf{U}^4$ norm to relate the system to a bilinear map from $G \times G$ to $G$ and they noted that a different argument could lead to a new proof of the inverse theorem.

\indent Since the group $G = \mathbb{F}_p^n$ has unbounded rank and thus linear maps have unbounded number of coefficients describing them, the question of understanding such a system of linear maps needs to be treated `abstractly' in the sense that we may only use abstract properties of linear maps, rather then their internal structure. In fact, owing to this abstract feature of the proof, it was possible to give a general $\mathsf{U}^4$ inverse theory in~\cite{QPU4genab}.

The distinction in behaviour between conditions~\eqref{gtzcondition} and~\eqref{matrixcondition}, and thus between cyclic groups and finite vector spaces, can also be explained using entropy of inverse theorems, defined by Tao in a blog post~\cite{TaoBlogEntropy}. In particular, the inverse theorems in the finite vector spaces are in entropic sense weaker than those in cyclic groups, implying that the inductive step requires additional work. Furthermore, in her survey~\cite{WolfSurvey}, Wolf notes that the inverse theory for $\|\|_{\mathsf{U}^k}$ norms for $k \geq 4$, has additional challenges in the finite vector space case, not present in the cyclic group case. For example, Manners~\cite{MannersUk} uses the specific structure of groups to avoid genuine “cohomological” obstructions. Such obstructions are one of the main sources of difficulty in the present work.

\indent Throughout the paper, by a system of maps, we mean a collection of maps $(\phi_a)_{a \in A}$, indexed by a dense subset $A \subseteq G$ for a finite vector space $G$. A recurrent theme in~\cite{QPU4} and this work are various notions of \emph{respectedness} for a system of $H$-valued maps $(\phi_a)_{a \in A}$. In~\cite{QPU4}, we have two primary such notions. Firstly, we say that an additive quadruple $a_1 + a_4 = a_2 + a_3$ is \textit{$r$-respected} if $\phi_{a_1} - \phi_{a_2}  - \phi_{a_3} + \phi_{a_4}$ has rank at most $r$. Thus, the starting assumption~\eqref{matrixcondition} says that a dense collection of additive quadruples is $r$-respected.

\indent The way the condition~\eqref{matrixcondition} is utilized in~\cite{QPU4} is by changing the category of maps and the notion of respectedness. Namely, instead of considering linear maps defined on the whole space $G$, we move to considering partially defined linear maps, meaning that we have $\phi_a : U_a \to H$, each equipped with its own domain $U_a$ of low codimension in $G$. While this change is detrimental, as the maps are no longer defined on $G$ and the situation is seemingly more complicated, we are able to obtain a stronger notion of respectedness. The second primary notion of respectedness is \emph{subspace-respectedness}, meaning that an additive quadruple has the linear combination $\phi_{a_1} - \phi_{a_2}  - \phi_{a_3} + \phi_{a_4}$ vanishing at all points where it is defined, which is the subspace $U_{a_1} \cap U_{a_2} \cap U_{a_3} \cap U_{a_4}$. Thus, the proof can be thought of as having four main phases:
\begin{itemize}
\item[\textbf{Phase 1.}] Starting with a system of linear maps $(\phi_a : G \to H)_{a \in A}$ for a dense set $A \subseteq G$ such that a dense collection of additive quadruples are $r$-respected, we relate it to another system of linear maps $(\phi'_a : G \to H)_{a \in U}$, defined on the whole of a subspace $U \leq G$ of low-codimension, in which all additive quadruples are $r'$-respected, for a reasonably small $r'$. The key ingredient introduced here is the abstract Balog-Szemer\'edi-Gowers theorem, which we shall describe in the next subsection.
\item[\textbf{Phase 2.}] We make the crucial change of category of maps and the notion of respectedness, passing to a system of partially defined linear maps $(\psi_a : V_a \to H)_{a \in U}$, in which almost all additive quadruples are subspace-respected.
\item[\textbf{Phase 3.}] We develop a variant of the bilinear Bogolyubov argument, to ensure that the subspaces $V_a$ depend linearly on $a$, i.e. $V_a = \langle \Phi_1(a), \dots, \Phi_s(a) \rangle^\perp$, for some linear maps $\Phi_1, \dots, \Phi_s : U \to G$. The price we pay is that the system again has maps defined only for a dense collection of indexing points in $U$, with only a dense collection of additive quadruples subspace-respected.
\item[\textbf{Phase 4.}] Finally, like in the first phase, we obtain another system of linear maps $(\psi'_a : V_a \to H)_{a \in U'}$, defined on the whole of a subspace $U' \leq G$ of low-codimension, in which all additive quadruples are subspace-respected. However, despite the superficial similarity and another use of the abstract Balog-Szemer\'edi-Gowers theorem, the details in this phase are significantly different than in the first phase, even when it comes to the application of the abstract Balog-Szemer\'edi-Gowers theorem, and require much more care and the use of algebraic regularity method, that we shall also describe in the next subsection.
\end{itemize}
\indent In the proof of Theorem~\ref{strFmult}, we follow similar broad strategy, however, the details are significantly different as we study systems of multilinear maps instead of systems of linear maps, which are much harder objects to understand. In particular, the algebraic regularity method is much subtler. In the next subsection, we discuss some of the main tools used in the proof, before giving the step-by-step outline, which appears in Section~\ref{proofoutlinesection}.

\subsection{Main tools} 

\noindent\textbf{Algebraic regularity method.} The algebraic regularity method originated in~\cite{U4paper}, where the observation that bilinear maps of high rank have good quasirandomness properties was used in the graph-theoretic fashion. More precisely, given a bilinear map $\beta : G \times G \to \mathbb{F}_p^r$ we may consider the bipartite graph whose vertex classes are copies of $G$ and $xy$ is an edge if $\beta(x,y) = 0$. Applying Szemer\'edi's regularity lemma directly would partition the vertices of this graph into subsets, most of whose pairs behave quasirandomly. However, the bounds stemming from that regularity lemma would eventually lead to bad bounds in the inverse theorem for $\mathsf{U}^4$ norm. Fortunately, the algebraic setting allows a stronger and more efficient regularity lemma, which we now state. 

\begin{theorem}[Algebraic regularity lemma, Corollary 5.2 in~\cite{U4paper} and Theorem 10 in~\cite{QPU4}] \label{arl}
    Suppose that $\beta \colon G \times G \to \mathbb{F}_p^r$ is a bilinear map and let $s > 0$. Then there exists a bilinear map $\gamma \colon G \times G \to \mathbb{F}_p^{r'}$, where $r' \leq r$, and a subspace $V$ of codimension at most $2rs$ such that 
    \begin{itemize}
        \item[\textbf{(i)}] for each $\lambda \not= 0$, the rank of $\lambda \cdot \gamma$ on $V \times V$ is at least $s$ (in this case we say that $\gamma$ has rank at least $s$),
        \item[\textbf{(ii)}] for all cosets $a + V$ and $b + V$ such that $\{\beta = 0\} \cap (a + V) \times (b + V) \not=\emptyset$ we have
        \begin{equation}\{\beta = 0\} \cap (a + V) \times (b + V) = \{\gamma = 0\} \cap (a + V) \times (b + V).\label{allcosetseq}\end{equation}
    \end{itemize}
\end{theorem}
Apart from good bounds, this lemma is stronger than Szemer\'edi's regularity lemma in two additional ways: the subsets in the partition of vertex classes are cosets of the same subspace, and all pairs of subsets induce quasirandom bipartite graphs.

\indent However, when trying to extend this approach directly to the higher order case, translating hypergraph regularity theory into the algebraic setting quickly runs into problems and leads to a similar growth of bounds in Grzegorczyk hierarchy. The only difference in the bounds between those two contexts is that in algebraic setting the bounds reside one level lower in the hierarchy, so they are only good in the bipartite case. A key realization in~\cite{FreimanMultihom} is that effective equidistribution theorems for the obstruction functions can be used in place of the higher order regularity lemmas and avoid the terrible bounds that would otherwise arise. The main equidistribution theorem in our case is a solution to the partition vs. analytic rank problem.

\begin{theorem}[Partition vs. analytic rank problem]\label{arankprankIntro}
    Let $\phi: G_1 \tdt G_k \to \mathbb{F}_p$ be a multilinear form. Suppose that $\on{bias} \phi \geq c$. Then $\on{prank} \phi \leq (\log 2c^{-1})^{O(1)}$.
\end{theorem}

The necessary definitions can be found in Subsection~\ref{multalgsubsec}. The above result was initially proved for polynomials rather than multilinear forms in a work of Green and Tao~\cite{GreenTaoPolys}. Their approach was refined by Kaufman and Lovett~\cite{KaufmanLovett} and finally Bhowmick and Lovett~\cite{BhowLov} adapted it to the multilinear setting. In those works, the bounds on the partition rank in terms of $c$ were Ackermannian-type. First polynomial bounds for this problem were obtained by Janzer~\cite{Janzer2} and by the author~\cite{LukaRank}, answering a question of Kazhdan and Ziegler~\cite{KazhZiegQn}. Moshkovitz and Zhu~\cite{MoshZhuRank} obtained almost linear bounds and there are even linear bounds in some closely related results~\cite{LukaLowCodim, MobiusOptimal}.

\indent Essentially, Theorem~\ref{arankprankIntro} says that if a multilinear form is not quasirandom, which means that it has an uneven distribution of values, then the form has a very strong algebraic structure, essentially it is built from a very small number of lower order multilinear forms. In~\cite{FreimanMultihom}, this theorem was used in several different ways, in particular as a way of partitioning a given variety into regular subvarieties. The present paper contains a significantly more direct proof of Theorem~\ref{strFmult}, and Theorem~\ref{arankprankIntro} is applied directly frequently, but the geometric perspective on the multilinear varieties and their quasirandom properties is key throughout the paper.

\noindent\textbf{MM-$r$-maps theory.} Recall that we begin with a system of multilinar maps from $G_1 \tdt G_k$ to $H$ and that an important step in the proof of the $\|\cdot\|_{\mathsf{U}^4}$ inverse theorem is passing from global linear maps to partially defined linear maps. Natural analogue in the higher dimensional case are multilinear maps defined on multilinear varieties of bounded codimension of the domain variety. We refer to these as the \textit{MM-$r$-maps}, where $r$ is the bound on the codimension. The difficulty in studying such maps stems from the fact that MM-$r$-maps are not merely restrictions of global multilinear maps. In full generality, such maps where first studied by Gowers and the author in~\cite{MultilinearExtensions}, where it was shown that, despite the mentioned difficulty, they still coincide with global multilinear maps, after possibly passing to a further small-codimensional variety. Special case of bilinear maps on bilinear varieties was treated in~\cite{U4paper}, and the problems of similar flavour were considered for the case of high-rank varieties by Kazhdan and Ziegler in~\cite{KazhZiegExtnQuad, KazhZiegExtnQuadHigh, KazhZiegQn}. However, finding a high-rank subvariety inside the given variety is inefficient and leads to significantly worse bounds on the codimension than those in~\cite{MultilinearExtensions}.

\indent The proof in~\cite{MultilinearExtensions} relies on the algebraic regularity method and a significant part of this paper is devoted to developing that theory further. For example, the multilinear Bogolyubov argument in the third phase of the proof of Theorem~\ref{strFmult} requires us to understand when two MM-$r$-maps $\phi_1 :U_1 \to H$ and $\phi_2 : U_2 \to H$, agreeing on the intersection of domains, have a common multilinear extension. For linear maps, such a question is an easy exercise in linear algebra. In the multilinear setting, such a question is surprisingly deep.

\indent Furthermore the fourth phase of the proof requires us to consider, for a multilinear variety $V \subseteq G_0 \times G_1 \tdt G_k$, systems of multilinear maps whose domains are slices $V_{x_0} = \{(y_1, \dots, y_k) : (x_0, y_1, \dots, y_k) \in V\}$. In this paper, we show that we may think of multilinear varieties as \textit{weak sheaves}, in the sense that, after intersection with a lower order variety, we have a form of locality and gluing property.

\indent Finally, we need to understand MM-$r$-maps which map most of their domain to 0. In the linear setting, it is an easy exercise to show that linear maps with $99\%$ of zeros are actually zero maps. That is no longer the case in the multilinear setting, and we employ the extension theory developed in~\cite{MultilinearExtensions} to get a useful structural result on such MM-$r$-maps. 

\noindent\textbf{Abstract Balog-Szemer\'edi-Gowers theorem.} As it is well-known the Balog-Szemer\'edi-Gowers theorem was proved with polynomial bounds by Gowers in~\cite{GowU4}. The place where Balog-Szemer\'edi-Gowers theorem is used in that paper is the proof of the structure theorem for approximate homomorphisms. Concretely, the principal role is to move from a function $f$ that respects a dense collection of additive quadruples in the sense that $f(a_1) - f(a_2) - f(a_3) + f(a_4) = 0$, to a function that respects all additive quadruples in its domain. However, for the purposes of studying systems with property~\eqref{matrixcondition}, that theorem is insufficient. In~\cite{QPU4}, an abstract version was developed, which was significantly more general and flexible, appearing as Theorem~\ref{absg} in this paper. Its flexibility allows us to control much less robust notions of respectedness that we are forced to consider.

\noindent\textbf{Paper organization.} In the next section, we give a detailed proof overview. In Section~\ref{prelimSection}, we gather useful auxiliary results. The theory of MM-$r$-maps is developed in Section~\ref{extensionSection}. After that, there are four sections, each corresponding to a phase of the proof. Finally, the last section is devoted to putting all ingredients together and completing the proof of Theorem~\ref{strFmult}. 

\noindent\textbf{Acknowledgements.} This research was supported by the Ministry of Science, Technological Development and Innovation of the Republic of Serbia through the Mathematical Institute of the Serbian Academy of Sciences and Arts, and by the Science Fund of the Republic of Serbia, Grant No.\ 11143, \textit{Approximate Algebraic Structures of Higher Order: Theory, Quantitative Aspects and Applications} - A-PLUS.

I am grateful to Rudi Mrazovi\'c for comments on a draft of the paper that helped improve the presentation.

\noindent\textbf{Declaration of AI use.} No AI was used in the idea generation, nor document creation.

\section{Proof overview}\label{proofoutlinesection}

We prove Theorem~\ref{strFmult} by induction on $k$. Hence, we consider a Freiman multihomomorphism $\varphi : A \to H$, where $A \subseteq G_0 \times G_1 \tdt G_k$, with $G_0$ having a special role, and assume the structure result for $k$ variables. We write $G_{[k]}$ as a shorthand for $G_1 \tdt G_k$.

As indicated in the introduction, the recurring theme in the paper is that we study systems of maps $(\phi_x)_{x \in X}$, indexed by a subset $X$ of $G_0$, that have various properties that serve as approximate linearity, which all have the form of respecting additive quadruples, namely, when $a_1 - a_2 + a_3 - a_4 = 0$, we have
\[\phi_{a_1} - \phi_{a_2} + \phi_{a_3} - \phi_{a_4}\]
vanishing entirely, or on a large set, depending on the context. The reason that we say that this is a recurring theme is that almost every aspect will change over time.
\begin{itemize}
    \item The category of maps will change. We begin with multilinear maps defined on $G_{[k]}$, then consider multilinear maps $\phi_x : V_x \to H$ whose domains $V_x$ are multilinear varieties, but at first $V_x$ do not have an additional special structure in terms of $x$, and finally the varieties $V_x$ themselves depend linearly on $x$.
    \item Respectedness condition will change. Firstly it will mean that $\phi_{a_1} - \phi_{a_2} + \phi_{a_3} - \phi_{a_4}$ has many zeros, then that $\phi_{a_1} - \phi_{a_2} + \phi_{a_3} - \phi_{a_4}$ is closely related to one of the members of a short list of multilinear maps, then that $\phi_{a_1} - \phi_{a_2} + \phi_{a_3} - \phi_{a_4}$ vanishes on a suitable variety.
    \item Sometimes the indexing set will be merely dense, and sometimes the full vector space $G_0$.
    \item Sometimes all additive quadruples in a set will be respected and sometimes only a dense collection.
\end{itemize}

The main changes in the notion of respectedness and the types of arguments in the proof will be marked as phases of the proof. Each phase might have a few steps.

\indent Write $Z(\phi)$ for the set of zeros of a $H$-valued map $\phi$. For an additive quadruple $(x_1, x_2, x_3, x_4)$ and multilinear maps $\phi_{x_1}, \dots, \phi_{x_4} : G_{[k]} \to H$, we say that $(x_1, x_2, x_3, x_4)$ is \textit{$c$-respected} if $|Z(\phi_{x_1} - \phi_{x_2} + \phi_{x_3} - \phi_{x_4})| \geq c|G_{[k]}|$. If $\phi_{x_i} : V_{x_i} \to H$ are multilinear maps on multilinear varieties, we say that the quadruple is \textit{variety-respected} if $\phi_{x_1} - \phi_{x_2} + \phi_{x_3} - \phi_{x_4} = 0$ on $V_{x_1} \cap \dots\cap V_{x_4}$.

\noindent\textbf{Preliminary step.} \textit{Obtaining a system of multilinear maps.} By using the inductive hypothesis, we get a dense collection of $x \in G_0$, for which $\varphi(x, \bcdot)$ coincides with a multiaffine map $\phi_x : G_{[k]} \to H$ on a dense set. We may use directional convolutions to ensure that $\phi_x$ are multilinear maps and use the assumption on Freiman homomorphism in direction $G_0$ to conclude that for many additive quadruples $(x_1, x_2, x_3, x_4)$ the linear combination $\phi_{x_1} - \phi_{x_2} + \phi_{x_3} - \phi_{x_4}$ vanishes frequently.

\noindent\textbf{Phase 1.} The first phase is \textit{completion} of a system of global multilinear maps, in the sense that, starting with a system indexed by a dense set, in which a dense collection of additive quadruples is $c$-respected, we obtain a system indexed by \textit{all points in a subspace}, in which \textit{all additive quadruples} are $O_c(1)$-respected.

\begin{itemize}
    \item[\textbf{Step 1.1.}] \textit{Abstract BSG step 1.} We define sets of additive quadruples $\mathcal{Q}_i$ which satisfy the zero density of $\phi_{x_1} - \phi_{x_2} + \phi_{x_3} - \phi_{x_4}$ being at least $c^i$ for some $c$. By positive correlation property of multilinear varieties, these collections of additive quadruples satisfy conditions of the abstract Balog-Szemer\'edi-Gowers theorem. Hence, we end up with a dense set $X$ and a reasonably short list of multilinear maps $\psi_1, \dots, \psi_m : G_{[k]} \to H$ such that every additive quadruple in $X$ satisfies $\phi_{x_1} - \phi_{x_2} + \phi_{x_3} - \phi_{x_4} = \psi_i$ on a dense set. In other words, all additive quadruples are $c^{O(1)}$-respected up to an \textit{error function} $\psi_i$.
    
    \item[\textbf{Step 1.2.}]\textit{Ensuring all additive quadruples are zeroes-respected.} We need to remove those $\psi_i$ with very few zeroes. This step is similar in spirit to analogous one in the proof of the $\mathsf{U}^4$ inverse theorem, based on the algebraic dependent random choice. Eventually, we choose products of bounded dimension subspaces $P_1, \dots, P_\ell \subseteq G_1 \tdt G_k$, projections $\pi_1, \dots, \pi_\ell : H \to \mathbb{F}_p^s$ and vectors $\mu_1, \dots, \mu_\ell \in \mathbb{F}_p^s$ suitably at random and define $X' = \{x \in X : (\forall i \in [\ell]) \pi_j \circ \phi_x |_{P_i} = \mu_i\}$ as the desired subset of indexing elements. The key point here is that if $P_1, \dots, P_\ell$ and $\pi_1, \dots, \pi_\ell$ are chosen in the right way, then all additive quadruples in $X'$ are zeroes-respected for any choice of $\mu_1, \dots, \mu_\ell$, which we may then simply choose to maximize the size of $X'$.

 \item[\textbf{Step 1.3.}]\textit{Obtaining a full system of linear maps.} This step is an application of the robust Bogolyubov-Ruzsa theorem and relies on the fact that the property of being $c$-respected behaves well with respect to addition. As a result, we have $x$ ranging over a whole subspace, and all additive quadruples being $O_c(1)$-respected, completing the linear system.
\end{itemize}

\noindent\textbf{Phase 2.} In this phase, our main goal is to get a stronger notion of respectedness. To achieve this, we make a \textit{change of category}. 

\textbf{Step 2.1.} \textit{Changing the category.} We pass to multilinear maps defined on multilinear varieties of bounded codimension instead of the globally defined multilinear maps, with the notion of respectedness strengthened, so that of most of additive quadruples now are variety-respected. This step is subtler than its $\mathsf{U}^4$ counterpart, where we used a simple random choice argument, and the fact that a linear map with $1-o(1)$ proportion of zeros is necessarily a zero map. Unfortunately, such a fact on zero density does not hold for multilinear maps on multilinear varieties. Instead, we need to carefully keep track of low partition rank decompositions of linear combinations $\sum_{i \in [\ell]} (-1)^i\phi_{x_i}$, for bounded $\ell$, and rely on the algebraic regularity method.

\noindent\textbf{Phase 3.} In this phase, our main goal is to obtain a \textit{linear system of varieties}, i.e. to ensure that domains $V_a$ depend linearly on $a$. In other words, we want to ensure that $V_a = \{x_{[k]} \in G_{[k]} : (a, x_{[k]}) \in W\}$ for a multilinear variety $W \subseteq G_{[0,k]}$ of bounded codimension.

\textbf{Step 3.1.} \textit{Multilinear Bogolyubov argument.} In $\mathsf{U}^4$ proof, we employed a variant of the bilinear Bogolyubov argument~\cite{bilbog1, bilbog2, HosseiniLovett}, based on the following observation. If $\phi_{x+a} : U_{x+a} \to H,$ $\phi_{x} : U_x \to H$, $\phi_{x} : U_x \to H$ and $\phi_{x} : U_x \to H$ are linear maps forming a subspace-respected additive quadruple, then $\phi_{x+a} - \phi_a$ and $\phi_{y+a} - \phi_y$ are linear maps defined on $U_{x+a} \cap U_x$ and $U_{y+a} \cap U_y$ respectively, agreeing on the intersection of their domains $U_{x+a} \cap U_x \cap U_{y+a} \cap U_y$. Elementary linear algebra gives us the existence of a common extension $\psi$, defined on the subspace $(U_{x+a} \cap U_x) + (U_{y+a} \cap U_y)$ and bilinear Bogolyubov argument allows us to find subspaces $V_a$ depending linearly on $a$ in such sums. These particular sums of subspaces were studied by Hosseini and Lovett~\cite{HosseiniLovett} and their quasipolynomial argument is important for this step.

\indent In the higher order case, due to the complicated nature of MM-$r$-maps, finding such extensions is not always possible, and is a much more difficult task, which requires development of the simultaneous extension theory in Section~\ref{extensionSection}.

\noindent\textbf{Phase 4.} In the fourth phase we obtain \textit{completion} of a system of MM-$r$-maps defined on a linear system of varieties.

\begin{itemize}
    \item[\textbf{Step 4.1.}] \textit{Abstract BSG step 2.} In the system of maps $(\phi_x)_{x \in X}$, the map $\phi_x$ now has domain $U_x \cap C$, for a linearly varying variety $U_x$ and a common variety $C$, meaning that it is independent of $x$. Allowing variation in $C$ is crucial to obtain the weak-transitivity for the abstract Balog-Szemer\'edi-Gowers theorem. We obtain a bounded codimension variety $W$ and a reasonably short list of multilinear maps $\psi_1, \dots, \psi_m : W \to H$, playing the role of error functions, that coincide on the intersection of domains with linear combinations of $\phi_x$. It is important that all these maps are defined on the same variety $W$.

\item[\textbf{Step 4.2.}] \textit{Ensuring that all quadruples in a subspace are variety-respected up to error functions.} Compared to first phase, we now make a twist, and apply robust Bogolyubov-Ruzsa theorem first to ensure that we have a full subspace as an indexing set, with all additive quadruples variety-respected up to error function in a finite list. This is a departure from the $\mathsf{U}^4$ proof, which is ultimately required due to difficulties in working with MM-$r$-maps, which are not present in the case of partially-defined linear maps.

\item[\textbf{Step 4.3.}] \textit{Getting a dense set in which all additive quadruples are variety-respected.} Like in the \textbf{Step 1.2}, we need to remove the error maps in the list originating in the \textbf{Step 4.1}. In order to apply an algebraic dependent random choice as in the \textbf{Step 1.2}, we must first remove those maps $\psi_i : W \to H$ that vanish on $1-c_k$ proportion of points in the variety $W$, for some constant $c_k$ depending on $k$ only. In the case of $\mathsf{U}^4$ theorem, we had linear maps, and they vanish automatically on the whole subspace $W$ in such a case. In the case of MM-$r$-maps, such a condition is very hard to work with and we need to use the full power of the extension theory of MM-$r$-maps to deduce the following key dichotomy: we may find a variety $W' \subseteq W$ of small codimension and a short list of points $x^{(1)}_{[k]}, \dots ,x^{(\ell)}_{[k]}$ such that, every $\psi_i$ with $(1-c_k)$-dense zero set either vanishes on the whole $W'$ or $\psi_i(x^{(j)}_{[k]}) \not=0$ for some $j \in [\ell]$. Once we obtain the dichotomy, the algebraic dependent random choice argument can once again be used.

\item[\textbf{Step 4.4.}] \textit{Getting a full subspace in which all additive quadruples are variety-respected.} This is another application of the robust Bogolyubov-Ruzsa theorem; in fact the argument of the \textbf{Step 4.2} can be applied here directly.
\end{itemize}

\noindent\textbf{Final step.} \textit{Completing the proof.} Once we get a multilinear map on a multilinear variety, using the known extension theory, we may pass to a global multilinear map and conclude the proof.

\section{Preliminaries}\label{prelimSection}

Throughout the paper, prime $p$ is fixed, $G_0, \dots, G_k$ and $H$ are finite-dimensional vector spaces over the prime field $\mathbb{F}_p$. We choose non-trivial additive character $\omega : \mathbb{F}_p \to \mathbb{C}$ given by $\omega(\lambda) = \exp(2 \pi i \lambda/p)$. As is standard, we use averaging notation $\ex_{x \in X}$ as a shorthand for $\frac{1}{|X|} \sum_{x \in X}$ for any finite set $X$. When $X$ is understood from the context, we write $\ex_x$.

We treat $p$ and $k$ as fixed constant and suppress them from the asymptotic notation.

\noindent\textbf{Sequence notation.} Given a sequence of elements $(a_i)_{i \in I}$, we frequently abbreviate it to $a_I$. We also write $G_I$ as a shorthand for the product $\prod_{i \in I} G_i$. Thus, points in the product $G_I$ will typically be denoted by $x_I$. In particular, we have $G_{[k]} = G_1 \tdt G_k$ and its elements are typically $x_{[k]}$.

\noindent\textbf{Main definitions.} For a subset $S \subseteq G_I$, we say that a function $\phi : S \to H$ is \textit{multilinear} if equality $\phi(x_I) + \phi(y_I) = \phi(z_I)$ holds for all triples of points $x_I, y_I, z_I$ in the set $S$ which differ in a single coordinate $i \in I$ and satisfy $x_i + y_i = z_i$. When $H = \mathbb{F}_p$, we say that $\phi$ is a \textit{multilinear form}. Frequently, $S = G_I$, in which case we sometimes say that $\phi$ is a \textit{global} multilinear map/form in order to stress the particular structure of the domain, though this will be clear from context.

\indent We say that a map $\beta : G_{[k]} \to \mathbb{F}_p^r$ is \textit{mixed-linear} if each of its component $\beta_i$ is multilinear on a subset of coordinate $I_i$, i.e. $\beta_i : G_{I_i} \to \mathbb{F}_p$ is a global multilinear form. If $\mathcal{P}$ is a collection of subsets of $[k]$, we say that $\beta$ is \textit{$\mathcal{P}$-dependent} if $I_i \in \mathcal{P}$ for all $i \in [r]$. A \textit{multilinear variety of codimension at most $r$} is a set $V$ of the shape $V = \{\beta = 0\}$ for a mixed-linear map $\beta : G_{[k]} \to \mathbb{F}_p^r$, where $\{\beta = 0\}$ is a shorthand for $\{x_{[k]} \in G_{[k]} : \beta(x_{[k]}) = 0\}$. Sometimes, instead of $\{\phi = 0\}$, we write $Z(\phi)$ and call it the \textit{vanishing set} of $\phi$.

\indent Finally, a central role in this paper will be played by multilinear maps $\phi : V \to H$ when $V$ is a multilinear variety of codimension at most $r$. In that case, we say that $\phi$ is a \textit{MM-$r$-map}. Similarly to above, we say that $\phi$ is \textit{MM-$r$-form} if the codomain is $\mathbb{F}_p$.

For a finite-dimensional vector space $G$ over $\mathbb{F}_p$, a dot product $\cdot : G \times G \to \mathbb{F}_p$ is a symmetric, non-degenerate bilinear form. For the canonical vector space $\mathbb{F}_p^r$, for vectors $\lambda, \mu \in \mathbb{F}_p^r$ we write $\lambda \cdot \mu = \sum_{i \in [r]} \lambda_i\mu_i$. We also misuse the notation and for a vector $\lambda \in \mathbb{F}_p$ and a sequence of elements $a_1, \dots, a_r \in G$, we write $\lambda \cdot a = \sum_{i \in [r]} \lambda_i a_i$, which is an element of $G$.

We recall a few important measures of structure of global multilinear forms. The \textit{bias} of a multilinear form $\alpha : G_{[k]} \to \mathbb{F}_p$ is defined by
\[\on{bias}\alpha = \exx_{x_{[k]} \in G_{[k]}} \omega(\alpha(x_{[k]})).\]
\textit{Partition rank} of $\alpha$ is the least integer $r$ such that $\alpha$ can be written as
\[\alpha(x_{[k]}) = \sum_{i \in [r]} \beta_i(x_{I_i})\gamma_i(x_{[k] \setminus I_i})\]
for some multilinear forms $\beta_i : G_{I_i} \to \mathbb{F}_p$ and $\gamma_i : G_{[k]\setminus I_i} \to \mathbb{F}_p$.

We say that a multilinear form $\alpha : G_I \to \mathbb{F}_p$ is \textit{$\eta$-quasirandom} with respect to a mixed-linear map $\beta : G_{[k]} \to \mathbb{F}_p^r$ if for all $\lambda \in \mathbb{F}_p^r$ we have
\[\Big|\exx_{x_{[k]} \in G_{[k]}} \omega\Big(\alpha(x_{I}) + \sum_{i \in [r]} \lambda_i \beta_i(x_{[k]})\Big)\Big| < \eta.\]
More generally, if  $\alpha : G_I \to \mathbb{F}_p^s$ is a multilinear map, we say that it is \textit{$\eta$-quasirandom} with respect to a mixed-linear map $\beta : G_{[k]} \to \mathbb{F}_p^r$ if all non-zero linear combinations of components of $\alpha$ are forms with that property, namely if for all $\mu \in \mathbb{F}_p^s\setminus\{0\}$ the form $\mu_1 \alpha_1 + \dots + \mu_s \alpha_s$ is $\eta$-quasirandom with respect to $\beta$. Sometimes, to simplify the notation, for a multilinear variety $V = \{\beta = 0\}$, we say that $\alpha$ is \textit{$\eta$-quasirandom} with respect to $V$, if $\alpha$ is $\eta$-quasirandom with respect to $\beta$.

\noindent\textbf{Additional notation for varieties and maps.} When $V$ is a multilinear variety, coming from a mixed-linear map $\beta$, we write $V^{(I)}$ for the multilinear variety in $G_{I}$ defined by $\beta_i$ with $I_i \subseteq I$. Furthermore, we denote \textit{slice} $V_{x_I} = \{y_{I^c} \in G_{I^c} : (x_I, y_{I^c}) \in V\}$. We also use the slice restriction notation for maps, if $\phi : X \to H$ is a map on a set $X \subseteq G_{[k]}$, we write $\phi_{x_I} : y_{I^c} \mapsto \phi(x_I, y_{I^c})$.

\noindent\textbf{Asymptotic notation.} As in~\cite{FreimanMultihom}, we use notation $\blc$ and $\bsc$ to simplify the mathematical statements that would otherwise read as `There exists a positive constant $C$ such that the following holds...'. This notation hides positive constants whose particular values are not important. There is no logical difference between $\blc$ and $\bsc$, but we use them to suggest to the reader that we think of some constants as large enough, so we use $\blc$, and of some as small enough, when we use $\bsc$. For example,
\[(\forall x, y > 1) \hspace{3pt} \Big(x \geq \blc y^{\blc}\hspace{3pt}\implies\hspace{3pt}x \geq 100 y^2 \log y\Big)\]
is a shorthand for 
\[(\exists C_1, C_2 > 0)(\forall x, y > 1) \hspace{3pt} \Big(x \geq C_1 y^{C_2}\implies x \geq 100 y^2 \log y\Big).\]
Another example of use of this notation is the following statemnt: if a linear map $\phi: G\to H$ has at least $(1 - \bsc)|G|$ zeros then it is a zero map.

\indent Formally, let $x, y_1, \dots, y_m$ be variables, let $p_1, \dots, p_n$ let parameters, let $X, Y_1, \dots, Y_m \subset \mathbb{R}$ be sets, let $f \colon \mathbb{R}^m \times \mathbb{R}^n \to \mathbb{R}$ be a function and let $P(x, y_1, \dots, y_m)$ be a proposition whose truth value depends on $x, y_1, \dots, y_m$. Then we define 
\[(\forall x \in X)(\forall y \in Y_1)\dots (\forall y \in Y_m)\ \Big( x \geq f(y_1, \dots, y_m; \blc, \dots, \blc) \implies P(x,y_1,\dots,y_m)\Big)\] 
to be a shorthand for
\begin{align}(\exists C_1 > 0) \dots (\exists C_n > 0)&(\forall x \in X)(\forall y \in Y_1)\dots (\forall y \in Y_m) \nonumber\\
&\Big(x \geq f(y_1, \dots, y_m; C_1, \dots, C_n) \implies P(x,y_1,\dots,y_m)\Big).\label{fullConEqn}\end{align} 

If there is additional dependency on some parameter $q$, we write $\blc_q$ and $\bsc_q$. For more examples, see Preliminaries section in~\cite{FreimanMultihom}.

The notation will be also useful in proofs in the situation when we need to assume that one quantity is sufficiently small in terms of another in the obvious way. For example, in the context of multilinear varieties of codimension $r$, frequently we shall need that some small parameter $\eta$ is smaller than, for example, $p^{-100^kr}$. The notation above simplifies the condition to $\eta \leq p^{-\blc r}$.

\subsection{Additive combinatorics}

The following lemma, is a variant of the graph-theoretic heart of Gowers's proof of Balog-Szemer\'edi-Gowers theorem,  appearing implicitly~\cite{GowU4}. 

\begin{lemma}[Lemma 6 in~\cite{QPU4}]\label{gowerspathssingle}Let $G$ be a graph on $n$ vertices with at least $cn^2$ edges. Then there exists a  subset of vertices $X$ of size at least $2^{-5} c n$ with the property that there are at least $2^{-35} c^9 n^5$ paths of length 6 between any two vertices in $X$.\end{lemma}

We use the robust version of Sanders's Bogolyubov-Ruzsa theorem, due to Schoen and Sisask~\cite{SchSis}.

\begin{theorem} [Robust Bogolyubov-Ruzsa lemma~\cite{SchSis}]\label{rbrlemma}
    Let $A \subseteq G$ be a set of density $\alpha$. Then there exists a subspace $U \leq G$ of codimension $O(\log^{O(1)} \alpha^{-1})$ such that for each $u \in U$, there are at least $\alpha^{O(1)}|G|^3$ quadruples $(a_1, a_2, a_3, a_4) \in A^4$ such that $a_1 + a_2 - a_3 - a_4 = u$.
\end{theorem}

We need a structure theorem which relates approximate homomorphisms to exact affine maps, which essentially stems from work of Gowers~\cite{GowU4}. Resolution of Marton's conjecture~\cite{Marton1, Marton2} due to Gowers, Green, Manners and Tao, provides us with polynomial bounds. For the purposes of this paper, an earlier version with quasipolynomial bounds due to Sanders~\cite{Sanders} would also suffice. 

\begin{theorem}[Structure theorem for approximate homomorphisms]\label{invhomm}
    Let $G$ and $H$ be finite-dimensional vector spaces over $\mathbb{F}_p$ and let $A \subset G$. Suppose that $\phi : A \to H$ is a map such that 
    \[
        \phi(x + a) - \phi(x) = \phi(y + a) - \phi(y)
    \]
    holds for at least $c|G|^3$ choices of $x,y, a \in G$ such that $x, y, x+a, y+a \in A$. Then there exists a global affine map $\Phi : G \to H$ such that $\phi(x) = \Phi(x)$ holds for at least $c^{O(1)}|G|$ elements $x \in G$.
\end{theorem}

We record the Gowers-Cauchy-Schwarz inequality.

\begin{lemma}[Gowers-Cauchy-Schwarz inequality]
    Let $X_1, \dots, X_k$ be finite sets, and let $f_I : \prod_{i \in I} X_i \to \mathbb{D}$ be functions for all sets $I \subseteq [k]$. Then
    \[\Big|\exx_{x_1 \in X_1, \dots, x_k \in X_k} \prod_{I \subseteq [k]} f_I(x_I) \Big| \leq \Big|\exx_{x_1,y_1 \in X_1, \dots, x_k, y_k \in X_k} \prod_{I \subseteq [k]} \on{Conj}^{k-|I|}f_{[k]}(x_I, y_{I^c})\Big|^{2^{-k}}.\]
\end{lemma}

\subsection{Multilinear algebra}\label{multalgsubsec}

In this subsection, we list some key results that concern multilinear maps and varieties.

We make heavy use of an effective solution to the so-called partition vs. analytic rank problem~\cite{Janzer2, LukaRank, MoshZhuRank}. 

\begin{theorem}[Partition vs. analytic rank problem]\label{arankprank}
    Let $\phi: G_{[k]} \to \mathbb{F}_p$ be a multilinear form. Suppose that $\on{bias} \phi \geq c$. Then $\on{prank} \phi \leq (\log 2c^{-1})^{O(1)}$.
\end{theorem}

An important related fact is that biased multilinear forms vanish on a variety of bounded codimension, independent of one of the coordinates. Unlike the previous theorem, this fact has optimal bounds up to implicit constant.

\begin{theorem}[Dense multilinear varieties contain low codimensional varieties, Theorem 2 in~\cite{LukaLowCodim}]\label{densetolowcodim}
    Let $\phi: G_{[k]} \to \mathbb{F}_p$ be a multilinear form. Suppose that $\on{bias} \phi \geq c$. Then there exists a multilinear variety $V \subseteq G_{[k-1]}$ of codimension $O(\log(2c^{-1}))$ such that $V \times G_k \subseteq \{\phi = 0\}$.
    
    \indent Equivalently, let $\Phi : G_{[k-1]} \to H$ be a multilinear map, such that $|Z(\Phi)| \geq c|G_{[k-1]}|$. Then there exists a multilinear variety $V \subseteq G_{[k-1]}$ of codimension $O(\log(2c^{-1}))$ such that $V \subseteq Z(\Phi)$.
\end{theorem}

% The following fact was proved in~\cite{MobiusOptimal} and also has essentially optimal bounds.

% \begin{theorem}[Variety of biased restrictions]\label{varietyofbiasedmaps}
%       Let $\alpha : U_{[k]} \times V_{[\ell]} \to \mathbb{F}_p$ be a multilinear form. For any $\delta > 0$ define $S_\delta = \{x_{[k]} \in U_{[k]} : \on{bias}(\alpha_{x_{[k]}}) \geq \delta\}$.\\
%         \indent Suppose that  for some $c > 0$ we have $|S_c| \geq c|U_{[k]}|$. Then $S_{\tilde{c}}$ contains a multilinear variety of codimension at most $O_k(2 \log_p c^{-1} + 2)$, where $\tilde{c} = (c/2p)^{O_k(1)}$.
% \end{theorem}

Let us record an easy lower bound on the size of low codimensional varieties.

\begin{lemma}\label{lowcodsize}
    Let $V \subseteq G_{[k]}$ be a multilinear variety of codimension $r$. Then $|V| \geq p^{-kr}|G_{[k]}|$.
\end{lemma}

We make use of an inequality of Lovett which shows that multilinear varieties are positively-correlated.

\begin{lemma}[Lovett, Claim 1.6 in~\cite{lovettSubaddRank}]\label{poscorvars}
    Let $U, V \subseteq G_{[k]}$ be multilinear varieties defined by global multilinear maps on $G_{[k]}$. Then $|U \cap V||G_{[k]}| \geq |U||V|$.
\end{lemma}

Frequently, we use the fact that a quasirandom form has most of its restrictions quasirandom as well.

\begin{lemma}[Quasirandom restrictions]\label{quasirandomrest}
    Let $\alpha : G_I \to \mathbb{F}_p$ be a multilinear form which is $\eta$-quasirandom with respect to a mixed-linear map $\beta : G_{[k]} \to \mathbb{F}_p^r$. Let $J \subseteq [k]$ be a set such that $I \cap J^c \not= \emptyset$. Then, for all but at most $\sqrt{\eta}p^r |G_{J}|$ points $y_J \in G_J$, we have that $\alpha_{y_{I \cap J}}$ is $\sqrt[4]{\eta}$-quasirandom with respect $\beta_{y_J}$.
\end{lemma}

\begin{proof}
    Let $\beta_i$ be multilinear on $G_{I_i}$ for $i \in [r]$. Let $F \subseteq G_J$ be the collection of $y_J \in G_J$ for which the conclusion fails. By averaging over possible linear combinations of components of $\beta$, we find some $\lambda \in \mathbb{F}_p^r$ such that
    \[\Big|\exx_{x_{J^c} \in G_{[k] \setminus J}} \omega\Big(\alpha(x_{I \cap J^c}, y_{I \cap J}) + \sum_{i \in [r]} \lambda_i \beta_i(x_{I_i \cap J^c}, y_{I_i \cap J})\Big)\Big| \geq \sqrt[4]{\eta}.\]
    holds for $y_J$ in some subset $F' \subseteq F$ of size $|F'| \geq p^{-r}|F|$. Thus, we have
    \begin{align*}
        \sqrt{\eta} p^{-r} \frac{|F|}{|G_J|}\leq \exx_{y_J \in G_J} \id_{F'}(y_J) \sqrt[4]{\eta}^2 \leq \exx_{y_J \in G_J} \Big|\exx_{x_{J^c} \in G_{[k] \setminus J}} \omega\Big(\alpha(x_{I \cap J^c}, y_{I \cap J}) + \sum_{i \in [r]} \lambda_i \beta_i(x_{I_i \cap J^c}, y_{I_i \cap J})\Big)\Big|^2.
    \end{align*}

    Let $i_0 \in I \cap J^c$ be arbitrary. By the Cauchy-Schwarz inequality, we have
    \[ \sqrt{\eta} p^{-r} \frac{|F|}{|G_J|} \leq \exx_{y_J \in G_J} \exx_{x_{J^c \setminus \{i_0\}} \in G_{J^c \setminus \{i_0\}}} \Big| \exx_{x_{i_0} \in G_{i_0}} \omega\Big(\alpha(x_{I \cap J^c}, y_{I \cap J}) + \sum_{i \in [r]} \lambda_i \beta_i(x_{I_i \cap J^c}, y_{I_i \cap J})\Big)\Big|^2.\]

    Let $\mathcal{I}$ be the collection of indices $i \in [r]$ such that $i_0 \in I_i$. After expanding, we obtain

    \begin{align*} \sqrt{\eta} p^{-r} \frac{|F|}{|G_J|} \leq & \exx_{\ssk{y_J \in G_J\\x_{J^c \setminus \{i_0\}} \in G_{J^c \setminus \{i_0\}}\\x_{i_0}, x'_{i_0}\in G_{i_0}}} \omega\Big(\alpha(x_{I \cap J^c \setminus \{i_0\}}, y_{I \cap J}, x_{i_0} - x'_{i_0}) + \sum_{i \in \mathcal{I}} \lambda_i \beta_i(x_{I_i \cap J^c \setminus \{i_0\}}, y_{I_i \cap J}, x_{i_0} - x'_{i_0})\Big)\\
    = & \exx_{\ssk{y_J \in G_J\\x_{J^c \setminus \{i_0\}} \in G_{J^c \setminus \{i_0\}}\\z_{i_0}\in G_{i_0}}} \omega\Big(\alpha(x_{I \cap J^c \setminus \{i_0\}}, y_{I \cap J}, z_{i_0}) + \sum_{i \in \mathcal{I}} \lambda_i \beta_i(x_{I_i \cap J^c \setminus \{i_0\}}, y_{I_i \cap J}, z_{i_0})\Big) \leq \eta,\end{align*}
    by the fact that $\alpha$ is $\eta$-quasirandom with respect to $\beta$.
\end{proof}

% \begin{proposition}
%     Let $M \subseteq G_{[k]}$ be a multilinear set of density $c$ and let $V$ be a mixed linear variety of codimension $r$. Then $|M \cap V| \geq (p^{-s} c / 2)^{O(1)} |G_{[k]}|$. 
% \end{proposition}

% \begin{proof}
%     Let $V' \subseteq G_{[k-1]}$ be the mixed linear variety which the vanishing set of multilinear forms $\alpha_i$ defining $V$, but which are independent of $G_k$. Hence $|V_{x_{[k-1]}}| \geq p^{-r} |G_k|$ for all $x_{[k-1]} \in V'$.\\ 

%     Let $X \subseteq G_{[k-1]}$ be the set of all $x_{[k-1]} \in G_{[k-1]}$ such that $|M_{x_{[k-1]}}| \geq \frac{c}{2} |G_k|$. By averaging, $|X| \geq \frac{c}{2}|G_{[k-1]}|$. By the pigeonhole principle, there exist a value $\lambda \in \mathbb{F}_p^s$, where $s$ is the number of forms $\alpha_i$ defining $V'$, such that $X' = \{x_{[k-1]} \in X' : (\forall i \in [s])\,\alpha_i(x_{I_i}) = \lambda_i\}$ has size at least $p^{-s}\frac{c}{2}|G_{[k-1]}|$. Now convolve in directions $1, \dots, k-1$. The resulting $|M_{x_{[k-1]}}| \geq (p^{-s}c/2)^{2^{k-1}} |G_k|$ and $x_{[k-1]} \in V'$. The claim follows.
% \end{proof}

For the next lemma, we need to recall the notion of connectedness from~\cite{LukaRank}. Firstly, we define a graph $\mathcal{G}$ on the vertex set $G_{[k]}$ where points $x_{[k]}$ and $y_{[k]}$ are joined by an edge if the differ in a single coordinate. A set $S \subseteq G_{[k]}$ is \textit{connected} if the induced graph $\mathcal{G}[S]$ is connected as a graph. In other words, a set in $G_{[k]}$ is connected if we can move between any two points by a sequence of steps, where we change a single coordinate at a time and remain in $S$ throughout the walk. We say that $S$ is of \emph{diameter at most $d$} if there is a path of length at most $d$ between any two vertices in $\mathcal{G}[S]$.

\indent The next lemma shows that intersections of $\{\rho = \mu\}$ with multilinear varieties for quasirandom maps $\rho$ are always non-empty and of small diameter. It is a straightforward generalization of Proposition 16 from~\cite{LukaRank}.

\begin{lemma}\label{connectedlayerslemma}
    Let $\rho : G_{[k]} \to \mathbb{F}_p^r$ be a multilinear map which is $\eta$-quasirandom with respect to a mixed-linear map $\beta : G_{[k]} \to \mathbb{F}_p^s$. If $\eta \leq p^{-\blc(r +s + 1)}$, then the set $\{\rho = \mu\} \cap \{\beta = 0\}$ is non-empty and of diameter at most $(k+2)^2$ for every $\mu \in \mathbb{F}_p^r$.
\end{lemma}

\begin{proof}
    We prove the lemma by induction on $k$. Write $V = \{\beta=0\}$ and let $\beta_i$ be multilinear on $G_{I_i}$ for $i \in [s]$. Let $V' \subseteq G_{[k-1]}$ be the variety defined by $\beta_i$ where $k \notin I_i$, i.e. $V' = V^{([k-1])}$.

    \begin{claim}
        For all but at most $\sqrt{\eta} p^{2r + 2s}|G_{[k-1]}|^2$ pairs $(x_{[k-1]}, y_{[k-1]})$ in $V'$ there exist at least $p^{-2r-2s}|G_k|$ of $z_k \in G_k$ such that 
        \[(x_{[k-1]}, z_k), (y_{[k-1]}, z_k) \in \{\rho = \mu\} \cap V.\]
    \end{claim}

    \begin{proof}
        It suffices to show existence of a single such element $z_k$, as then the set of such elements is a non-empty coset of a subspace of codimension at most $2r+2s$. Let $\mathcal{I}$ be the collection of $i \in [s]$ such that $k \in I_i$. Write $\rho_i(x_{[k]}) = R_i(x_{[k-1]}) \cdot x_k$ for $i \in [r]$ and $\beta_i(x_{I_i}) = B_i(x_{I_i \setminus \{k\}}) \cdot x_k$ for $i \in \mathcal{I}$. Note that, when $R_i(x_{[k-1]})$, $R_i(y_{[k-1]})$, $i \in [r]$ are all independent and
        \[\langle R_i(x_{[k-1]}), R_i(y_{[k-1]}) : i \in [r] \rangle  \cap \langle B_i(x_{I_i \setminus \{k\}}), B_i(y_{I_i \setminus \{k\}}) : i \in \mathcal{I} \rangle = \{0\}\]
        then we certainly have desired $z_k$. Let $M$ be the number of pairs $(x_{[k-1]}, y_{[k-1]}) \in G_{[k-1]}\times G_{[k-1]}$ where that fails. Thus for some $\lambda^{(1)}, \lambda^{(2)} \in \mathbb{F}_p^r$, not both zero, we either have
        \[\sum_{i \in [r]} \lambda^{(1)}_i R_i(x_{[k-1]}) + \sum_{i \in [r]}\lambda^{(2)}_i R_i(y_{[k-1]}) = 0\]
        or there are some further $\lambda^{(3)}, \lambda^{(4)} \in \mathbb{F}_p^{\mathcal{I}}$ for which
        \[\sum_{i \in [r]} \lambda^{(1)}_i R_i(x_{[k-1]}) + \sum_{i \in [r]}\lambda^{(2)}_i R_i(y_{[k-1]})  = \sum_{ i\in \mathcal{I}} \lambda^{(3)}_i B_i(x_{[k-1]}) + \sum_{ i\in \mathcal{I}} \lambda^{(4)}_i B_i(y_{[k-1]}) \not= 0.\]
        In the former case, we may set $\lambda^{(3)} = \lambda^{(4)} = 0$, to unify these conclusions to
        \[\sum_{i \in [r]} \lambda^{(1)}_i R_i(x_{[k-1]}) + \sum_{i \in [r]}\lambda^{(2)}_i R_i(y_{[k-1]})  = \sum_{ i\in \mathcal{I}} \lambda^{(3)}_i B_i(x_{[k-1]}) + \sum_{ i\in \mathcal{I}} \lambda^{(4)}_i B_i(y_{[k-1]}) .\]
        By averaging over linear combinations $\lambda^{(1)}, \dots, \lambda^{(4)}$, we have $p^{-2r - 2s}M$ pairs where 
        \[\sum_{i \in [r]} \lambda^{(1)}_i R_i(x_{[k-1]}) + \sum_{i \in [r]}\lambda^{(2)}_i R_i(y_{[k-1]}) = \sum_{ i\in \mathcal{I}} \lambda^{(3)}_i B_i(x_{[k-1]}) + \sum_{ i\in \mathcal{I}} \lambda^{(4)}_i B_i(y_{[k-1]}).\]

        Without loss of generality, $\lambda^{(1)} \not = 0$. By averaging, there exists $y_{[k-1]} \in G_{[k-1]}$ such that the above equality holds for at least $p^{-2r - 2s}\frac{M}{|G_{[k-1]}|}$ choices of $x_{[k-1]} \in G_{[k-1]}$. Write $v = \sum_{i \in [r]}\lambda^{(2)}_i R_i(y_{[k-1]}) -  \sum_{ i\in \mathcal{I}} \lambda^{(4)}_i B_i(y_{[k-1]})$.
        
        \indent Taking expectation over $x_{[k-1]}$ and taking dot product with $x_k$, we get 
        \begin{align*}p^{-2r - 2s}\frac{M}{|G_{[k-1]}|^2} \leq &\exx_{x_{[k-1]} \in G_{[k-1]}} \id\Big(\sum_{i \in [r]} \lambda^{(1)}_i R_i(x_{[k-1]})  - \sum_{ i\in \mathcal{I}} \lambda^{(3)}_i B_i(x_{[k-1]}) - v= 0\Big) \\
        = & \exx_{x_{[k]} \in G_{[k]}} \omega\Big(\sum_{i \in [r]} \lambda^{(1)}_i R_i(x_{[k-1]}) \cdot x_k  - \sum_{ i\in \mathcal{I}} \lambda^{(3)}_i B_i(x_{[k-1]})  \cdot x_k - v \cdot x_k\Big) \\
        = & \exx_{x_{[k]} \in G_{[k]}} \omega\Big(\sum_{i \in [r]} \lambda^{(1)}_i \rho_i(x_{[k]})  - \sum_{ i\in \mathcal{I}} \lambda^{(3)}_i \beta_i(x_{[k]}) - v\cdot x_k\Big).\end{align*}

        By Cauchy-Schwarz inequality and the fact that $\sum_{i \in [r]} \lambda^{(1)}_i \rho_i$ is $\eta$-quasirandom with respect to $\beta$, we have $p^{-2r - 2s}\frac{M}{|G_{[k-1]}|^2} \leq \sqrt{\eta}$, proving the claim.
    \end{proof}

    By Lemma~\ref{quasirandomrest}, the multilinear map $\rho_{z_k}$ is $\sqrt[4]{\eta}$-quasirandom with respect to $V_{z_k}$ for all but at most $\sqrt{\eta} p^{2r + 2s}|G_k|$ elements $z_k \in G_k$. Hence, the induction hypothesis applies to the slice $V_{z_k}$ for each such $z_k$.

    % \begin{claim}
    %     For all  but at most $\sqrt{\eta} p^{2r + 2s}|G_k|$ elements $z_k \in G_k$ we have that the multilinear map $\rho_{z_k}$ is $\sqrt[8]{\eta}$-quasirandom with respect to $V_{z_k}$.
    % \end{claim}

    % \begin{proof}
    %     Let $F \subseteq G_k$ be the set of all $z_k \in G_k$ for which the multilinear map $\rho_{z_k}$ fails to be $\sqrt[8]{\eta}$-quasirandom with respect to $V_{z_k}$. By averaging, we get a non-zero linear combination $\lambda \in \mathbb{F}_p^r$ and a linear combination $\mu \in \mathbb{F}_p^s$ for which $\lambda \cdot \rho_{z_k} + \mu \cdot \beta_{z_k}$ fails to be quasirandom for at least $p^{-r -s}|F|$ elements $z_k$. Let $F'$ be the set of such elements. Then
    %     \[ \sqrt[4]{\eta} p^{-r-s} |F| \leq \exx_{x_k \in G_k} \id(z_k \in F') \sqrt[4]{\eta} \leq \exx_{x_k \in G_k} \Big|\exx_{x_{[k-1]}} \omega\Big(\lambda \cdot \rho(x_{[k-1]}, z_k) + \mu \cdot \beta(x_{[k-1]}, z_k)\Big)\Big|^2 = .\]
    %     The final bound follows from the fact that $\lambda  \cdot \rho$ is $\eta$-quasirandom with respect to $\beta$.
    % \end{proof}

    By induction hypothesis for $\{\rho_{z_k} = \mu\} \cap V_{z_k}$ for suitable $z_k$ depending on the pair $(x_{[k-1]}, y_{[k-1]})$, for all but at most $\eta^{\Omega(1)}p^{O(r +s)}|G_{[k]}|^2$ of pairs $(x_{[k]}, y_{[k]})$ we can join them by paths of length at most $(k+1)^2 + 2$. Any point in the set is at distance at most $k$ from $p^{-k(r+s)}$ points, so we are done.
\end{proof}

The next lemma shows that non-zero global multilinear maps have a significant proportion of non-zero values.

\begin{lemma}\label{nonzeromapslemma}
    Let $\phi : G_{[k]} \to H$ be a non-zero multilinear map. Then $\phi(x_{[k]}) \not= 0$ for at least $2^{-k}|G_{[k]}|$ points $x_{[k]} \in G_{[k]}$.
\end{lemma}

\begin{proof}
    Take any $x_{[k]}$ where $\phi$ is non-zero. By induction on $d \in [0, k]$, we show that there are at least $2^{-d}G_{[d]}$ choices of $y_{[d]} \in G_{[d]}$ such that $\phi(y_{[d]}, x_{[d + 1, k]}) \not= 0$.
    
    The base case $d= 0$ is trivial. Suppose now that the claim holds for some $d \in [k]$. For each $y_{[d]}$ such that $ \phi(y_{[d]}, x_{[d + 1, k]}) \not= 0$, the map $z_{d + 1} \mapsto \phi(y_{[d]}, z_{d+1}, x_{[d + 2, k})$ is a non-zero linear map, so it vanishes for at most half of values of $z_{d+1}$. The claim now follows.
\end{proof}

In section~\ref{extensionSection} we shall develop extension theory for MM-$r$-maps in detail. Here we record known results.

The following result, proved in~\cite{MultilinearExtensions}, will be of key importance. It gives a useful criterion for extensions of MM-$r$-maps to varieties containing the domain.

\begin{theorem}[Theorem 3.1 in~\cite{MultilinearExtensions}]\label{qrExtension}Let $\rho \colon G_{[k]} \to \mathbb{F}_p$ and $\beta_i \colon G_{I_i} \to \mathbb{F}_p$, $i \in [m]$ be multilinear forms. Write $\mathcal{I} = \{i \in [m] \colon I_i = [k]\}$. Let $B = \{x_{[k]} \in G_{[k]} \colon (\forall i \in [m])\, \beta_i(x_{I_i}) = 0\}$ and let $B^0 = \{x_{[k]} \in B \colon \rho(x_{[k]}) = 0\}$. Let $\phi \colon B^0 \to H$ be a multilinear map.

\indent Suppose that for each $\lambda \in \mathbb{F}^{\mathcal{I}}$
\begin{equation}\label{qrConditionForExtensions}\exx_{x_{[k]}} \omega\Big(\rho(x_{[k]}) + \sum_{i \in \mathcal{I}} \lambda_i \beta_i(x_{[k]})\Big) < \frac{1}{2k^2}p^{-(2k^2 + k + 1) 2^{2k+3}(m + 1)}.\end{equation}
Then, for each $z_{[k]} \in B \setminus B^0$ and $h_0 \in H$, there is a unique multilinear map $\phi^{\on{ext}} \colon B \to H$ such that $\phi^{\on{ext}}|_{B^0} = \phi$ and $\phi^{\on{ext}}(z_{[k]}) = h_0$.\end{theorem}

The previous theorem was proved in order to show that MM-$r$-maps are closely related to global multilinear maps. The following theorem was the main result of~\cite{MultilinearExtensions}.

\begin{theorem}[Gowers and Mili\'cevi\'c, Theorem 1.4 in~\cite{MultilinearExtensions}]\label{basicextensionstheory}
     Let $V \subseteq G_{[k]}$ be a multilinear variety of codimension at most $r$ and let $\phi : V \to H$ be a multilinear map. Then $\phi$ coincides with a global multilinear map on a multilinear variety of codimension at most $(2r)^{O(1)}$.
\end{theorem}

\subsection{Abstract Balog-Szemer\'edi-Gowers theorem}

To make sure that we do not cause confusion when working with additive quadruples, we introduce the additional notation that indicates the choice of signs that we have in mind. We write a bold dot above variable to indicate that we take it with negative sign. Namely, writing $(\upd{a}, b, \overset{\bcdot}{c}, d)$ means that we have $a + c = b + d$ and $(\overset{\bcdot}{a}, b, c, \overset{\bcdot}{d})$ means that we have $a + d = b + c$.

The following theorem, called abstract Balog-Szemer\'edi-Gowers theorem, was proved implicitly in~\cite{QPU4}. The following version is a slight specialization of Theorem 4.1 in~\cite{QPU4genab}. The latter theorem applies to approximate groups, which we do not need in the present paper.

\begin{theorem}[Abstract Balog-Szemer\'edi-Gowers theorem]\label{absg}
    \indent Let $G$ be an abelian group and let $A \subseteq G$. Suppose that, for each $i \in [36]$, we have a collection $\mathcal{Q}_i$ of additive quadruples in $A$, satisfying the following properties:
    \begin{itemize}
    \item[\textbf{(i)}] (largeness) $|\mathcal{Q}_1| \geq c |G|^3$,
    \item[\textbf{(ii)}] (symmetry) for each $i\in [36]$, if $(\upd a_1, a_2, a_3, \upd a_4) \in \mathcal{Q}_i$, then 
    \begin{itemize}
        \item[\textbf{(S1)}] $(\upd a_3,a_4,$ $ a_1, \upd a_2) \in \mathcal{Q}_i$,
        \item[\textbf{(S2)}] $(\upd a_2, a_1, a_4, \upd a_3) \in \mathcal{Q}_i$, and
        \item[\textbf{(S3)}] $(\upd a_1, a_3, a_2, \upd a_4) \in \mathcal{Q}_i$,
    \end{itemize} 
    \item[\textbf{(iii)}] (weak transitivity) for all indices $i, j \in [36], i + j \leq 36$, for any additive quadruple $(\upd a_1, a_2, a_3, \upd a_4) \in A^4$, if there are at least $c' |G|$ pairs $(b, b') \in A^2$ such that $(\upd a_1, a_2, b, \upd{b'}) \in \mathcal{Q}_i$ and $(\upd b, b', a_3, \upd a_4) \in \mathcal{Q}_j$, then $(\upd a_1, a_2, a_3, \upd a_4) \in \mathcal{Q}_{i+j}$.
    \end{itemize}

    Fix an integer $k$. Then, provided $c' \leq (c/2)^{\blc}$ and $|G| \geq 8c^{-1}$, there exists a subset $A' \subseteq A$, of size $|A'| \geq (c/2)^{O(1)}|G|$, with the following property. For $\ell \in [k]$, given an $\ell$-tuple $a_{[\ell]} \in {A'}^{\ell}$, define recursively a collection $\mathcal{Z}_\ell^{(a_{[\ell]})}$ of $(3\ell)$-tuples in $A$, (depending on elements $a_{[\ell]}$) as follows. For $\ell = 1$, $\mathcal{Z}_1^{(a_1)}$ consists of all triples $(y_1, y_2, y_3) \in A^3$ such that $(\upd{a}_1, y_1, y_3, \upd{y}_2) \in \mathcal{Q}_{36}$. For $\ell \geq 2$, the collection $\mathcal{Z}_\ell^{(a_{[\ell]})}$ consists of all $3\ell$-tuples $(y_1, \dots, y_{3\ell}) \in A^{3\ell}$ such that 
    \begin{itemize}
        \item $\sum_{i \in [\ell]} (-1)^i a_i = \sum_{i \in [3\ell]} (-1)^i y_i$,
        \item $(\upd{a_{\ell}}, a_{\ell} + y_{3\ell -1} - y_{3\ell}, y_{3\ell}, \upd{y_{3\ell - 1}})\in \mathcal{Q}_{36}$,
        \item $\Big((y_{3\ell-3})^\bcdot, y_{3\ell-3} - y_{3\ell -2} + y_{3\ell -1} - y_{3\ell} + a_{\ell}, y_{3\ell - 2}, (y_{3\ell - 1} - y_{3\ell} + a_{\ell })^\bcdot\Big)\in \mathcal{Q}_{36}$, and,
        \item $(y_1, \dots, y_{3\ell - 4}, a_\ell + y_{3\ell - 3} - y_{3\ell - 2} + y_{3\ell - 1} - y_{3\ell}) \in \mathcal{Z}_{\ell - 1}^{(a_{[\ell - 1]})}$.
    \end{itemize}
    Then, for all $\ell \in [k]$, we have $|\mathcal{Z}_\ell^{(a_{[\ell]})}| \geq (c/2)^{O_\ell(1)} |G|^{3\ell - 1}$.
\end{theorem}

The flexibility of the theorem above is crucial, allowing us to use collections of additive quadruples $\mathcal{Q}_i$ coming from complicated respectedness conditions. In the \textbf{Step 1.1} we shall use $\mathcal{Q}_i$ as the set of additive quadruples $(\upd{x}_1, x_2, x_3, \upd{x}_4)$ such that 
    \[|Z(\phi_{x_1} - \phi_{x_2} - \phi_{x_3} + \phi_{x_4})| \geq c^i|G_{[k]}|,\]
where we stress that the zero density changes with $\mathcal{Q}_i$, and in the \textbf{Step 4.1} $\mathcal{Q}_i$ will be the set of additive quadruples $(\upd{x}_1, x_2, x_3, \upd{x}_4)$ such that \[\phi_{x_1} - \phi_{x_2} - \phi_{x_3} + \phi_{x_4} = 0\text{ on }W_i\cap V_{x_1} \cap V_{x_2} \cap V_{x_3} \cap V_{x_4}\]
where $\phi_{x_1} : V_{x_1} \to H, \dots, \phi_{x_4} : V_{x_4} \to H$ and the variety $W_i$ changes with $\mathcal{Q}_i$. In particular, the abstract Balog-Szemer\'edi-Gowers theorem interacts well with the weak sheaf properties of multilinear varieties which is key in the \textbf{Step 4.1}, .

\subsection{From Freiman multihomomorphisms to systems of multilinear maps}

In this subsection, we carry out the \textbf{Preliminary step}. Recall that we assume Theorem~\ref{strFmult} for the product of $k$ vector spaces.

\begin{proposition}\label{mlmaps-pass-step-proposition}
    Let $\phi : A \to H$ be a Freiman multihomomorphism on a set $A \subseteq G_{[0,k]}$ of density $c$. Then there exists a set $X \subseteq G_0$ and a collection of global multilinear map $\phi_x : G_{[k]} \to H$, where $x \in X$, such that $|Z(\phi_{x_1} - \phi_{x_2} + \phi_{x_3} - \phi_{x_4}))| \geq \eplog{c}|G_{[k]}|$ holds for at least $\eplog{c}|G_0|^3$ additive quadruples $(\upd{x}_1, x_2, \upd{x}_3, x_4)$ in $X$. Moreover, there exist a global multiaffine map $\psi : G_{[k]} \to H$ and a point $t_{[k]} \in G_{[k]}$ such that for each $x \in X$, there exist $\eplog{c}|G_{[k]}|$ choices of $y_{[k]} \in G_{[k]}$ such that
    \[\phi_x(y_{[k]}) = \phi(x, (y + t)_{[k]}) + \psi(x, y_{[k]}).\]
\end{proposition}

\begin{proof}
    Let $X$ be the collection of $x \in G_0$ such that $|A_{x}| \geq \frac{c}{2}|G_{[k]}|$. By averaging, $|X| \geq \frac{c}{2} |G_0|$. For each $x \in X$, we may apply Theorem~\ref{strFmult} for products of $k$ vector spaces, to conclude that there exists a global multiaffine map $\phi^{\on{aff}}_x : G_{[k]} \to H$ such that $\phi^{\on{aff}}_x(y_{[k]}) = \phi(x, y_{[k]})$ for a set $S_x$ of $k$-tuples $y_{[k]} \in G_{[k]}$ of size at least $\eplog{c}|G_{[k]}|$. By Cauchy-Schwarz inequality, we have 
    \begin{align*}&\exx_{\ssk{x_1, x_2, d \in G_0\\y_{[k]} \in G_{[k]}}} \id_{S_{x_1 + d}}(y_{[k]})\id_{S_{x_2}}(y_{[k]})\id_{S_{x_2 + d}}(y_{[k]})\id_{S_{x_2}}(y_{[k]}) = \exx_{\ssk{d \in G_0\\y_{[k]} \in G_{[k]}}} \Big(\exx_{x \in G_0}  \id_{S_{x + d}}(y_{[k]})\id_{S_{x}}(y_{[k]})\Big)^2 \\
    &\hspace{2cm}\geq\Big( \exx_{\ssk{d \in G_0\\y_{[k]} \in G_{[k]}}} \exx_{x \in G_0}  \id_{S_{x + d}}(y_{[k]})\id_{S_{x}}(y_{[k]})\Big)^2  = \Big( \exx_{y_{[k]} \in G_{[k]}} \Big(\exx_{x \in G_0} \id_{S_{x}}(y_{[k]})\Big)^2 \Big)^2 \\
    &\hspace{2cm}\geq \Big(\exx_{\ssk{y_{[k]} \in G_{[k]}\\x \in G_0}}\id_{S_{x}}(y_{[k]}) \Big)^4 \geq \eplog{c}. \end{align*}
    Hence there are $\eplog{c} |G_0|^3$ additive quadruples $(\upd{x}_1, x_2, \upd{x}_3, x_4)$ in $X$ such that $|S_{x_1} \cap S_{x_2} \cap S_{x_3} \cap S_{x_4}| \geq \eplog{c} |G_{[k]}|$. Since $\phi$ is a Freiman multihomomorphism, we have $\phi^{\on{aff}}_{x_1}(y_{[k]}) - \phi^{\on{aff}}_{x_2}(y_{[k]}) + \phi^{\on{aff}}_{x_3}(y_{[k]}) - \phi^{\on{aff}}_{x_4}(y_{[k]}) = 0$ when $y_{[k]} \in  S_{x_1} \cap S_{x_2} \cap S_{x_3} \cap S_{x_4}$.

    Let $\phi_x$ be the multilinear part of $\phi^{\on{aff}}_x$ for each $x \in X$. Note the identity 
    \[\phi_x(y_1 - z_1, \dots, y_k - z_k) = \sum_{I \subseteq [k]} (-1)^{|I|} \phi^{\on{aff}}_x(y_I, z_{[k] \setminus I}).\]
    Note that for each additive quadruple $(\upd{x}_1, x_2, \upd{x}_3, x_4)$ above we have at least $\eplog{c} |G_{[k]}|^2$ choices of $y_{[k]}, z_{[k]} \in G_{[k]}$ such that $(y_I, z_{[k] \setminus I}) \in S_{x_1} \cap S_{x_2} \cap S_{x_3} \cap S_{x_4}$ holds for all $I \subseteq [k]$. Thus $|Z(\phi_{x_1} - \phi_{x_2} + \phi_{x_3} - \phi_{x_4}))| \geq \eplog{c}|G_{[k]}|$.

    For the second part of the proposition, for each $x \in X$, we similarly have $\eplog{c} |G_{[k]}|^2$ choices of $y_{[k]}, z_{[k]} \in G_{[k]}$ such that $(y_I, z_{[k] \setminus I}) \in S_{x}$. Thus
    \[\phi_x(y_1 - z_1, \dots, y_k - z_k) = \sum_{I \subseteq [k]} (-1)^{|I|} \phi^{\on{aff}}_x(y_I, z_{[k] \setminus I}) = \sum_{I \subseteq [k]} (-1)^{|I|} \phi(x, y_I, z_{[k] \setminus I}).\]

    By averaging over $z_{[k]}$ and making a change of variables, where we replace $y_i$ with $y_i + z_i$, we conclude that
    \[\phi_x(y_1, \dots, y_k) = \sum_{I \subseteq [k]} (-1)^{|I|} \phi(x, (y + z)_I, z_{[k] \setminus I})\]
    holds for $\eplog{c}|G_{[k]}|$ choices of $y_{[k]}$. We may apply the main result, Theorem~\ref{strFmult}, for lower arities for terms where $I \not=\emptyset$, to complete the proof.   
\end{proof}

\section{Extension theory for MM-$r$-maps}\label{extensionSection}

\subsection{Vanishing and simultaneous extensions of MM-maps}

Recall that a general extension result for MM-$r$-maps was proved in~\cite{MultilinearExtensions}, which was stated as Theorem~\ref{qrExtension} in the present paper. Our goal in this subsection is to derive a result about simultaneous extension of such maps, which appears as Theorem~\ref{sim-ext-final-ver}. 

% Firstly, we prove a uniqueness result.

% \begin{lemma}[Extension uniqueness]
%     Let $V$ be a multilinear variety defined by forms $\beta_i : G_{I_i} \to \mathbb{F}_p$ for $i \in [s]$ and let $\alpha : G_{I_0} \to \mathbb{F}_p^{r}$ be a multilinear map. Suppose that
%     \[\Big|\exx_{x_{[k]} \in G_{[k]}} \omega\Big(\lambda \cdot \alpha_1(x_{I_0}) + \lambda' \cdot \alpha_2(x_{I_0}) + \sum_{i \in [s]} \mu_i \beta_i(x_{I_i})\Big)\Big| \leq .\]
%     Let $\phi : V \to H$ be a multilinear map such that $\phi = 0$ on $V \cap (\{\alpha = 0\} \times G_{[k] \setminus I_0})$. Then $\phi = 0$ on $V$.
% \end{lemma}

% \begin{proof}
    
% \end{proof}

Before proving our main vanishing and extension results, we need an easy extension lemma when the domain is extremely dense inside a multilinear variety.

\begin{lemma}\label{trivialextensionsverydense}
    Let $V$ be a multilinear variety of codimension at most $r$ inside $G_{[k]}$. Let $\varepsilon > 0$ and let $A \subseteq V$ satisfy $|V \setminus A| \leq \varepsilon|G_{[k]}|$. Provided $\varepsilon \leq p^{-\blc (r+1)}$, there exists a subset $A' \subseteq A$ of size $|V \setminus A| \leq \varepsilon^{\Omega(1)}|G_{[k]}|$ with the following property.
    
    Let $\phi : A \to H$ be a multilinear map. Then there exists a unique multilinear map $\Phi : V \to H$ such that $\Phi = \phi$ on $A'$.
\end{lemma}

In the proof, we make use of the following lemma concerning almost linear maps in vector spaces.

\begin{lemma}\label{triviallinearextension}
     Let $\varepsilon \in (0, \bsc)$. Suppose that $G$ and $H$ are finite-dimensional $\mathbb{F}_p$-vector spaces. Let $X \subseteq G$ and let $P \subseteq X \times X$ be a set of pairs of size $|P| \geq (1-\varepsilon)|G|^2$. Assume that $x + y \in X$ for all $(x,y) \in P$. Then there exists a set $S \subseteq X$ of size $(1 - \varepsilon^{\Omega(1)})|G|$ with the following property.
     
     Suppose that $\phi \colon X \to H$ is a map on $X \subseteq G$ such that $\phi(x) + \phi(y) = \phi(x + y)$ holds for all $(x, y) \in P$. Then there is a unique linear map $\alpha \colon G \to H$ such that $\alpha = \phi$ on $S$.
\end{lemma}

\begin{proof} Before we proceed with the proof, we make an observation concerning respected additive    quadruples. By Cauchy-Schwarz inequality, we have at least $(1-2\varepsilon)|G|^3$ additive quadruples $(x_1, x_2, x_3, x_4)$ such that $x_1 + x_2 = x_3 + x_4$ and $(x_1, x_2), (x_3, x_4) \in P$. Thus, $\phi(x_1) + \phi(x_2) = \phi(x_3) + \phi(x_4)$ holds. By averaging, we obtain a set of pairs $P_1 \subseteq X \times X$ of size $(1-2\sqrt{\varepsilon})|G|^2$, such that for each $(x_1, x_2) \in P_1$ there exist at least $(1-2\sqrt{\varepsilon})|G|$ pairs $(x_3, x_4) \in X^2$ with $x_1 + x_2 = x_3 + x_4$ and $\phi(x_1) + \phi(x_2) = \phi(x_3) + \phi(x_4)$. Similarly, by reordering variables and averaging, we may obtain a set of pairs $P_2 \subseteq X \times X$ of size $(1-2\sqrt{\varepsilon})|G|^2$, such that for each $(x_1, x_2) \in P_2$ there exist at least $(1-2\sqrt{\varepsilon})|G|$ pairs $(x_3, x_4) \in X^2$ with $x_1 - x_2 = x_3 - x_4$ and $\phi(x_1) - \phi(x_2) = \phi(x_3) - \phi(x_4)$.

    Let $S$ be the set of elements $s \in X$ such that $(x,s-x) \in P$ for at least $(1 - \sqrt{\varepsilon})|G|$ choices of $x \in X$, i.e. popular sums. By averaging, $|S| \geq (1 - \sqrt{\varepsilon})|G|$. The following claim shows that $\phi$ behaves linearly on $S$.

    \begin{claim}
        Let $\lambda_1, \dots, \lambda_d \in \{-1,1\}$. Then, for all $x_1, \dots, x_d \in S$, we have at least $(1 - 2^{O_d(1)} \varepsilon^{\Omega_d(1)})|G|^{2d-1}$ $2d$-tuples $(y_1, y'_1, \dots, y_d, y'_d)$ in $X$ such that 
        \[\sum_{i \in [d]} \lambda_i x_i = \sum_{i \in [d]} \lambda_i (y_i + y'_i)\]
        and
        \[\sum_{i \in [d]} \lambda_i \phi(x_i) = \sum_{i \in [d]} \lambda_i (\phi(y_i) + \phi(y'_i)).\]
    \end{claim}

    \begin{proof} We prove the claim by induction on $d$. The base case $d = 1$ follows from the definition of $S$.

    Suppose now that the claim holds for some $d \geq 1$ and let $\lambda_1, \dots, \lambda_{d+1}\in \{-1,1\}$ and $x_1, \dots, x_{d+1}$ be given. Apply the induction hypothesis for the first $d$ scalars and elements, to get a set $Y$ of the $2d$-tuples in $X$ with the claimed property. Let $Z$ be the set of all $(z_{d+1}, z'_{d+1}) \in P$ such that $z_{d+1} + z'_{d+1} = x_{d+1}$. Since $x_{d+1} \in S$, we have $|Z| \geq (1- \sqrt{\varepsilon})|G|$.

    In particular, there are at least $(1 - 2^{O_d(1)} \varepsilon^{\Omega_d(1)})|G|^2$ pairs $(y_d, z_{d+1})$ in $X \times X$ such that $y_d$ is the appropriate element of at least $(1 - 2^{O_d(1)} \varepsilon^{\Omega_d(1)})|G|^{2d-2}$ $2d$-tuples in $Y$ and $z_{d+1}$ is the first element of a pair in $Z$. Let $\tilde{P}$ be the set of such pairs. If $\lambda_d = \lambda_{d+1}$, we note that $|\tilde{P} \cap P_1| \geq (1 - 2^{O_d(1)} \varepsilon^{\Omega_d(1)})|G|^2$. Thus, there exist at least $(1-2\sqrt{\varepsilon})|G|$ of pairs $(w_d, w_{d+1}) \in X \times X$ such that $y_d +z_{d+1} = w_d + w_{d+1}$ and $\phi(y_d) + \phi(z_{d+1}) = \phi(w_d) + \phi(w_{d+1})$. Combining everything, we have
    \[\sum_{i \in [d+1]} \lambda_i x_i = \sum_{i \in [d-1]} \lambda_i (y_i + y'_i) + \lambda (w_d + y'_d) + \lambda_{d+1}(w_{d+1} + z'_{d+1})\]
    and
    \[\sum_{i \in [d+1]} \lambda_i \phi(x_i) = \sum_{i \in [d-1]} \lambda_i (\phi(y_i) + \phi(y'_i)) + \lambda (\phi(w_d) + \phi(y'_d)) + \lambda_{d+1}(\phi(w_{d+1}) + \phi(z'_{d+1})).\]
    The number of $(2d+2)$-tuples $(y_1, y'_1, \dots, y_{d-1}, y'_{d-1}, w_d, y'_d, w_{d+1}, z'_{d+1})$ appearing above is at least $(1 - 2^{O_d(1)} \varepsilon^{\Omega_d(1)})|G|^{2d + 1}$, as desired.

    If $\lambda_d \not= \lambda_{d+1}$, we repeat the same argument, except that we use $P_2$ in place of $P_1$. 
    \end{proof}

    Apply the claim above for $d = 9$ and sings $\lambda = (1, -1, 1, 1, -1, 1, -1, 1, -1)$. It means that the expression 
    \[\phi(x_1) - \phi(x_2) + \phi(x_3) + \phi(x_4) - \phi(x_5) + \phi(x_6) - \phi(x_7) + \phi(x_8) - \phi(x_9)\] 
    takes a single value when $x_1, \dots, x_9$ are elements in $S$ such that $\sum_{i\in [9]} \lambda_i x_i = 0$. However, for the choice $x_1 = x_2 = x_3$, $x_4 = x_5 = x_6$, $x_7 = x_8 = x_9$ with $x_9  = x_3 + x_6$ and $(x_3, x_6) \in P$ we obtain zero value, so the expression always vanishes. Finally, this allows us to find a linear extension of $\phi$ from $S$ in the obvious way, by defining $\phi(a) = \phi(x_1) - \phi(x_2) + \phi(x_3)$ for any triple in $S$ with $a = x_1 - x_2 + x_3$. 
\end{proof}

We may now prove Lemma~\ref{trivialextensionsverydense}.

\begin{proof}[Proof of Lemma~\ref{trivialextensionsverydense}]
    We prove the claim by induction on $k$. The base case $k=1$ follows from Lemma~\ref{triviallinearextension}.
    
    \indent Suppose that the claim holds for some $k - 1$ and that $\phi :A \to H$ is multilinear and $|V\setminus A| \leq \varepsilon |G_{[k]}|$. Let $U \subseteq G_{k}$ be the subspace of elements $x_k$ such that the slice $V_{x_k}$ is non-empty, i.e. $U  = V^{(\{k\})}$. Let $\eta > 0$ be an error parameter for which the induction hypothesis holds for varieties of codimension at most $4r$, to be specified later. Let $X \subseteq U$ be the collection of $x_k$ such that $|V_{x_{k}} \setminus A_{x_{k}}| \leq \frac{1}{4}\eta |G_{[k-1]}|$. By averaging, we have $|U \setminus X| \leq 4\eta^{-1}\varepsilon |G_k|$.
    
    \indent For each $x_k \in X$, by induction hypothesis, there exists a set $A'_{x_k} \subseteq A_{x_k}$ such that $|V_{x_k} \setminus A'_{x_k}| \leq \eta^{\Omega(1)}|G_{[k-1]}|$ and a unique multilinear map $\phi_{x_k}^{\on{ext}} : V_{x_k} \to H$ such that $\phi_{x_k} = \phi_{x_k}^{\on{ext}}$ on $A'_{x_k}$. We set $A' = \cup_{x_k \in X} A'_{x_k} \times \{x_k\}$, which is the desired set. It remains to define $\Phi$ at the remaining points $((U \setminus X) \times G_{[k-1]}) \cap V$.
    
    \indent Let $x_{[k]} \in V$ be arbitrary. Note that $V_{x_{[k-1]}}$ is a subspace of $U$ of codimension at most $r$. Hence $|X \cap V_{x_{[k-1]}}| \geq (1 - 4\eta^{-1}p^r \varepsilon)|V_{x_{[k-1]}}| \geq \frac{2}{3}|V_{x_{[k-1]}}|$, provided $\varepsilon \leq \eta^2$ and $\varepsilon \leq  p^{-\blc (r+1)}$. Therefore, there exist $y_k, z_k \in X \cap V_{x_{[k-1]}}$ such that $y_k + z_k = x_k$. We define $\Phi(x_{[k]}) = \phi_{y_k}^{\on{ext}}(x_{[k-1]}) + \phi_{z_k}^{\on{ext}}(x_{[k-1]})$. This shows that $\Phi$ is unique as well. It remains to check that $\Phi$ is well-defined and multilinear.
    
    \indent Let us first show that $\Phi$ is well-defined. If $y'_k, z'_k \in X \cap V_{x_{[k-1]}}$ are another choice of elements with $y'_k + z'_k = x_k$, then we have
    \[|(V_{y_k} \cap V_{z_k} \cap V_{y'_k} \cap V_{z'_k}) \setminus (A_{y_k} \cap A_{z_k} \cap A_{y'_k} \cap A_{z'_k})| \leq |V_{y_k} \setminus A_{y_k}| + \dots + |V_{z'_k} \setminus A_{z'_k}| \leq \eta |G_{[k-1]}|,\]
    so $\phi_{y_k}^{\on{ext}} + \phi_{z_k}^{\on{ext}}$ and $\phi_{y'_k}^{\on{ext}} + \phi_{z'_k}^{\on{ext}}$ are both extensions of a multilinear map from $A_{y_k} \cap A_{z_k} \cap A_{y'_k} \cap A_{z'_k}$ to $V_{y_k} \cap V_{z_k} \cap V_{y'_k} \cap V_{z'_k}$. By the inductive hypothesis, such extension is unique so $\phi_{y_k}^{\on{ext}} + \phi_{z_k}^{\on{ext}} = \phi_{y'_k}^{\on{ext}} + \phi_{z'_k}^{\on{ext}}$ implying that $\Phi$ is well-defined. A very similar argument shows that $\Phi$ is linear in direction $G_k$.
    
    \indent Finally, let us check that $\Phi$ is multilinear in directions $G_{[k-1]}$. To that end, let $x_{[k]}, y_{[k]}, z_{[k]} \in V$ coincide in all coordinates except some $d \in [k-1]$ and assume $x_d + y_d = z_d$. Then we have
    \[|(V_{x_{[k-1]}} \cap V_{y_{[k-1]}} \cap V_{z_{[k-1]}}) \setminus X| \leq 4\eta^{-1}p^{3r} \varepsilon |V_{x_{[k-1]}} \cap V_{y_{[k-1]}} \cap V_{z_{[k-1]}}|,\]
    so, provided $\varepsilon \leq \eta^2$ and $\varepsilon \leq  p^{-\blc (r+1)}$, we have $|V_{x_{[k-1]}} \cap V_{y_{[k-1]}} \cap V_{z_{[k-1]}} \cap X| \geq \frac{2}{3}|V_{x_{[k-1]}} \cap V_{y_{[k-1]}} \cap V_{z_{[k-1]}}|$ and thus we have $w_k, w'_k \in X \cap V_{x_{[k-1]}} \cap V_{y_{[k-1]}} \cap V_{z_{[k-1]}}$ such that $w_k + w'_k = x_k = y_k = z_k$. By definition of $\Phi$, we have
    \begin{align*}\Phi(x_{[k]}) + &\Phi(y_{[k]}) - \Phi(z_{[k]}) \\
    = &\phi_{w_k}^{\on{ext}}(x_{[k-1]}) + \phi_{w'_k}^{\on{ext}}(x_{[k-1]}) + \phi_{w_k}^{\on{ext}}(y_{[k-1]}) + \phi_{w'_k}^{\on{ext}}(y_{[k-1]}) - \phi_{w_k}^{\on{ext}}(z_{[k-1]}) + \phi_{w'_k}^{\on{ext}}(z_{[k-1]}) = 0.\end{align*}

    Set $\eta = \sqrt{\varepsilon}$ to finish the proof.
\end{proof}

At a couple of places in the paper, we are forced to consider maps that are almost multilinear. The next lemma treats the case of such maps.

\begin{lemma}\label{almost-multilinear-maps}
     Let $V$ be a multilinear variety of codimension at most $r$ inside $G_{[k]}$. Let $\varepsilon > 0$. Suppose that $A \subseteq V$ and let $\phi : A \to H$ be a map such that, for each direction $d \in [k]$, for all but at most $\varepsilon |G_{[k] \setminus \{d\}}|$ points $a_{[k]\setminus \{d\}} \in V^{([k]\setminus \{d\})}$, the map $x_d \mapsto \phi(a_{[k]\setminus \{d\}}, x_d)$ respects all but at most $\varepsilon|G_d|^2$ additive triples in $A_{a_{[k]\setminus \{d\}}}$. Then, provided $\varepsilon \leq p^{-\blc(r+1)}$ there exists a subset $A' \subseteq A$ of size $|A'| \geq (1 - O(\varepsilon^{\Omega(1)}))|V|$ on which $\phi$ is multilinear.
     
     \indent Moreover, the set $A'$ depends only on the set of the given respected directional additive triples and not on the map $\phi$.
\end{lemma}

\begin{proof}
    Apply Lemma~\ref{triviallinearextension} in each direction.
\end{proof}

% We say that a multilinear form $\rho : G_I \to \mathbb{F}_p$ is \textit{$\eta$-quasirandom} with respect to a mixed-linear map $(\beta_j : G_{I_i} \to \mathbb{F}_p)_{j \in [r]}$ if
% \[\Big| \exx_{x_{[k]}} \omega\Big(\rho(x_I) + \sum_{i \in [r]} \lambda_i \beta_i(x_{I_i})\Big)\Big| \leq \eta\]
% for all choices of scalars $\lambda_1, \dots, \lambda_r \in \mathbb{F}_p$.\\

% \begin{theorem}[General extension]
%     Suppose that $I_1, \dots, I_\ell = [k]$ is a partition. Suppose that $\rho_i : G_{I_i} \to \mathbb{F}_p$ are multilinear forms. Let $\beta_i : G_{I_i} \to \mathbb{F}_p$, for $i \in [r]$, be further multilinear forms, such that each $\rho_i$ is $\eta$-quasirandom w.r.t. $(\beta_j)_{j \in [r]}$. Write $V = \{x_{[k]} \in G_{[k]} : (\forall i \in [r])\, \beta_i(x_{I_i}) = 0\}$. Then, given any multilinear form $\phi : V \cap \{\alpha_1 \alpha_2\dots\alpha_\ell = 0\}$, point $x_{[k]}$ such that $\alpha_1(x_{I_1}) \alpha_2\dots\alpha_\ell(x_{I_\ell}) \not= 0$ and an element $h_0$, there exists a unique multilinear map $\phi^{\on{ext}} : V \to H$ extending $\phi$ and having $\phi^{\on{ext}}(x_{[k]}) = h_0$.
% \end{theorem}

% \vspace{\baselineskip}

We now turn to vanishing and extension results.

\noindent\textbf{Vanishing result.} We begin with a simpler result, which shows that if a MM-$r$-map vanishes on two quasirandom pieces, then it vanishes everywhere.

\begin{theorem}
    \label{vanishingmmmaps}
    Let $V \subseteq G_{[k]}$ be a multilinear variety defined a mixed-linear map $\beta : G_{[k]} \to \mathbb{F}_p^s$, let $\alpha_1 : G_{I} \to \mathbb{F}_p^{r_1}$ and $\alpha_2 : G_{I} \to \mathbb{F}_p^{r_2}$ be multilinear maps for some non-empty $I \subseteq [k]$. Suppose that $(\alpha_1, \alpha_2)$ is $\eta$-quasirandom with respect to $\beta$.
    
    \indent Let $Q_1 = \{\alpha_1 = 0\} \times G_{[k] \setminus I}$ and $Q_2 = \{\alpha_2 = 0\} \times G_{[k] \setminus I}$. Suppose that $\phi : V \to H$ is a multilinear map such that $\phi = 0$ on $(V \cap Q_1) \cup (V \cap Q_2)$. Provided $\eta \leq p^{-\blc(r_1 + r_2 + s)}$, then $\phi = 0$ on $V$.
\end{theorem}

\begin{proof}
    We shall prove that $\phi(x_{[k]}) = 0$ holds for a vast majority of points in $V$ and apply Lemma~\ref{trivialextensionsverydense} to finish the proof.
    
    \indent By Lemma~\ref{quasirandomrest}, we have a set $Y \subseteq V^{(I^c)}$ of size $|Y| \geq | V^{(I^c)}| - \sqrt{\eta}p^{r_1 + r_2 + s}|G_{I^c}|$ such that $(\alpha_1, \alpha_2)$ is $\sqrt[4]{\eta}$-quasirandom with respect to $\beta_{y_{I^c}}$ for all $y_{I^c} \in Y$. We show that $\phi_{y_{I^c}} = 0$ for each such $y_{I^c}$.
    
    \indent Let $y_{I^c} \in Y$ be arbitrary. By Lemma~\ref{connectedlayerslemma}, there exist points $a^{(1)}_I, \dots, a^{(r_1)}_I$ such that $a^{(i)}_I \in V_{y_{I^c}} \cap Q_2$ and $(\alpha_1)_j(a^{(i)}_I) = \id(i = j)$. Apply Theorem~\ref{qrExtension} $r_1$ times to deduce that $\phi_{y_{I^c}} = 0$ on $V_{y_{I^c}}$.
    
    \indent Since $|Y| \geq | V^{(I^c)}| - \sqrt{\eta}p^{r_1 + r_2 + s}|G_{I^c}|$, the map $\phi$ is multilinear and vanishes at all but at most $\sqrt{\eta}p^{r_1 + r_2 + s}|G_{[k]}|$ points, so by Lemma~\ref{trivialextensionsverydense} it must coincide with the 0 map.
\end{proof}

\noindent\textbf{Simultaneous extensions.} We now move on to the question of simultaneous extensions of MM-$r$-maps. Suppose now that $V$ and $U$ are multilinear varieties defined by mixed-linear maps $(\beta_i : G_{I_i} \to \mathbb{F}_p)_{i \in [r]}$ and $(\gamma_i : G_{J_i} \to \mathbb{F}_p)_{i \in [r]}$, respectively. Note that when $\alpha : G_I \to \mathbb{F}_p$, then 
\[\Big(V \cap (\{\alpha= 0 \} \times G_{I^c}) \Big)\cup\Big(V \cap (G_I \times U^{(I^c)})\Big) = V \cap\Big(\cap_{J_i \subseteq I^c} \{\alpha \gamma_i = 0\} \times G_{I^c \setminus J_i}\Big)\]
is a variety containing both $V \cap (\{\alpha= 0 \} \times G_{I^c})$ and $V \cap (G_I \times U^{(I^c)})$. Here, $\alpha \gamma_i$ is the multilinear form on $G_{I \cup J_i}$ given by $x_{I \cup J_i} \mapsto \alpha(x_I) \gamma_i(x_{J_i})$, for which it is necessary to have $I$ and $J_i$ disjoint.

\begin{proposition}[Simultaneous extension - single step]\label{simext-singlestepprop}
    Suppose that $I \subseteq [k]$ be a set. Suppose that $\alpha : G_{I} \to \mathbb{F}_p$ is a  multilinear form. Let $V$ and $U$ be multilinear varieties in $G_{[k]}$ defined by mixed-linear maps $(\beta_i : G_{I_i} \to \mathbb{F}_p)_{i \in [r]}$ and $(\gamma_i : G_{J_i} \to \mathbb{F}_p)_{i \in [r]}$, respectively. Suppose that, for all $d \in I^c$, the form $\alpha$ is $\eta$-quasirandom with respect to $(\beta_{z_d}, \gamma_{z_d}, \beta_{z'_d}, \gamma_{z'_d})$ for all but at most $\eta |G_d|^2$ pairs $(z_d, z'_d) \in G_d^2$.
    
    \indent Then, provided
    \[\eta \leq p^{-\blc (r + 1)},\]
    given any multilinear maps $\phi : V \cap (\{ \alpha = 0\} \times G_{I^c}) \to H$ and $\psi : U \to H$ agreeing on the intersection of the domains, there exists a unique multilinear map $\theta^{\on{ext}} : V \cap \Big(\cap_{J_i \subseteq I^c} \{\alpha \cdot \gamma_i = 0\} \times G_{I^c \setminus J_i}\Big) \to H$ such that
    \begin{itemize}
        \item we have $\phi = \theta^{\on{ext}}$ on $V \cap (\{\alpha = 0\} \times G_{I^c})$, and
        \item we have $\psi = \theta^{\on{ext}}$ on $V \cap U$.
    \end{itemize}
\end{proposition}

When the multilinear form $\beta_i$ does not depend on $G_d$, we simply interpret $(\beta_i)_{z_d}$ as $\beta_i$.

\begin{proof} The extension domain is $(V \cap U^{(I^c)} \times G_I) \cup (V \cap \{\alpha = 0\} \times G_{I^c})$, which is a union of two multilinear varieties we denote as $U_1 = (V \cap U^{(I^c)} \times G_I)$ and $U_2 = (V \cap \{\alpha = 0\} \times G_{I^c})$. Observe that a map $A : U_1 \cup U_2 \to H$ is multilinear if it is multilinear on each $U_i$ separately. Namely, if $x_{[k]},y_{[k]},z_{[k]} \in U_1 \cup U_2$ are three points differing only in the coordinate $d$, where $x_d + y_d = z_d$, then by the pigeonhole principle, two of them belong to the same $U_i$. But $U_i$ is a multilinear variety so, owing to the condition $x_d + y_d = z_d$, all three points belong to $U_i$. Multilinearity on $U_i$ gives $A(x_{[k]}) + A(y_{[k]}) = A(z_{[k]})$. Therefore, it suffices to extend $\phi$ to a multilinear map on $V \cap (U^{(I^c)} \times G_I)$, so that it coincides with $\phi$ on the intersection $V \cap (U^{(I^c)} \times G_I) \cap (\{\alpha = 0\} \times G_{I^c})$.

Hence, for each $y_{I^c} \in V^{(I^c)} \cap U^{(I^c)}$ we need to extend $\phi_{y_{I^c}}$ from $V_{y_{I^c}} \cap \{\alpha = 0\}$ to $V_{y_{I^c}}$.

\begin{claim}
    Provided $\eta \leq p^{-\blc r}$, we may find a set $Y \subseteq V^{(I^c)} \cap U^{(I^c)}$ of size $|Y| \geq |V^{(I^c)} \cap U^{(I^c)}| - p^{2r}\sqrt{\eta}|G_{I^c}|$ such that for all $y_{I^c} \in Y$
    \begin{itemize}
        \item[\textbf{(i)}] $\alpha$ is $\sqrt[4]{\eta}$-quasirandom with respect to $(\beta_{y_{I^c}}, \gamma_{y_{I^c}})$,
        \item[\textbf{(ii)}] the set $\{\alpha \not= 0\} \cap V_{y_{I^c}} \cap U_{y_{I^c}}$ is non-empty and connected.
    \end{itemize}
\end{claim}

\begin{proof}
    The claim follows form Lemmas~\ref{quasirandomrest} and~\ref{connectedlayerslemma}. 
\end{proof}

Let us extend $\phi_{y_{I^c}}$ from $V_{y_{I^c}} \cap \{ \alpha = 0\}$ to $V_{y_{I^c}}$ for each $y_{I^c} \in Y$, using Theorem~\ref{qrExtension}, by taking any $a_I \in V_{y_{I^c}} \cap U_{y_{I^c}} \cap\{ \alpha \not= 0\}$ and using value $\psi(a_I, y_{I^c})$. Call this extension $\phi^{\on{ext}}_{y_{I^c}} : V_{y_{I^c}}  \to H$. The extension is independent of the actual choice of $a_I$, as the set $\{\alpha \not=0\} \cap V_{y_{I^c}} \cap U_{y_{I^c}}$ is connected. We need to show multilinearity in directions $G_{I^c}$.

Let $d \in I^c$ and let $(x_{I}, y_{I^c \setminus \{d\}}, z_i)$ be points in the extended domain for $z_1, z_2, z_3 \in G_d$ with $z_1 + z_2 = z_3$. Consider the map $\tau: a_I \mapsto \phi^{\on{ext}}_{y_{I^c \setminus \{d\}}, z_1}(a_I) + \phi^{\on{ext}}_{y_{I^c \setminus \{d\}}, z_2}(a_I) - \phi^{\on{ext}}_{y_{I^c \setminus \{d\}}, z_3}(a_I)$. On the variety
\[V_{(y_{I^c \setminus \{d\}}, z_1)} \cap V_{(y_{I^c \setminus \{d\}}, z_2)} \cap V_{(y_{I^c \setminus \{d\}}, z_3)} \cap \{\alpha = 0\} = V_{(y_{I^c \setminus \{d\}}, z_1)} \cap V_{(y_{I^c \setminus \{d\}}, z_2)} \cap \{\alpha = 0\}\]
it coincides with $a_I \mapsto \phi_{y_{y_{I^c \setminus \{d\}}, z_1}}(a_I) + \phi_{y_{I^c \setminus \{d\}}, z_2}(a_I) - \phi_{y_{I^c \setminus \{d\}}, z_3}(a_I)$, which vanishes. Also, taking any
\[a_I \in V_{(y_{I^c \setminus \{d\}}, z_1)} \cap V_{(y_{I^c \setminus \{d\}}, z_2)} \cap U_{(y_{I^c \setminus \{d\}}, z_1)} \cap U_{(y_{I^c \setminus \{d\}}, z_2)} \cap\{\alpha\not=0\},\]
we have 
\begin{align*}\tau(a_I) = &\phi^{\on{ext}}_{y_{I^c \setminus \{d\}}, z_1}(a_I) + \phi^{\on{ext}}_{y_{I^c \setminus \{d\}}, z_2}(a_I) - \phi^{\on{ext}}_{y_{I^c \setminus \{d\}}, z_3}(a_I) \\
= &\psi(a_I, y_{I^c \setminus \{d\}}, z_1) + \psi(a_I, y_{I^c \setminus \{d\}}, z_2) - \psi(a_I, y_{I^c \setminus \{d\}}, z_3) = 0.\end{align*}
In other words, $\tau$ is extension of zero map on $V_{(y_{I^c \setminus \{d\}}, z_1)} \cap V_{(y_{I^c \setminus \{d\}}, z_2)} \cap \{\alpha = 0\}$ with a zero value. By assumption, for all but at most $\eta|G_d|^2$ pairs $(z_1, z_2) \in G_d \times G_d$ we have that $\alpha$ is $\eta$-quasirandom with respect to $(\beta_{z_1}, \beta_{z_2})$, so by Theorem~\ref{qrExtension}, such an extension is unique and thus $\tau = 0$.

To sum up, defining $\phi'$ on $(Y \times G_I) \cap V$ by merging together the extensions $\phi^{\on{ext}}_{y_{I^c}}$ on slices, we get a map is multilinear in directions $G_I$ and for each $d \in I^c$ the map $\phi'$ respects all but at most $\eta |G_d||G_{[k]}|$ additive triples in direction $G_d$. By Lemma~\ref{almost-multilinear-maps} there exists a subset $Z \subseteq V \cap (U^{(I^c)} \times G_I)$ of size $|Z| \geq |V \cap (U^{(I^c)} \times G_I)| - p^{O(r)} \eta^{\Omega(1)}|G_{[k]}|$ on which $\phi'$ is multilinear. Apply Lemma~\ref{trivialextensionsverydense} to find a multilinear map $\theta^{\on{ext}} : V \cap (U^{(I^c)} \times G_I) \to H$ coinciding with $\phi'$ on a set $Z' \subseteq Z$ of size $|Z'| \geq |V \cap (U^{(I^c)} \times G_I)| - p^{O(r)} \eta^{\Omega(1)}|G_{[k]}|$.

It remains to show that $\theta^{\on{ext}}$ coincides with $\phi$ and $\psi$ on suitable varieties. Note firstly that $\theta^{\on{ext}} = \phi'$ on $Z'$, and thus $\theta^{\on{ext}}(x_{[k]}) = \phi(x_{[k]})$ holds for all but at most $p^{O(r)} \eta^{\Omega(1)}|G_{[k]}|$ points $x_{[k]}$ in $V \cap (U^{(I^c)} \times G_I) \cap\{\alpha = 0\}$. By uniqueness part of Lemma~\ref{trivialextensionsverydense}, we have $\theta^{\on{ext}} = \phi$ on the whole variety $V \cap (U^{(I^c)} \times G_I) \cap\{\alpha = 0\}$.

\indent Similarly, we have  $\theta^{\on{ext}} = \phi'$ on $Z'$, so  $\theta^{\on{ext}} = \psi$ on $Z' \cap U$ and the conclusion follows from another application of Lemma~\ref{trivialextensionsverydense}.
\end{proof}

We may now deduce the main result of this subsection.

\begin{theorem}[Simultaneous extension]\label{sim-ext-final-ver}
    Suppose that $I \subseteq [k]$ be a set. Suppose that $\alpha : G_{I} \to \mathbb{F}_p^{r_1}$ and $\beta : G_I \to \mathbb{F}_p^{r_2}$ are multilinear maps. Let $V$ and $U$ be multilinear varieties in $G_{[k]}$ of codimension at most $r$ such that, for all $d \in I^c$, $(\alpha, \beta)$ is $\eta$-quasirandom with respect to $(U \cap V)_{z_d} \cap (U \cap V)_{z'_d}$ for all but at most $\eta |G_d|^2$ pairs $(z_d, z'_d) \in G_d^2$.
    
    \indent Then, provided
    \[\eta \leq p^{-\blc (r + r_1 + r_2)},\]
    given any multilinear maps $\phi : V \cap (\{ \alpha = 0\} \times G_{I^c}) \to H$ and $\psi : U \cap (\{\beta = 0\} \times G_{I^c}) \to H$ agreeing on the intersection of the domains, there exist unique multilinear maps $\phi^{\on{ext}} : V \cap (U^{(I^c)} \times G_I) \to H$ and $\psi^{\on{ext}} : U \cap (V^{(I^c)} \times G_I) \to H$ such that
    \begin{itemize}
        \item[\textbf{(i)}] we have $\phi = \phi^{\on{ext}}$ on $V \cap (\{\alpha = 0\} \times G_{I^c}) \cap (U^{(I^c)} \times G_I)$ and $\psi = \phi^{\on{ext}}$ on $U \cap V \cap (\{\beta = 0\} \times G_{I^c})$, and
        \item[\textbf{(ii)}] we have $\psi = \psi^{\on{ext}}$ on $U \cap (\{\beta = 0\} \times G_{I^c}) \cap (V^{(I^c)} \times G_I)$ and $\phi = \psi^{\on{ext}}$ on $U \cap V \cap (\{\alpha = 0\} \times G_{I^c})$.
    \end{itemize}

    Moreover, whenever $\phi^{\on{ext}} : V \cap (U^{(I^c)} \times G_I) \to H$ and $\psi^{\on{ext}} : U \cap (V^{(I^c)} \times G_I) \to H$ satisfy \textbf{(i)} and \textbf{(ii)}, then we have
    \[\phi^{\on{ext}} = \psi^{\on{ext}}\text{ on }V \cap U.\]
\end{theorem}

\begin{proof}
    We may restrict $\phi$ to $V \cap (\{\alpha = 0\} \times G_{I^c}) \cap (U^{(I^c)} \times G_I)$ and $\psi$ to $U \cap (\{\beta = 0\} \times G_{I^c}) \cap (V^{(I^c)} \times G_I)$. Apply Proposition~\ref{simext-singlestepprop} iteratively to each component of $\alpha$ to get a multilinear map $\phi^{\on{ext}} : V \cap (U^{(I^c)} \times G_I) \to H$ such that 
    \begin{itemize}
        \item $\phi = \phi^{\on{ext}}$ on $V \cap (\{\alpha = 0\} \times G_{I^c}) \cap (U^{(I^c)} \times G_I)$, and
        \item $\psi = \phi^{\on{ext}}$ on $V \cap U \cap (\{\beta = 0\} \times G_{I^c})$.
    \end{itemize}

    Apply Proposition~\ref{simext-singlestepprop} another time, to maps $ \phi^{\on{ext}}$ and $\psi$, to get a multilinear map $\psi^{\on{ext}} : U \cap (V^{(I^c)} \times G_I) \to H$ such that
     \begin{itemize}
        \item $\phi^{\on{ext}} = \psi^{\on{ext}}$ on $V \cap U$, and
        \item $\psi = \psi^{\on{ext}}$ on $U \cap (\{\beta = 0\} \times G_{I^c}) \cap (V^{(I^c)} \times G_I)$.
    \end{itemize}
    The last point follows from the uniqueness of extensions.
\end{proof}

The following corollary is the specialization to the case when the varieties $U$ and $V$ coincide. We record it separately as it is the most common application of the theorem above.

\begin{corollary}\label{simextcorsinglevar}
    Suppose that $I \subseteq [k]$ be a set. Suppose that $\alpha : G_{I} \to \mathbb{F}_p^{r_1}$ and $\beta : G_I \to \mathbb{F}_p^{r_2}$ be multilinear maps. Let $V$ be a multilinear variety in $G_{[k]}$ of codimension at most $r$ such that for each $d \in I^c$, the map $(\alpha, \beta)$ is $\eta$-quasirandom with respect to $V_{z_d} \cap V_{z'_d}$ for all but at most $\eta |G_d|^2$ pairs $(z_d, z'_d) \in G_d^2$.
    
    \indent Then, provided
    \[\eta \leq p^{-\blc (r + r_1 + r_2)},\]
    given any multilinear maps $\phi : V \cap (\{ \alpha = 0\} \times G_{I^c}) \to H$ and $\psi : V \cap (\{\beta = 0\} \times G_{I^c}) \to H$ agreeing on the intersection of the domains, there exist a unique multilinear map $\theta^{\on{ext}} : V \to H$ such that
    \begin{itemize}
        \item[\textbf{(i)}] we have $\phi = \theta^{\on{ext}}$ on $V \cap (\{\alpha = 0\} \times G_{I^c})$, and
        \item[\textbf{(ii)}] we have $\psi = \theta^{\on{ext}}$ on $V \cap (\{\beta = 0\} \times G_{I^c})$.
    \end{itemize}
\end{corollary}

\subsection{Vanishing and common extensions in linear systems of varieties}

Let $\alpha : G_{[0,k]} \to \mathbb{F}_p^r$ be a mixed-linear map whose components depend linearly on $G_0$. A \textit{linear system of varieties} $(V_a)_{a \in G_0}$ for $\alpha$ is a collection of multilinear varieties defined as $V_a = \{x_{[k]} \in G_{[k]} : \alpha(a, x_{[k]}) = 0\}$, or in shorter notation $V_a = \{\alpha_a = 0\}$. Note that this is equivalent to defining $V \subseteq G_{[0,k]}$ as the variety $\{\alpha = 0\}$ and then taking slices $V_a$. However, we take the perspective of having a system of varieties as it emphasizes the role of $G_0$.

\indent In this subsection, we consider the questions of vanishing and simultaneous extensions of MM-$r$-maps in the context of linear systems of varieties. Ultimately, such results will be required in the fourth phase of our main proof. For example, given a multilinear map $\phi : G_{[k]} \to H$, what can we say about it if we know that $\phi|_{V_a} = 0$ holds for many indexing elements $a$? If $V$ is quasirandom, then we would expect to have $\phi$ vanishing on the whole space $G_{[k]}$. However, given a general linear system of varieties, that might not be the case. Remarkably, after a passing to a bounded codimension multilinear variety $Z$, depending only on $V$ and not on $\phi$, we have that $\phi$ vanishes on $Z$.

\indent The next theorem, in analogy with the sheaf theory, may be thought of as the \textit{weak locality} property for multilinear varieties. Similarly, in the other main result of this subsection, Theorem~\ref{linsystextns}, we show that multilinear varieties have an analogue of the second important sheaf property, namely they posses \emph{weak gluing} property.

\begin{theorem}[Vanishing in linear systems of varieties -- weak locality]\label{lin-sys-vanishing}
    Let $(V_a)_{a \in G_0}$ be a linear system of multilinear varieties of codimension $r$. Let $s \in \mathbb{N}$ and $c > 0$. Then there exists a multilinear variety $Z \subseteq G_{[k]}$ of codimension $( r + s + \log c^{-1})^{O(1)}$ such that, whenever $\phi : W \to H$ is a multilinear map on a multilinear variety $W \subseteq G_{[k]}$ of codimension $s$ such that $\phi = 0$ on $W \cap V_a$ for at least $c |G_0|$ choices of $a \in G_0$, then $\phi  = 0$ on $W \cap Z$.
\end{theorem}

 Slightly more generally, for small $\ell$, the theorem applies when $\phi = 0$ on $W \cap V_{a_1} \cap \dots \cap V_{a_\ell}$ for a dense collection of $\ell$-tuples $a_{[\ell]}$. Namely, we may consider a variety $\tilde{V} \subseteq G_0^\ell \times G_{[k]}$, given by a mixed-linear map $\tilde{\alpha} : G_0^\ell \times G_{[k]} \to (\mathbb{F}_p^r)^\ell$, with $\tilde{\alpha}_{i,j}(a_{[\ell]}, x_{[k]}) = \alpha_{i}(a_j, x_{[k]})$.
 
\indent In the proofs in this section, we will use induction over collections of sets of coordinates that components of the underlying mixed-linear map $\alpha$ depend on. We need the concept of a \textit{down-set} $\mathcal{P}$ inside the power-set $\mathcal{P}[d] = \{I : I \subseteq [d]\}$, meaning that $\mathcal{P}$ is closed under taking subsets.

\begin{proof}
    Let $\alpha : G_{[0,k]} \to \mathbb{F}_p^r$ be a mixed-linear map defining $V$. By downwards induction on $d \in [0, k]$, then on down-sets $\mathcal{P} \subseteq \mathcal{P}[d]$, we show that there exist a positive integer $\ell = O(1)$, a parameter $c' \geq (c/2)^{O(1)}$ and a variety $Z\subseteq G_{[k]}$ of codimension $( r + s + \log c^{-1})^{O(1)}$ such that whenever $\phi : W \to H$ has the given property, then $\phi = 0$ on $W \cap Z \cap V'_{a_1} \cap \dots \cap V'_{a_\ell}$ for at least $c'|G_0|^\ell$ choices of $(a_1, \dots, a_\ell) \in G^\ell_0$, where $V'_a$ is the zero set of those multilinear forms $\alpha_i$ depending on $G_{\{0\} \cup I}$ for some $I \in \mathcal{P}$. Note that $Z$ can have any dependencies on $G_{[k]}$, but that it does not depend on $G_0$ instead.

    The base case when $d = k$ and $\mathcal{P} = \mathcal{P}[k]$ is trivial. Suppose that the claim holds for some down-set $\mathcal{P} \cup \{I\}$, in which $I$ is a maximal set, with $d \in I$. Without loss of generality, suppose that $\alpha_1, \dots, \alpha_{r'}$ are components of $\alpha$ depending on $G_{\{0\} \cup I}$ and write $\alpha' = (\alpha_1, \dots, \alpha_{r'})$. Write $V''_a$ for the the zero set of those multilinear forms $\alpha_i$ depending on $G_{\{0\} \cup J}$ for some $J \in \mathcal{P}$.

    Apply the induction hypothesis for $\mathcal{P} \cup \{I\}$ to obtain the relevant $\ell, c'$ and $Z$ such that whenever $\phi : W \to H$ has the given property, then $\phi = 0$ on $W \cap Z \cap V''_{a_1} \cap \{\alpha'_{a_1} = 0\} \cap \dots \cap V''_{a_\ell} \cap \{\alpha'_{a_\ell} = 0\}$ for at least $c'|G_0|^\ell$ choices of $(a_1, \dots, a_\ell) \in G^\ell_0$. The goal of the inductive step is to remove $\alpha'$ from this conclusion. Let $s' \leq ( r + s + \log c^{-1})^{O(1)}$ be the codimension of $Z$.

    Let $\eta > 0$ be a parameter to be chosen later. Take a maximal linearly independent set $\lambda_1, \dots, \lambda_m \in \mathbb{F}_p^{r'}$ of linear combinations such that $\lambda_i \cdot \alpha'$ fails to be $\eta$-quasirandom with respect to  $\mathcal{P}$-dependent forms in $\alpha$ and forms defining $W \cap Z$. Extend these linear combinations to a basis $\mu_1, \dots, \mu_{r' - m}$ of $\mathbb{F}_p^{r'}$. Write $\sigma_i = \lambda_i \cdot \alpha'$ and $\rho_i = \mu_i \cdot \alpha'$. With this notation, $ \{\alpha'_{a} = 0\} = \{\sigma_{a} = 0\} \cap  \{\rho_{a} = 0\}$.

    In particular, we have  $\phi = 0$ on $W \cap Z \cap \Big(\cap_{i \in [\ell]} V''_{a_i}\Big) \cap \Big(\cap_{i \in [\ell]} \{\sigma = 0\}_{a_i}\Big)  \cap \Big(\cap_{i \in [\ell]} \{\rho = 0\}_{a_i}\Big) $ and on $W \cap Z \cap \Big(\cap_{i \in [\ell]} V''_{b_i}\Big) \cap \Big(\cap_{i \in [\ell]} \{\sigma = 0\}_{b_i}\Big)  \cap \Big(\cap_{i \in [\ell]} \{\rho = 0\}_{b_i}\Big)$ for at least ${c'}^2|G_0|^{2\ell}$ choices of $(a_{[\ell]}, b_{[\ell]})$. By Lemma~\ref{quasirandomrest}, map $(\rho_{a_1}, \dots, \rho_{a_\ell}, \rho_{b_1}, \dots, \rho_{b_\ell})$ is $\sqrt[4]{\eta}$-quasirandom with respect to 
    \[W \cap Z \cap \Big(\cap_{i \in [\ell]} V''_{a_i}\Big) \cap \Big(\cap_{i \in [\ell]} \{\sigma = 0\}_{a_i}\Big) \cap \Big(\cap_{i \in [\ell]} V''_{b_i}\Big) \cap \Big(\cap_{i \in [\ell]} \{\sigma = 0\}_{b_i}\Big) \]
    for all but at most $p^{O_k(r + s + s')}\eta^{\Omega_k(1)}|G_0|^{2\ell}$ choices of $(a_{[\ell]}, b_{[\ell]})$. Provided $\eta \leq p^{-\blc(r + s + s')} {c'}^{\blc}$, for at least $\frac{1}{2}{c'}^2|G_0|^{2\ell}$ choices of $(a_{[\ell]}, b_{[\ell]})$ we may apply
    Corollary~\ref{vanishingmmmaps} to deduce that $\phi = 0$ on 
    \[W \cap Z \cap \Big(\cap_{i \in [\ell]} \{\sigma = 0\}_{a_i}\Big) \cap  \Big(\cap_{i \in [\ell]} \{\sigma = 0\}_{b_i}\Big) \cap \Big(\cap_{i \in [\ell]} V''_{a_i}\Big) \cap \Big(\cap_{i \in [\ell]} V''_{b_i}\Big).\]
    
    We are done after choosing a suitable $\eta \geq p^{-O(r + s + s')} {c'}^{O(1)}$ and using Theorem~\ref{arankprank} to replace the $ \{\sigma = 0\}_{a_i}$ varieties with a variety independent of $G_0$.
\end{proof}

% By a signature, we mean a tuple $\sigma = (I_1, \dots, I_\ell)$ of disjoint sets inside $[k]$. The sets are always ordered by their minimal elements, so $\min I_1 < \min I_2 < \dots$. We say that a signature $\sigma'$ refines $\sigma$, if it has all its members and has more members.

% \begin{proof}
%     Let $\mathcal{S}$ be the collection of all signatures inside $[k]$. At each step, we maintain $\mathcal{S}$, closed under taking refinements, an integer $m = O(1)$, a map $\on{Ind} : \{ (\sigma, i) : \sigma \in \mathcal{S}, i\text{ index of member}\} \to [m]$, assigning an index to each member of $\sigma$, a dense collection of parameters $X \subseteq G_0^m$ and a multilinear variety $Z \subseteq G_{[k]}$, such that $\phi = 0$ on
%     \[W \cap Z \cap \{x_{[k]} : (\forall \sigma = (I_1, \dots, I_\ell) \in \mathcal{S}) (\forall j_1, \dots, j_\ell) \,\,\prod_{i \in [\ell]}\alpha_{j_i}(x_{\on{Ind}(\sigma, i)}; y_{I_i})\}\]
%     for all $(x_1, \dots, x_m) \in X$, where $j_i$ ranges over indices of $\alpha_\kappa$ that depend on $G_0 \times G_{I_i}$.\\

%     The base case is trivial, by assumptions. Suppose now that $\sigma = (I_1, \dots, I_\ell)$ is a minimal signature in $\mathcal{S}$. Our goal is to remove it from $\mathcal{S}$.
% \end{proof}

We now turn to the problem of extensions in linear systems of varieties, which in this setting becomes the following. Suppose that for a linear system of varieties $(V_a)_{a \in G_0}$, we have a multilinear map $\phi_a : V_a \to H$ for many indexing elements  $a \in G_0$, and that these maps frequently agree on the intersection of their domains, namely $\phi_a = \phi_b$ on $V_a \cap V_b$. The natural guess would be that maps $\phi_a$ arise as restrictions of a single global multilinear map to varieties $V_a$. The following theorem says that, after passing to appropriate subvarieties, that is indeed the case.

\begin{theorem}[Common extensions in linear systems of varieties -- weak gluing]\label{linsystextns}
    Given a linear system of varieties $(V_a)_{a \in G_0}$ of codimension $r$, an integer $s \in \mathbb{N}$ and $c > 0$, there exist a multilinear variety $Z$ of codimension $( r + s + \log c^{-1})^{O(1)}$  and a further linear system of varieties $(\tilde{V}_a)_{a \in G_0}$, with $\tilde{V}_a \subseteq V_a$, of codimension $( r + s + \log c^{-1})^{O(1)}$, such that, whenever we are given:
    \begin{itemize}
        \item a set $X \subseteq G_0$,
        \item a multilinear variety $W$ of codimension $s$,
        \item a multilinear map $\phi_a : W \cap V_a \to H$ for each $a \in X$, such that
        \begin{equation}\label{agreeing-triangles-cond}\phi_a = \phi_b\text{ on }W \cap V_a \cap V_b\end{equation}
        for $c|G_0|^2$ pairs in $X$,
    \end{itemize}
    there exists a map $\Psi : W \cap Z \to H$ such that 
    \[\Psi = \phi_a\text{ on }W \cap Z \cap \tilde{V}_a\]
    holds for at least $(c/2)^{O(1)}|G_0|$ elements $a \in G_0$.
\end{theorem}

Similarly to the vanishing theorem for linear systems, we have to carefully induct on the support set of multilinear forms defining the linear system. To that end, we formulate the following proposition, which captures the inductive proof. The notation $\mathcal{P} + \{0\}$ in the statement denotes the family $\{I \cup \{0\} : I \in \mathcal{P}\}$.

\begin{proposition}[Common extensions in linear systems -- inductive statement]\label{linsystextns-induction}
    Let $1 \leq d \leq k$ and let $\mathcal{P} \subseteq \mathcal{P}[d] \setminus \{\emptyset\}$ be a down-set (i.e. closed under taking non-empty subsets). Then Theorem~\ref{linsystextns} holds for $(\mathcal{P} + \{0\})$-supported forms defining the linear system $(V_x)_{x \in G_0}$, with the linear system of varieties $(\tilde{V}_a)_{a \in G_0}$  in the conclusion being $(\mathcal{P} + \{0\})$-supported.
\end{proposition}

Note that the case when $d = k, \mathcal{P} = \mathcal{P}[k]$ proves Theorem~\ref{linsystextns}.

\begin{proof}[Proof of Proposition~\ref{linsystextns-induction}] We prove the proposition by induction on $d$ and $|\mathcal{P}|$. The base case when $\mathcal{P} = \emptyset$ is trivial. Let now $d$ and $\mathcal{P}$ be given, and let the claim hold for smaller values of $d$ and smaller sizes of $\mathcal{P}$. Let $I \in \mathcal{P}$ be a maximal set containing $d$. If there is no such set, then we may decrease $d$ to $d-1$, so induction hypothesis applies. Also, we remark that the induction hypothesis applies to $\mathcal{P} \setminus \{I\}$.

Firstly, we strengthen the assumption by showing that the graph of pairs $(x,y)$ satisfying~\eqref{agreeing-triangles-cond} contains large cliques, upon passing to a suitable subvariety of $W$. Namely, let $\Gamma$ be a graph on the vertex set $X$, whose edges are given by the condition~\eqref{agreeing-triangles-cond}. It has density $c$. By Lemma~\ref{gowerspathssingle}, there exists a set $X' \subseteq X$ of size $c_1|G_0|$ for some $c_1 \geq (c/2)^{O(1)}$ such that any two $x, y \in X'$ are joined by $(c/2)^{O(1)}|G_0|^5$ paths of length 6 in $\Gamma$. But, if $x, z_1, \dots, z_5, y$ is such a path, then
\[\phi_x = \phi_{z_1} = \phi_{z_2} = \dots = \phi_y\text{ on }W \cap V_x \cap V_{z_1} \cap \dots \cap V_{z_5} \cap V_y.\]
By Theorem~\ref{lin-sys-vanishing} for the given system and codimension parameter $s + 2r$, we may pass to a multilinear variety $W' \subseteq W$ of codimension $r' \leq (r + s + \log c^{-1})^{O(1)}$, independent of $x$ and $y$, such that $\phi_x - \phi_y$ vanishes on $W' \cap V_x \cap V_y$. Hence, for all $x, y \in X'$, we have $\phi_x = \phi_y$ on $W' \cap V_x \cap V_y$.

Write $V_x = V^{\on{low}}_x \cap (\{\alpha_x = 0\} \times G_{I^c})$ for a $(\mathcal{P} \setminus \{I\} + 0)$-supported linear system $(V^{\on{low}}_x)_{x \in G_0}$ and a multilinear map $\alpha : G_{\{0\} \cup I} \to \mathbb{F}_p^r$. Let $\eta > 0$ be a parameter to be chosen later and let $\lambda_1, \dots, \lambda_m \in \mathbb{F}_p^r$ be a maximal independent set of linear combinations such that $\on{bias}  (\lambda_{i} \cdot \alpha) \geq \eta$. We may extend these linear combinations with further $\mu_1, \dots, \mu_{r - m}$ to a basis $\mathbb{F}_p^r$. Write $\rho_i = \mu_i \cdot \alpha$ for the quasirandom linear combinations and $\beta_i = \lambda_i \cdot \alpha$ for the biased ones. In particular, we may write
\[V_x = V^{\on{low}}_x \cap (\{\beta_x = 0\} \times G_{I^c}) \cap (\{\rho_x = 0\} \times G_{I^c}).\]

Let $Q$ be the set of pairs $(x,y) \subseteq X' \times X'$ such that for all $\ell \notin I$, the map $(\rho_x, \rho_y)$ is $\eta^{2^{-k-2}}$-quasirandom with respect to $(V^{\on{low}}_x \cap  V^{\on{low}}_y \cap W')_{z_\ell}$, $(V^{\on{low}}_x \cap  V^{\on{low}}_y \cap W')_{z'_\ell}$ and $(\beta_x, \beta_y)$ for all but at most $\eta^{2^{-k-2}}|G_\ell|^2$ pairs $(z_\ell,z'_\ell) \in G_\ell^2$.

\begin{claim}
    We have $|Q| \geq (1 - p^{16r+4r'}\eta^{2^{-k-1}}) |X'|^2$.
\end{claim}

\begin{proof}
    Let $F$ be the set of all pairs $(x,y) \in X' \times X' \setminus Q$. By averaging, there exists $\ell \notin I$, a non-trivial linear combination of all involved forms 
    \begin{align*}A_{x,y,z_\ell, z'_\ell}(a_{[k] \setminus \{\ell\}}) = \lambda \cdot \alpha_x(a_I) + &\mu \cdot \alpha_y(a_I) + \tau(x, a_{[k] \setminus \{\ell\}}, z_\ell)\\
    +&\tau(x, a_{[k] \setminus \{\ell\}}, z'_\ell) + \tau'(y, a_{[k] \setminus \{\ell\}}, z_\ell) + \tau'(y, a_{[k] \setminus \{\ell\}}, z'_\ell),\end{align*}
    in which at least one of $\rho_x$ and $\rho_y$ appears, such that for at least $p^{-8r-2r'}|F|$ pairs $(x,y)$ we have at least $\eta^{2^{-k-2}}|G_\ell|^2$ pairs $(z_\ell,z'_\ell) \in G_\ell^2$ with
    \[\Big|\exx_{a_{[k] \setminus \{\ell\}}}\omega(A_{x,y,z_\ell, z'_\ell}(a_{[k] \setminus \{\ell\}}))\Big| \geq \eta^{2^{-k-2}}.\]
    Here, $\tau$ and $\tau'$ are multiaffine forms arising from forms defining varieites $(V^{\on{low}}_x)_{z_\ell},$ $(V^{\on{low}}_y)_{z_\ell},$ $(V^{\on{low}}_x)_{z'_\ell},$ $(V^{\on{low}}_y)_{z'_\ell},$ $(W')_{z_\ell}$ and $(W')_{z'_\ell}$. In particular, multilinear parts of $\tau$ and $\tau'$ either do not depend on $G_0$ or are $(\mathcal{P} \setminus \{I\} + \{0\})$-supported. Furthermore, since $\rho_x$ or $\rho_y$ appears in the linear combination, we have $\lambda$ or $\mu$ outside $\langle \lambda_1, \dots, \lambda_m \rangle$, say $\lambda$. Hence
    \[p^{-16r-4r'}\eta^{2^{-k-1}} \frac{|F|}{|G_0|^2} \leq \exx_{x,y \in G_0, z_\ell, z'_\ell \in G_\ell} \Big|\exx_{a_{[k] \setminus \{\ell\}}}\omega(A_{x,y,z_\ell, z'_\ell}(a_{[k] \setminus \{\ell\}}))\Big|^2 \leq \on{bias} (\lambda \cdot \alpha)^{2^{-k}},\]
    by the Gowers-Cauchy-Schwarz inequality, proving the claim.
\end{proof} 

For such a pair $(x,y)$, Corollary~\ref{simextcorsinglevar} applies to give a multilinear map $\theta_{x,y} : W' \cap V^{\on{low}}_x \cap (\{\beta_x = 0\} \times G_{I^c}) \cap V^{\on{low}}_y \cap (\{\beta_y = 0\} \times G_{I^c}) \to H$ such that $\phi_x = \theta_{x,y}$ on $W' \cap V^{\on{low}}_x \cap (\{\beta_x = 0\} \times G_{I^c}) \cap V^{\on{low}}_y \cap (\{\beta_y = 0\} \times G_{I^c}) \cap (\{\rho_x = 0\} \times G_{I^c})$ and similarly for $\phi_y$. Then
\begin{equation}\phi_x = \theta_{x,y}\text{ on }W' \cap V_x \cap V^{\on{low}}_y  \cap (\{\beta_y = 0\} \times G_{I^c}) \text{ and }\phi_y = \theta_{x,y}\text{ on }W'\cap V_y \cap V^{\on{low}}_x  \cap (\{\beta_x = 0\} \times G_{I^c}) .\label{commonextgoodpair}\end{equation}

Provided $\eta \leq p^{-\blc (r + r')}$, we have $|Q| \geq (1-\eta^{\Omega(1)})|X'|^2$. Moreover, 
\begin{align}(\rho_x, \rho_y, \rho_z)&\text{ is }\eta^{\Omega(1)}\text{-quasirandom with respect to }\nonumber\\
&W' \cap V^{\on{low}}_x \cap V^{\on{low}}_y \cap V^{\on{low}}_z\cap (\{\beta_x = 0\} \times G_{I^c}) \cap  (\{\beta_y = 0\} \times G_{I^c}) \cap (\{\beta_z = 0\} \times G_{I^c})\label{triangleQR}\end{align}
for all but at most $p^{O(r + r')}\eta^{\Omega(1)}|G_0|^3$ triples $(x,y,z) \in G_0^3$. Thus, provided $\eta \leq p^{-\blc(r + r')} c^{\blc}$, by averaging, we get subset $X'' \subseteq X'$ of size $|X''| \geq (1-\eta^{\Omega(1)})|X'|$ such that for each $x \in X''$ there are at least $(1-\eta^{\Omega(1)})|X'|^2$ choices of $(y,z) \in X$ such that all three pairs $(x,y), (x,z)$ and $(y,z)$ belong to $Q$, so~\eqref{commonextgoodpair} holds for each of them, and condition~\eqref{triangleQR} holds for $(x,y,z)$. Let $T_x$ be the collection of such $(y,z)$. Note that, when $(y,z) \in T_x$, then $\theta_{x,y}$ and $\theta_{x,z}$ are both simultaneous extensions of maps 
\begin{align*}&\phi_x : W' \cap V^{\on{low}}_x \cap (\{\beta_x = 0\} \times G_{I^c}) \cap (\{\rho_x = 0\} \times G_{I^c}) \to H\\
\text{ and }&\phi_{y} = \phi_z : W'\cap V^{\on{low}}_y \cap V^{\on{low}}_z \cap  (\{\beta_y = 0\} \times G_{I^c}) \cap (\{\beta_z = 0\} \times G_{I^c})\\
&\hspace{8cm}\cap (\{\rho_y = 0\} \times G_{I^c})\cap (\{\rho_z = 0\} \times G_{I^c}) \to H,\end{align*}
as both contain the extension domain $W' \cap V^{\on{low}}_x \cap V^{\on{low}}_y \cap V^{\on{low}}_z\cap (\{\beta_x = 0\} \times G_{I^c}) \cap  (\{\beta_y = 0\} \times G_{I^c}) \cap (\{\beta_z = 0\} \times G_{I^c})$. By Theorem~\ref{vanishingmmmaps}, provided $\eta \leq p^{-\blc(r + r')} c^{\blc}$, by condition~\eqref{triangleQR} we have $\theta_{x,y} = \theta_{x,z}$ on that variety.

Let $c_2 \geq (c/2)^{O(1)}$ be the density in the conclusion of Proposition~\ref{linsystextns-induction} (i.e.\ in the conclusion of Theorem~\ref{linsystextns}), when the induction hypothesis is applied for the collection of sets $\mathcal{P} \setminus \{I\}$ and the density parameter $c_1^2/100$ in assumption~\eqref{agreeing-triangles-cond}.

\indent We may choose $\eta \geq p^{-O(r + r')}c^{O(1)} \geq \exp(-(r + s + \log c^{-1})^{O(1)})$ so that all required bounds on $\eta$ hold and so that bounds $|X''| \geq (1-\eta^{\Omega(1)})|X'|$ and $|T_x| \geq (1-\eta^{\Omega(1)})|X'|^2$ become $|X''| \geq (1-c_1^2c^2_2/100)|X'|$ and $|T_x| \geq (1-c_1^2c^2_2/100)|X'|^2$. By Theorem~\ref{arankprank}, we may find a multilinear variety $Z \subseteq G_{[k]}$ of codimension $s' \leq (r + \log \eta^{-1})^{O(1)}$ such that $Z \subseteq (\{\beta_a = 0\} \times G_{I^c})$ for all $a$. Hence, for the given $x \in X''$ we get at least $\frac{c^2_1}{2}|G_0|^2$ pairs $(y,z)$ such that $\theta_{x,y} = \theta_{x,z}$ on $(W' \cap Z \cap V^{\on{low}}_x) \cap V^{\on{low}}_y \cap V^{\on{low}}_z$.

As linear forms defining $V^{\on{low}}_y$ are $\mathcal{P} \setminus \{I\}  + \{0\}$-supported, the induction hypothesis applies to the linear system $(V^{\on{low}}_y)_{y \in G_0}$, codimension parameter $r' + s' + r$ and density parameter $\frac{c^2_1}{2}$ to give a multilinear variety $Z^{\on{IH}}$ of codimension $s'' \leq (r + r' + s' + \log c^{-1})^{O(1)}$ and a further linear system $(V^{\on{IH}}_y)_{y \in G_0}$ of codimension $r'' \leq (r + r' + s' + \log c^{-1})^{O(1)}$ with the described property. Hence, for each $x \in X''$, we get a multilinear map $\Psi_x : Z^{\on{IH}} \cap W' \cap Z \cap V^{\on{low}}_x \to H$ such that $\Psi_x = \theta_{x,y}$ on $Z^{\on{IH}}\cap W' \cap Z \cap V^{\on{low}}_x \cap V^{\on{IH}}_y$ for $c_2|G_0|$ elements $y \in G_0$ (recall that $c_2$ was chosen with the application of induction hypothesis in mind). In particular, we have $\Psi_x = \theta_{x,y} = \phi_x$ on $Z^{\on{IH}}\cap W' \cap Z  \cap V_x \cap V^{\on{IH}}_y$ for many $y$. Theorem~\ref{lin-sys-vanishing} applies another time to find a multilinear variety $\tilde{Z} \subseteq Z^{\on{IH}} \cap W' \cap Z$ such that $\Psi_x- \phi_x$ vanishes on $\tilde{Z} \cap V_x$.

Finally, we get that $\Psi_x = \theta_{x,y}$ on $Z^{\on{IH}}\cap W' \cap Z \cap V^{\on{low}}_x \cap V^{\on{IH}}_y$ and $ \Psi_{x'} = \theta_{x', y}$ on $Z^{\on{IH}}\cap W' \cap Z \cap V^{\on{low}}_{x'} \cap V^{\on{IH}}_y$ hold for at least $\frac{c_1^2c_2^2}{4}|G_0|^3$ triples $(x, x', y)$ in $X''$. Due to choice of $\eta$, for at least half of such triples, we additionally have $(x,x') \in T_y$, so
\[\theta_{x,y} = \theta_{x',y}\text{ on }W' \cap V^{\on{low}}_x \cap V^{\on{low}}_y \cap V^{\on{low}}_{x'}\cap Z.\]

Hence, for $(c/2)^{O(1)}|G_0|^2$ pairs $(x,x')$ we have
\[\Psi_x = \Psi_{x'}\text{ on }Z^{\on{IH}}\cap W' \cap Z \cap V^{\on{low}}_x \cap V^{\on{IH}}_x \cap V^{\on{low}}_{x'} \cap V^{\on{IH}}_{x'} \cap V^{\on{IH}}_y\]
for at least $(c/2)^{O(1)}|G_0|$ elements $y \in G_0$.

\indent Apply Theorem~\ref{lin-sys-vanishing} for such pairs to find another multilinear variety $Z'$ of codimension $(r + r' + r'' + s' + s'' + \log c^{-1})^{O(1)}$ such that 
\[\Psi_x = \Psi_{x'}\text{ on }Z' \cap Z^{\on{IH}}\cap W' \cap Z \cap V^{\on{low}}_x \cap V^{\on{IH}}_x \cap V^{\on{low}}_{x'} \cap V^{\on{IH}}_{x'}.\]
Noting that $\Psi_x$ is defined on $Z' \cap Z^{\on{IH}} \cap W' \cap Z \cap V^{\on{low}}_x$, apply the induction hypothesis another time to complete the proof.
\end{proof}

\subsection{Simultaneous vanishing of MM-$r$-maps}

In this subsection, we treat the problem of MM-$r$-maps that have very dense zero sets. As it was remarked in the introduction, the only linear maps with such a property are zero maps. However, the situation in the multilinear setting is considerably more difficult and we prove a significantly weaker result in the form of the following dichotomy.

\begin{proposition}\label{nonvanishing-or-variety-dichotomy}
    Let $W \subseteq G_{[k]}$ be a multilinear variety of codimension $r$. Then there exists a multilinear variety $Z \subseteq W$ of codimension $r^{O(1)}$ such that, whenever $\psi_1, \dots, \psi_m : W \to H$ are multilinear maps, then we have
    \begin{itemize}
        \item[\textbf{(A)}] a point $x_{[k]} \in W$ such that $\psi_i(x_{[k]}) \not = 0$ for at least $\Omega(m)$ maps, or
        \item[\textbf{(B)}]we have $Z \subseteq \on{Z}(\psi_i)$ for at least $\Omega(m)$ maps.
    \end{itemize}
\end{proposition}

\noindent\textbf{Brief proof overview.} In this subsection, we use extension theory in more detail. The extension theory of MM-$r$-maps gives us a procedure which `peels off' quasirandom forms defining the domain one-by-one, and replaces biased forms by lower order ones. Eventually we get global maps, by extending with zero values at each step. If from a given map $\psi_i : W \to H$ we get a zero global map, then the zero map coincides with the map $\psi_i$ on a dense variety, independent of the actual $\psi_i$, but depending only on $W$ instead. On the other hand, if the extension of $\psi_i$ does not vanish everywhere, since it is a global multilinear map, it has at least $2^{-k}|G_{[k]}|$ points with non-zero values. Extension theory tells us that value of the extension map at each point in $G_{[k]}$ is determined by a formula of fixed length, so we are done after averaging.

\indent Therefore, during the proof, we shall describe an extension procedure that depends on the starting variety $W$, but not on the actual map $\phi : W \to H$. The extensions will typically be denoted by $\phi^{\on{ext}}$. We say that a point $x_{[k]} \in G_{[k]}$ has a \textit{formula} of length $\ell$ in a set $W \subseteq G_{[k]}$ if there are pairs $(y^{(i)}, \lambda_i) \in W \times \mathbb{F}_p$, where $i \in [\ell]$, such that, for any multilinear map $\phi : W \to H$ our extension procedure gives a map $\phi^{\on{ext}}$ such that
\[\phi^{\on{ext}}(x_{[k]}) = \sum_{i \in [\ell]} \lambda_i \phi(y^{(i)}_{[k]}).\]
As a simple example, if $U \leq G$ is a subspace and $\pi : G \to U$ is a projection, then way extend $\phi : U \to H$ by setting $\phi^{\on{ext}}(x) = \phi(\pi(x))$. This is a formula of length 1, with point $\pi(x)$.

\begin{proof}[Proof of Proposition~\ref{nonvanishing-or-variety-dichotomy}]
    During the proof, for a down-set $\mathcal{P}$, we shall find a $\mathcal{P}$-supported multilinear variety $D^{\mathcal{P}}$ of codimension $r^{O(1)}$ such that most maps $\psi_i$ can be extended to a map $\psi^{(\mathcal{P})}_i : D^{\mathcal{P}} \to H$ with the property that for any $x_{[k]} \in D^{\mathcal{P}}$ there is a formula independent of $\psi_i$, namely, the identity holds 
    \[\phi^{\on{ext}}(x_{[k]}) = \sum \lambda_j \phi(y^{(j)}_{[k]})\]
    for any multilinear map $\phi$ on $W$ and its resulting extension $\phi^{\on{ext}}$ to $D^{\mathcal{P}}$. We shall simultaneously find a multilinear variety $C^{\mathcal{P}}$ of codimension $r^{O(1)}$, not necessarily $\mathcal{P}$-supported, such that $\phi^{\on{ext}} = \phi$ on $C^{\mathcal{P}}$.

    The claim is true for the base case $\mathcal{P} = \mathcal{P}[k]$. Let a down-set $\mathcal{P}$ for which the inductive hypothesis holds be given with a maximal set $I$. Let $D = D^{\mathcal{P}}$ and $C = C^{\mathcal{P}}$ be the varieties of codimension $s \leq r^{O(1)}$ obtained so far.
    
    \indent Let $\eta > 0$ be a parameter to be chosen later. Let $\alpha_1, \dots, \alpha_{s_I} : G_I \to \mathbb{F}_p$ be the $I$-dependent multilinear forms defining $D$. Let us write $D = V \cap (\{\alpha = 0\} \times G_{I^c})$ for a $(\mathcal{P}\setminus \{I\})$-dependent variety $V$. Let $\beta_1, \dots, \beta_t : G_I \to \mathbb{F}_p$ be a maximal collection of linearly independent linear combinations of $\alpha_i$ which are not $\eta$-quasirandom with respect to $V$. Extend these to a basis of linear combinations $\rho_1, \dots, \rho_{t'} : G_I \to \mathbb{F}_p$. Thus, we have that each $\rho_i$ is $\eta$-quasirandom with respect to $V \cap (\{\beta = 0\} \times G_{I^c}) \cap (\cap_{j \not=i }\{\rho_j = 0\} \times G_{I^c})$.

    Take any $i_0 \in I$.

    \begin{claim}
        There exist points $e^{(1)}_I, \dots, e^{(t')}_I \in \{\beta = 0\} \cap V^{(I)}$, differing only in coordinate $i_0$, such that for all $i, j \in [t']$
        \[\rho_i(e^{(j)}_I) = \id(i = j).\]
    \end{claim}

    \begin{proof}
        Write $\rho_i(x_I) = R_i(x_{I \setminus \{i_0\}}) \cdot x_{i_0}$ and similarly let $\Gamma_j(x_{J_j})$ be multilinear maps arising from others forms defining $V$ which depend on $G_{i_0}$ and let $B_j(x_{I \setminus \{i_0\}})$ be multilinear maps arising from $\beta_1, \dots, \beta_{t}$. It suffices to find a point $e_{I \setminus \{i_0\}} \in V^{(I \setminus \{i_0\})}$ such that $R_i(e_{I \setminus \{i_0\}})$ are independent from other $\Gamma_j(e_{J_j})$, when $J_j \subseteq I \setminus \{i_0\}$. If it is not possible to find such a point $e_{I \setminus \{i_0\}}$, we have a linear combination which is non-zero at some $R_i$ and at least $p^{-s}|V^{(I \setminus \{i_0\})}| \geq p^{-2s}|G_{I \setminus \{i_0\}}|$ elements $e_{I \setminus \{i_0\}}$ such that 
        \[\sum_{i} \lambda_i R_i(e_{I \setminus \{i_0\}}) + \sum_j \mu_j \Gamma_j(e_{J_j}) + \sum_{j'} \mu'_jB_{j'}(e_{I \setminus \{i_0\}}) = 0.\]
        Hence,
        \begin{align*}&\exx_{e_I} \omega\Big(\sum_{i} \lambda_i R_i(e_{I \setminus \{i_0\}}) \cdot e_{i_0} + \sum_j \mu_j \Gamma_j(e_{J_j}) \cdot e_{i_0} + \sum_j \mu'_j B_j(e_{I \setminus \{i_0\}}) \cdot e_{i_0}\Big) \\
        &\hspace{1cm}= \exx_{e_I} \omega\Big(\sum_{i} \lambda_i \rho_i(e_I) + \sum_j \mu_j \gamma_j(e_{J_j \cup \{i_0\}}) + \sum_{i} \mu'_i \beta_i(e_I)\Big) \leq \eta,\end{align*}
        which is a contradiction with quasirandomness.
    \end{proof}

    Since the points $e^{(1)}_I, \dots, e^{(t')}_I$ coincide in the coordinates $I \setminus \{i_0\}$, we may write $e_{I \setminus \{i_0\}}$ for their common values. Let $U = \cap_{i \in [t']} V_{e^{(i)}_I}$. Our goal is to extend the given maps to $V \cap (U \times G_I) \cap (\{\beta = 0\} \times G_{I^c})$ using extension theory. This is done by first considering quasirandom slices $V_{x_{I^c}} \cap \{\beta = 0\}$ for most of $x_{I^c} \in U$ and then extending to the full variety. To that end, define $X \subseteq U$ to be the set of all $x_{I^c} \in U$ such that $\rho$ is $\sqrt[4]{\eta}$-quasirandom with respect to $\{\beta =0 \} \cap V_{x_{I^c}}$. By Lemma~\ref{quasirandomrest}, we have $|U \setminus X| \leq p^{2s}\sqrt{\eta}|G_{I^c}|$.

    Recall that $D = V \cap (\{\beta = 0\} \times G_{I^c}) \cap (\{\rho = 0\} \times G_{I^c})$ and let $\theta : V \cap (\{\beta = 0\} \times G_{I^c}) \cap (\{\rho = 0\} \times G_{I^c}) \to H$ be an arbitrary multilinear map.

    \noindent\textbf{Extension to most of $V \cap (X \times G_I) \cap (\{\beta = 0\} \times G_{I^c})$.} Let $x_{I^c} \in X$ be given. Observe that, since $x_{I^c} \in U = \cap_{i \in [t']} V_{e^{(i)}_I}$, for each $i \in [t']$, $(x_{I^c}, e^{(i)}_I) \in V$, so $e^{(i)}_I \in V_{x_{I^c}}$. In particular, the point $e^{(i)}_I$ belongs to $V_{x_{I^c}} \cap \{\beta = 0\} \cap (\cap_{j \in [i-1]} \{\rho_j = 0\})$. Hence, by Theorem~\ref{qrExtension} applied $t'$ times, there exists a unique multilinear map $\theta'_{x_{I^c}} : V_{x_{I^c}} \cap \{\beta = 0\} \to H$ extending $\theta_{x_{I^c}}$ with values $\theta'_{x_{I^c}}(e^{(i)}_I) = 0$ for all $i \in [t']$.
    
    \indent We observe that for each $x_{[k]} \in V \cap (X \times G_I) \cap (\{\beta = 0\} \times G_{I^c})$ we have a bounded length formula. Namely, if $x_{I^c} \in X$, by Lemma~\ref{connectedlayerslemma}, the layer $V_{x_{I^c}} \cap \{\beta = 0\} \cap \{\rho  = \rho(x_{I^c})\}$ is connected and of diameter at most $(k+2)^2$. The bounded diameter implies existence of desired formulas for points in $V \cap (X \times G_I) \cap (\{\beta = 0\} \times G_{I^c})$.

    \begin{claim}
        If $x_{[k]} \in V \cap (X \times G_I) \cap (\{\beta = 0\} \times G_{I^c})$, then there exists a formula for $x_{[k]}$ of length at most $(k+2)^2$, with all arguments in $D$.
    \end{claim}

    \begin{proof}
        The layer $V_{x_{I^c}} \cap \{\beta = 0\} \cap \{\rho  = \rho(x_{I^c})\}$ is connected and of diameter at most $(k+2)^2$. Since $x_{I^c} \in U$, we have $e^{(i)}_I \in  V_{x_{I^c}}$ for all $i \in [t']$, so the point $e'_I$, given by coordinates $e'_{I \setminus \{i_0\}} = e_{I \setminus \{i_0\}}$ and $e'_{i_0} = \sum_{j \in [t']} \rho_j(x_I)e^{(j)}_{i_0}$ belongs to the layer. Hence, there is a path of length at most $(k+2)^2$ from $x_{I}$ to $e'_I$. The differences of consecutive elements give points in $D$, and the extension vanishes at the final point.
    \end{proof}

    Define map $\theta' : V \cap (X \times G_I) \cap (\{\beta = 0\} \times G_{I^c}) \to H$ by putting these maps together, meaning that $\theta'(x_{[k]}) = \theta'_{x_{I^c}}(x_I)$. So far, we know that $\theta'$ is multilinear in coordinates $G_I$. We move on to proving that it is almost multilinear in the remaining coordinates and using this to obtain multilinear extension to the full variety $V \cap (U \times G_I) \cap (\{\beta = 0\} \times G_{I^c})$.

    \noindent\textbf{Extension to $V \cap (U \times G_I) \cap (\{\beta = 0\} \times G_{I^c})$.} Let $i_0 \in I^c$ be a direction, and let $x_{[k]}, y_{[k]}, z_{[k]}$ be three points in $ V \cap (X \times G_I) \cap (\{\beta = 0\} \times G_{I^c})$ differing only in the coordinate $i_0$, with $x_{i_0} + y_{i_0} = z_{i_0}$. Then the map $w_{I} \mapsto \theta'_{x_{I^c}}(w_I) + \theta'_{y_{I^c}}(w_I) - \theta'_{z_{I^c}}(w_I)$ is a multilinear extension of the zero map from $V_{x_{I^c}} \cap V_{y_{I^c}}\cap \{\beta = 0\} \cap \{\rho = 0\}$ to $V_{x_{I^c}} \cap V_{y_{I^c}}\cap \{\beta = 0\}$, with values $e^{(1)}_I, \dots, e^{(t')}_I \mapsto 0$. Provided the map $\rho$ is $p^{-\blc(s+1)}$-qausirandom with respect to the variety $V_{x_{I^c}} \cap V_{y_{I^c}}\cap \{\beta = 0\}$, Theorem~\ref{qrExtension} implies that such an extension is unique and therefore 
    \[\theta'(x_{[k]}) + \theta'(y_{[k]}) - \theta'(z_{[k]}) =  \theta'_{x_{I^c}}(x_I) + \theta'_{y_{I^c}}(y_I) - \theta'_{z_{I^c}}(z_I) = \theta'_{x_{I^c}}(x_I) + \theta'_{y_{I^c}}(x_I) - \theta'_{z_{I^c}}(x_I) = 0.\]
    The required quasiradnomness condition holds for all but at most $p^{O(s+1)}\eta^{\Omega(1)}|G_{i_0}||G_{[k]}|$ such triples, so, provided $\eta \leq p^{-\blc s}$, we have that $\theta'$ respects all but at most $\eta^{\Omega(1)}|G_{i_0}||G_{[k]}|$ additive triples in direction $i_0$. Crucially, the structure of the respected directional additive triples depends only on the variety $V$ and maps $\beta$ and $\rho$. Hence, by Lemma~\ref{almost-multilinear-maps}, there exists a set $Y \subseteq V \cap (U \times G_I) \cap (\{\beta = 0\} \times G_{I^c})$ on which it is multilinear. Lemma~\ref{trivialextensionsverydense} gives a subset $Y' \subseteq Y$ of size $|Y'| \geq |V \cap (U \times G_I) \cap (\{\beta = 0\} \times G_{I^c})| - p^{O(s)}\eta^{\Omega(1)}|G_{[k]}|$, depending on $Y$ (and thus variety $D$) only, such that $\theta'$ extends uniquely from $Y'$ to $V \cap (U \times G_I) \cap (\{\beta = 0\} \times G_{I^c})$. It remains to show that this extension has bounded length formulas.

    For the points outside $Y'$, we may reach them using directional convolutions in $O(1)$ steps, so the formulas still have bounded length.

    Finally, apply Theorem~\ref{arankprank} to components of $\beta_i$, to find a multilinear variety $Z \subseteq G_{I \setminus \{\max I\}}$ of codimension $\log^{O(1)} (2\eta^{-1})$ such that $Z \times G_{\max I} \subseteq \{\beta = 0\}$. We define desired varieties $D^{\mathcal{P} \setminus \{I\}} = V \cap (U \times G_I) \cap (Z \times G_{I^c \cup \{\max I\}})$ and $C^{\mathcal{P} \setminus \{I\}}$ as $C^{\mathcal{P}} \cap (Z \times G_{I^c \cup \{\max I\}})$. After choosing $\eta \geq p^{-\blc s}$ so that all required bounds on $\eta$ hold, this completes the proof of the inductive step.

    \noindent\textbf{Obtaining the dichotomy.} After we conclude the induction above, we reach the case $\mathcal{P} = \emptyset$, so we get variety $C$ of codimension $r^{O(1)}$ and for each $x_{[k]} \in G_{[k]}$ we get a formula of length $O(1)$ with points in the initial variety $W$. We distinguish two cases.

    \noindent\textbf{Case 1.} Suppose first that at least half of maps are extended to the zero map. As extensions coincide with initial maps on the variety $C$, we obtain the conclusion \textbf{(B)}.

    \noindent\textbf{Case 2.} Now assume that at least half of maps extend to a non-zero map. Global multilinear map has at least $2^{-k}|G_{[k]}|$ points that are non-zero, so by averaging, there exists a point $x_{[k]}$ such that $\psi^{\on{ext}}_i(x_{[k]}) \not=0$ for at least $2^{-k-1}m$ maps. But we have a formula of length $O(1)$ for $x_{[k]}$. By averaging, for some $y_{[k]}$ among the arguments in the formula $\psi_i(y_{[k]}) \not=0$ for $\Omega(m)$ functions.
\end{proof}

\section{Approximately linear system of global multilinear maps}

In this section we carry out the first phase of the proof, comprising \textbf{Steps 1.1--1.3}. Our starting point is a set $X \subseteq G_0$ and a collection of global multilinear map $\phi_x : G_{[k]} \to H$ such that $|Z(\phi_{x_1} - \phi_{x_2} + \phi_{x_3} - \phi_{x_4})| \geq c|G_{[k]}|$ holds for at least $c|G_0|^3$ additive quadruples $(\upd{x}_1, x_2, \upd{x}_3, x_4)$ in $X$. The goal in this section is to improve the structure in two ways, firstly, to have a full subspace in place of $X$ and, secondly, to have all additive quadruples respected. Thus, the work in this section can be summarized by the following theorem.

\begin{theorem}[Phase 1 of the proof]\label{phase1mainres}
    Suppose that $X \subseteq G_0$ and that $(\phi_x : G_{[k]} \to H)_{x \in X}$ is a collection of global multilinear maps such that $|Z(\phi_{x_1} - \phi_{x_2} + \phi_{x_3} - \phi_{x_4})| \geq c|G_{[k]}|$ holds for at least $c|G_0|^3$ additive quadruples $(\upd{x}_1, x_2, \upd{x}_3, x_4)$ in $X$. Then there exist a quantity $r \leq \log^{O(1)}(2c^{-1})$, subspace $U \leq G_0$ of codimension $r$, a system of global multilinear maps  $(\psi_x : G_{[k]} \to H)_{x \in U}$ such that every additive quadruple in $(\psi_x)_{x \in U}$ is $\exp(-r)$-respected and for each $x \in U$ we have  at least  $\exp(-r)|G_0|^3$ quadruples $(x_1 ,x_2 , x_3 , x_4) \in X^4$ such that $x_1 + x_2 - x_3 - x_4 = x$ and 
    \[|Z(\psi_x - \phi_{x_1} - \phi_{x_2} + \phi_{x_3} + \phi_{x_4})| \geq \exp(-r)|G_{[k]}|.\]
\end{theorem}

Recall that, throughout the paper, we assume the main result, Theorem~\ref{strFmult}, for products of $k$ vector spaces.

For the rest of this section, we consider systems of global multilinear maps $\phi_x : G_{[k]} \to H$, defined for $x$ belonging to a subset of $G_0$. By an additive $2d$-tuple $x_{[2d]}$ we think of elements satisfying $\sum_{i \in [2d]} (-1)^i x_i = 0$, and we allow variance in signs for quadruples. We say that an additive $2d$-tuple is \textit{$c$-respected} if $\Big|Z\Big(\sum_{i \in [2d]} (-1)^i \phi_{x_i}\Big)\Big| \geq c |G_{[k]}|$. In this language, the main assumption of this section becomes that $c|G_0|^3$ additive quadruples in $X$ are $c$-respected.

The rest of the section is split into three subsections, each covering a step of the proof.

\subsection{First abstract Balog-Szemer\'edi-Gowers step}

In this subsection, we apply the abstract Balog-Szemer\'edi-Gowers theorem for the first time. This step provides us with a short list of multilinear maps $\psi_1, \dots, \psi_m : G_{[k]} \to H$ and a large sub-collection of the original multilinear maps in which all additive $O(1)$-tuples are $(c/2)^{(O(1)}$-respected, after subtracting a multilinear map among $\psi_1, \dots, \psi_m$.

\begin{proposition}\label{phase1step1prop}
    Let $d \leq O(1)$ be a natural number. Let $X \subseteq G_0$ be a set and let $(\phi_x : G_{[k]} \to H)_{x \in X}$ be a system of global multilinear maps. Suppose that at least $c|G_0|^3$ additive quadruples in $X$ are $c$-respected. Then there exists a subset $X' \subseteq X$ of size $|X'| \geq (c/2)^{O(1)}|G_0|$ and a collection of multilinear maps $\psi_1, \dots, \psi_m : G_{[k]} \to H$, where $m \leq (2/c)^{O(1)}$, such that for each additive $2d$-tuple $x_{[2d]}$ in $X'$ we have 
    \[\Big|Z\Big(\sum_{i \in [2d]} (-1)^i\phi_{x_i} - \psi_j\Big)\Big| \geq (c/2)^{O(1)}|G_{[k]}|\]
    for some $j \in [m]$.
\end{proposition}

\begin{proof}
    For $i \in [36]$, let $\mathcal{Q}_i$ be the collection of all additive quadruples $(\upd{x}_1, x_2, x_3, \upd{x}_4)$ in $X$ such that 
    \[|Z(\phi_{x_1} - \phi_{x_2} - \phi_{x_3} + \phi_{x_4})| \geq c^i|G_{[k]}|.\]

    Let us check the conditions of the abstract Balog-Szemer\'edi--Gowers theorem (Theorem~\ref{absg}), where we we use $X$ in place of $A$. The conditions \textbf{(i)} and \textbf{(ii)} trivially hold. Weak transitivity holds for $c' = \frac{1}{|G_0|}$, i.e. a single pair suffices due to Lemma~\ref{poscorvars}.

    Apply Theorem~\ref{absg} to find a subset $X' \subseteq X$ of size $|X'| \geq (c/2)^{O(1)} |G_0|$ such that for each additive $2d$-tuple $x_{[2d]}$ in $X'$ we may pass to $(c/2)^{O(1)}|G_0|^{6d - 1}$ additive $6d$-tuples $u_{[6d]}$ via additive quadruples in $\mathcal{Q}_{36}$, as in the conclusion of Theorem~\ref{absg}. Unpacking the conclusion and using Lemma~\ref{poscorvars} several times, we get
    \[\bigg|Z\bigg(\Big(\sum_{i \in [2d]} (-1)^i\phi_{x_i}\Big) - \Big(\sum_{i \in [6d]} (-1)^i\phi_{u_i}\Big)\bigg)\bigg| \geq (c/2)^{O(1)}|G_{[k]}|.\]
    Let $U_{x_{[2d]}}$ be the collection of such additive $6d$-tuples $u_{[6d]}$.

    Let $U_1, \dots, U_m$ be a maximal collection of mutually disjoint sets among $U_{x_{[2d]}}$. The bounds on size of $U_{x_{[2d]}}$ imply that $m \leq (2/c)^{O(1)}$. For each $i \in [m]$, take an arbitrary $u_{[6d]} \in U_i$, and define $\psi_i = \sum_{j \in [6d]} (-1)^j \phi_{u_j}$. We claim that these maps have the desired property.

    Take an arbitrary additive $2d$-tuple $x_{[2d]}$ in $X'$. Thus $U_{x_{[2d]}}$ intersects some $U_j$, coming from $2d$-tuple $x'_{[2d]}$, at some $v_{[6d]}$, say. Let $\psi_j$ be defined by $u_{[6d]} \in U_j$. Then we have
    \begin{align*}\Big(\sum_{i \in [2d]} (-1)^i\phi_{x_i}\Big) - \psi_j = &\sum_{i \in [2d]} (-1)^i\phi_{x_i} - \sum_{i \in [6d]} (-1)^i \phi_{u_i} \\
    = &\Big(\sum_{i \in [2d]} (-1)^i\phi_{x_i} - \sum_{i \in [6d]} (-1)^i \phi_{v_i}\Big) + \Big(\sum_{i \in [6d]} (-1)^i \phi_{v_i} - \sum_{i \in [2d]} (-1)^i\phi_{x'_i}\Big) \\
    &\hspace{6cm}+ \Big(\sum_{i \in [2d]} (-1)^i\phi_{x'_i} - \sum_{i \in [6d]} (-1)^i \phi_{u_i}\Big).\end{align*}
    The proposition follows after applying Lemma~\ref{poscorvars} a few times.
\end{proof}

\subsection{Ensuring all quadruples are respected}\label{ensuringrespStep}

The outcome of the previous step is not quite what we want, as there are error functions $\psi_1, \dots, \psi_m$ present in the respectedness condition. In this step, we remove these functions.

\begin{proposition}\label{phase1step2prop}
    Let $d \leq O(1)$ be a natural number. Let $X \subseteq G_0$ be a set of density $c$ and let $(\phi_x : G_{[k]} \to H)_{x \in X}$ be a system of global multilinear maps. Let $\psi_1, \dots, \psi_m : G_{[k]} \to H$ be a further collection of multilinear maps. Suppose that every additive $2d$-tuple $x_{[2d]}$ in $X$ satisfies
    \[\Big|Z\Big(\psi_j - \sum_{i \in [2d]} (-1)^i\phi_{x_i}\Big)\Big| \geq c|G_{[k]}|\]
    for some $j \in [m]$. Then there exists a subset $X' \subseteq X$ of size $|X'| \geq \exp(-O(\log m + \log^{O(1)}(2c^{-1}))) |G_0|$ such that every additive $2d$-tuple in $X'$ is $p^{-k \log^{O(1)}(2c^{-1})}$-respected.
\end{proposition}

\begin{proof}
    We employ algebraic dependent random choice. Let $r, s$ be positive integers and let $\eta > 0$, to be chosen later. Let $J$ be the set of indices $j \in [m]$ such that $|Z(\psi_j)| \leq \eta |G_{[k]}|$. We begin by choosing a sequence of pairs $(P_i, \pi_i)$, where $P_i = W_{i,1} \times \dots \times W_{i, k}$ is a product of subspaces $W_{i, j} \leq G_j$ of dimension $r$, and $\pi_i:H \to \mathbb{F}_p^s$ is a linear projection, for $i \in [\ell]$ for some $\ell$ that will not be too large.

    We choose these pairs iteratively, until for each $j \in J$ we have some pair $(P_i, \pi_i)$ such that the only points $x_{[k]} \in P_i$ for which $\pi_i \circ \psi_j(x_{[k]}) = 0$ have one of $x_1, \dots, x_k$ equal to 0. Writing $\mathsf{Z}$ for the set of $x_{[k]} \in G_{[k]}$ such that some $x_i = 0$, we thus want $Z(\pi_i \circ \psi_j) \cap P_i \subseteq \mathsf{Z}$. Note also that any global multilinear map necessarily vanishes on $\mathsf{Z}$.

    Let $J' \subseteq J$ be the collection of those indices $j$ for which $\psi_j$ still does not have the desired pair. Let $P$ and $\pi$ be randomly chosen uniformly over all admissible such objects.

    \begin{claim}
        Suppose that $\eta \leq \frac{1}{4}p^{-kr}$ and $s \geq kr + 2$. For each $j \in J'$, the probability that $Z(\pi \circ \psi_j) \cap P  \subseteq \mathsf{Z}$ is at least $1/2$.
    \end{claim}

    \begin{proof}
        We first show that the probability that $Z(\psi_j) \cap P  \subseteq \mathsf{Z}$ is at least $3/4$. Namely, for any $x_{[k]} \in G_{[k]} \setminus \mathsf{Z}$ we have $\mathbb{P}(x_{[k]} \in P) = \prod_{i \in [k]} \frac{p^r - 1}{|G_i| - 1} \leq p^{kr}/|G_{[k]}|$. Hence
        \[\exx |(Z(\psi_j) \setminus \mathsf{Z}) \cap P| = \sum_{x_{[k]} \in Z(\psi_j) \setminus \mathsf{Z}} \mathbb{P}(x_{[k]} \in P) \leq \eta |G_{[k]}| p^{kr}/|G_{[k]}| = \eta p^{kr}.\]

        By Markov's inequality, the probability that $|(Z(\psi_j) \setminus \mathsf{Z}) \cap P| \geq 1$ is at most $\eta p^{kr}$, proving the claim, as long as $\eta p^{kr}\leq \frac{1}{4}$.

        Fix now a choice of $P$ such that $Z(\psi_j) \cap P  \subseteq \mathsf{Z}$. We show that the probability that $Z(\pi \circ \psi_j)  \cap P  \subseteq \mathsf{Z}$ also holds is at least 3/4. Namely, given any point $x_{[k]} \in P \setminus \mathsf{Z}$, we know that $\psi_j(x_{[k]}) \not= 0$. Then $\mathbb{P}(\pi \circ \psi(x_{[k]}) = 0) = p^{-s}$. Thus, the probability that $\pi \circ \psi_j \not=0$ holds at all points in $P \setminus \mathsf{Z}$ is at least $1 - p^{kr-s}$. As long as $kr + 2\leq s$, we get the desired bound on the probability. With the work above, the claim follows.
    \end{proof}

    Assuming $\eta \leq \frac{1}{4}p^{-kr}$ and $s \geq kr + 2$, at each step, we may choose a pair $(P_i, \pi_i)$ halving the number of $\psi_j$ without the desired pair. Hence, the procedure terminates in $\ell \leq \log_2 m$ steps.

    \noindent\textbf{Choosing the set $X'$.} To choose the set $X'$, we take random multilinear maps $\mu_i : P_i \to \mathbb{F}_p^s$ uniformly and independently at random and set $X' = \{x \in X : (\forall i \in [\ell])\,\pi_i \circ \phi_x |_{P_i} = \mu_i\}$. The number of possible choices of the map is $\mu_i$ is $p^{sr^k}$ as a multilinear map on $P_i = W_{i, 1} \tdt W_{i,k}$ corresponds to a linear map on the tensor product $W_{i, 1} \otimes \dots \otimes W_{i,k}$ and the dimension that tensor product is $r^k$. Hence, there exists a choice of $X'$ such that $|X'| \geq p^{-\ell sr^k}|X|$. We claim that $X'$ has the desired property.

    Take any additive $2d$-tuple $x_{[2d]}$ in $X'$. By assumptions of the proposition, there exists $j \in [m]$ such that 
     \[\Big|Z\Big(\psi_j - \sum_{i \in [2d]} (-1)^i\phi_{x_i}\Big)\Big| \geq c|G_{[k]}|.\]
    If $j \notin J$, then $|Z(\psi_j)| \geq \eta |G_{[k]}|$ and by Lemma~\ref{poscorvars}, we have 
    \[\Big|Z\Big(\sum_{i \in [2d]} (-1)^i\phi_{x_i}\Big)\Big| \geq \eta c|G_{[k]}|,\]
    which completes the proof in this case.

    If $j \in J$, then there exists $i \in [\ell]$ such that $Z(\pi_i \circ \psi_j) \cap P_i  \subseteq \mathsf{Z}$. On the other hand, $Z\Big(\psi_j - \sum_{i \in [2d]} (-1)^i\phi_{x_i}\Big)$ is a variety in $G_{[k]}$ of density $c$, so by Theorem~\ref{densetolowcodim}, there exists a multilinear variety $W$ of codimension $r_0 = \log^{O(1)}(2c^{-1})$ inside it. Note that by Lemma~\ref{lowcodsize} $|W \cap P_i| \geq p^{-kr_0}|P_i|$. On the other hand, $|P_i \cap \mathsf{Z}| \leq k p^{(k-1)r} = k p^{-r} |P_i|$, so, as long as $r > k (r_0 + 1)$, $W \cap P_i$ contains a point $y_{[k]}$ outside $\mathsf{Z}$. For such a point, we have
    \[\pi_i \circ \psi_j(y_{[k]}) = \pi_i\circ \Big( \sum_{j' \in [2d]} (-1)^{j'
    }\phi_{x_i}\Big)(y_{[k]}) = \sum_{j' \in [2d]} (-1)^{j'} \pi_i \circ \phi_{x_{j'}}(y_{[k]}) = 0,\]
    which is a contradiction.

    We set $r =  k (r_0 + 1) + 1$, $\eta = \frac{1}{4}p^{-kr}$ and $s = kr + 2$ to finish the proof.
\end{proof}

\subsection{Linear system indexed by a subspace}

We now apply the robust Bogolyubov-Ruzsa theorem to obtain the completion of the given system of global multilinear maps.

\begin{proposition}\label{phase1step3prop}
    Let $d$ be a natural number. Let $X \subseteq G_0$ be a set of density $c$ and let $(\phi_x : G_{[k]} \to H)_{x \in X}$ be a system of global multilinear maps such that every additive $16$-tuple is $c$-respected. Then there exist a subspace $U \leq G_0$ of codimension $\log^{O(1)}(2c^{-1})$, a system of global multilinear maps  $(\psi_x : G_{[k]} \to H)_{x \in U}$ such that every additive quadruple in $(\psi_x)_{x \in U}$ is $c$-respected and for each $x \in U$ we have  at least  $(c/2)^{O(1)}|G_0|^3$ quadruples $(x_1 ,x_2 , x_3 , x_4) \in X^4$ such that $x_1 + x_2 - x_3 - x_4 = x$ and 
    \[|Z(\psi_x - \phi_{x_1} - \phi_{x_2} + \phi_{x_3} + \phi_{x_4})| \geq c|G_{[k]}|.\]
\end{proposition}

\begin{proof}
    Apply robust Bogolyubov-Ruzsa theorem (Theorem~\ref{rbrlemma}) to the set $X$ to obtain a subspace $U \leq G_0$ of the claimed codimension such that for each $x \in U$ we have  at least  $(c/2)^{O(1)}|G_0|^3$ quadruples $(x_1 ,x_2 , x_3 , x_4) \in X^4$ such that $x_1 + x_2 - x_3 - x_4 = x$. Take arbitrary such quadruple and define $\psi_x =  \phi_{x_1} +\phi_{x_2} - \phi_{x_3} - \phi_{x_4}$. If $(x'_1 ,x'_2 , x'_3 , x'_4) \in X^4$ is any other quadruple with $x'_1 + x'_2 - x'_3 - x'_4 = x$, then joining it with $(x_1, \dots, x_4)$ gives a $c$-respected addditive 8-tuple, so we have
    \[|Z(\psi_x - \phi_{x'_1} - \phi_{x'_2} + \phi_{x'_3} + \phi_{x'_4})| = |Z( \phi_{x_1} + \phi_{x_2} - \phi_{x_3} - \phi_{x_4} - \phi_{x'_1} - \phi_{x'_2} + \phi_{x'_3} + \phi_{x'_4}))|  \geq c|G_{[k]}|.\]

    Similarly, for any additive quadruple $(\upd{y_1}, y_2, \upd{y_3}, y_4)$ in $X'$, the multilinear map $\psi_{y_1} - \psi_{y_2} + \psi_{y_3} -\psi_{y_4}$ equals a linear combination $\sum_{i \in [16]} (-1)^{i}\phi_{z_i}$ for a suitable additive 16-tuple $z_{[16]}$. Hence, every additive quadruple in $(\psi_y)_{y \in U}$ is $c$-respected.
\end{proof}

Combining all steps in this section, we conclude Theorem~\ref{phase1mainres}.

\begin{proof}[Proof of Theorem~\ref{phase1mainres}]
    The theorem follows after applying Propositions~\ref{phase1step1prop},~\ref{phase1step2prop} and~\ref{phase1step3prop}.
\end{proof}

\section{Changing the category}

In this section, we carry out the second phase of the argument and pass from a system of global multilinear maps to a system of MM-$r$-maps. Given a system of MM-$r$-maps $(\phi_x: V_x \to H)_{x \in X}$ we use a different notion of respectedness. Namely, we say that an additive $2d$-tuple $x_{[2d]}$ is \textit{variety-respected} if $\sum_{i \in [2d]} (-1)^i \phi_{x_i}(y_{[k]}) = 0$, whenever $y_{[k]} \in \cap_{i \in [2d]} V_{x_i}$. Our goal is to find a system where most of additive tuples are variety-respected.

\begin{theorem}\label{changecategorystep}
     Let $d \leq O(1)$ be a natural number and let $\varepsilon > 0$. Let $(\phi_x : G_{[k]} \to H)_{x \in G_0}$ be a system of global multilinear maps such that every additive quadruple is $c$-respected. Then there exist multilinear varieties $V_x$ for $x \in G_0$ of codimension at most $ \log^{O(1)}(2\varepsilon^{-1}c^{-1})$, such that $(1-\varepsilon)|G_0|^{2d-1}$ additive $2d$-tuples in $G_0$ are variety-respected in the system $(\phi_x : V_x \to H)_{x \in G_0}$.
\end{theorem}

In the proof, we need a preliminary lemma that shows that random products of small subspaces meet dense sets with high probability.

    \begin{lemma}\label{randomproductintersection}
        Let $s$ be a positive integer. For each $i \in [k]$, take elements $e_{i, j} \in G_i$ for $j \in [s]$ uniformly and independently at random and define subspace $E_i = \langle e_{i, 1}, \dots, e_{i, s}\rangle$.
        
        \indent Let $X \subseteq G_{[k]}$ be a set of density $\alpha$. Then 
        \[\mathbb{P}(X \cap E_{1} \tdt E_k \not= \emptyset) \geq 1 - \frac{4pk\alpha^{-1}}{p^s - 1}.\]
    \end{lemma}

    \begin{proof}
        We use the second moment method. Let $\Lambda$ be the collection of $k$-tuples of vectors $(\lambda_1, \dots, \lambda_{k}) \in (\mathbb{F}_p^s \setminus \{0\})^k$ such that $(\sum_{j = 1}^s \lambda_{1, j} e_{1,j}, \dots, \sum_{j = 1}^s \lambda_{k,j} e_{k, j}) \in X$. We have $\mathbb{P}\Big((\lambda_{1}, \dots, \lambda_{k}) \in \Lambda\Big) = \frac{|X|}{|G_{[k]}|}$, so 
        \[\exx |\Lambda| = (p^{s} - 1)^k \frac{|X|}{|G_{[k]}|}.\]
        Note also that if $(\mu_1, \dots, \mu_{k})$ is another tuple of vectors such that $\mu_i$ and $\lambda_i$ are independent for all $i$, then 
        \[\mathbb{P}\Big((\lambda_1, \dots, \lambda_{k}), (\mu_1, \dots, \mu_{k}) \in \Lambda\Big) = \frac{|X|^2}{|G_{[k]}|^2}.\]
        If there are dependencies, then we have the trivial bound
        \[\mathbb{P}\Big((\lambda_1, \dots, \lambda_{k}), (\mu_1, \dots, \mu_{k}) \in \Lambda\Big) \leq \mathbb{P}\Big((\lambda_1, \dots, \lambda_{k}) \in \Lambda \Big) = \frac{|X|}{|G_{[k]}|}.\]
        Hence 
        \begin{align*}\exx |\Lambda|^2 \leq & (p-1)k (p^s - 1)^{2k - 1}\frac{|X|}{|G_{[k]}|} + (p^s - 1)^{2k}\frac{|X|^2}{|G_{[k]}|^2} \\
        = &(p-1)k (p^s - 1)^{2k - 1} \alpha^{-1} \Big(\exx |\Lambda|\Big)^2 + \Big(\exx |\Lambda|\Big)^2 \leq \Big(1 + \frac{pk\alpha^{-1}}{p^s - 1}\Big)\Big(\exx |\Lambda|\Big)^2.\end{align*}

        By Chebyshev's inequality, we have
        \[\mathbb{P}\Big(\Big||\Lambda| - \ex |\Lambda|\Big| \geq \frac{1}{2}\ex |\Lambda|\Big) \leq \frac{4pk\alpha^{-1}}{p^s - 1},\]
        proving the claim.
    \end{proof}

    We may now prove Theorem~\ref{changecategorystep}. We rely on Theorem~\ref{arankprank}. However, it is applied only at the beginning, and the rest of the proof is about algebra of multilinear forms and careful manipulation of their partition rank.

\begin{proof}[Proof of Theorem~\ref{changecategorystep}]
    For each $x \in G_0$, let $\tilde{\phi}_x : G_{[k]} \times H \to \mathbb{F}_p$ be the multilinear form given by $\tilde{\phi}_x (y_{[k]}, h) = \phi_x(y_{[k]}) \cdot h$. The assumption of the proposition translates to 
    \[\on{bias}\Big(\sum_{i \in [4]}  (-1)^{i}\tilde{\phi}_{x_i} \Big) \geq c\]
    for any additive quadruple $(\upd{x_1}, x_2, \upd{x_3}, x_4)$. During this proof, we write $G_{k+1}$ for $H$ to simplify the notation.

    The proof is inductive. Let $\mathcal{P} \subseteq \mathcal{P}[k] \setminus \{\emptyset\}$ denote a down-set. Initially, $\mathcal{P} = \mathcal{P}[k] \setminus \{\emptyset\}$ and at each step of the proof we remove a set from $\mathcal{P}$ until it becomes empty.

    We claim that for each $\mathcal{P}$ we have a collection $\mathcal{Q}$ of additive $2d$-tuples in $G_0$ of size $|\mathcal{Q}| \geq \Big(1 - \frac{2^k - |\mathcal{P}|}{2^k}\varepsilon\Big)|G_0|^{2d-1}$, a parameter $r \leq \log^{O(1)}(2\varepsilon^{-1}c^{-1})$, a collection of varieties $(V_x)_{x \in G_0} \subseteq G_{[k]}$ of codimension at most $r$ such that, for each $q = x_{[2d]} \in \mathcal{Q}$, we have a low partition rank decomposition such that
    \begin{equation}\label{controlledprankeqonvars}
           \Big(\forall y_{[k+1]} \in (V_{x_1} \cap \dots \cap V_{x_{2d}}) \times G_{k+1} \Big) \hspace{1cm}\sum_{i \in [2d]} (-1)^{i} \tilde{\phi}_{x_i}(y_{[k+1]}) = \sum_{j \in [t_q]} \beta_{q, j}(y_{I_{q,j}}) \gamma_{q, j}(y_{[k+1] \setminus I_{q,j}}),
    \end{equation}

    with $t_q \leq r$ and for all $j \in [t_q]$, we have $k+1 \notin I_{q,j} \in \mathcal{P}$.

    When $\mathcal{P} = \mathcal{P}[k] \setminus \{\emptyset\}$, by Lemma~\ref{poscorvars}, $\on{bias}\Big(\sum_{i \in [2d]}  (-1)^{i}\tilde{\phi}_{x_i} \Big) \geq c^{O(1)}$, so we may apply Theorem~\ref{arankprank} to conclude the desired equality in this case, with $V_x = G_{[k]}$ for all $x \in G_0$, settling the induction base.

    Suppose now that~\eqref{controlledprankeqonvars} holds for some non-empty $\mathcal{P}$ with maximal set $I$. Reordering the forms, we may rewrite the right-hand side of~\eqref{controlledprankeqonvars} as
    \begin{equation}\sum_{j \in [t'_q]} \beta_{q, j}(y_{I}) \gamma_{q, j}(y_{[k+1] \setminus I}) + L_q(y_{[k+1]})\label{controlledprankeqonvarsNewRHS}\end{equation}
    where $L_q(y_{[k+1]})$ has the desired shape for the down-set $\mathcal{P}\setminus \{I\}$. Also, write $V_q = V_{x_1} \cap \dots \cap V_{x_{2d}}$.

    Let $\Lambda_q$ be the set of all values attained by the vector $(\gamma_{q, j}(y_{[k+1] \setminus I}))_{j\in t'_q}$ inside $\mathbb{F}_p^{t'_q}$ as $y_{[k+1] \setminus I}$ ranges over $V_q^{([k] \setminus I)} \times G_{k+1}$. We claim that for $\lambda \in \Lambda_q$, the linear combination $\lambda \cdot \beta_q$ is closely related to $\sum_{i \in [2d]} (-1)^{i} \tilde{\phi}_{x_i}$.
    
    \begin{claim}\label{claimyqlrankdec}
        Let $\lambda \in \Lambda_q$. Then there exists a set $Y_{q, \lambda} \subseteq G_{[k+1] \setminus I}$ of size at least $p^{-k(2d+1)r}|G_{[k+1] \setminus I}|$ such that for each $y_{[k+1] \setminus I} \in Y_{q, \lambda}$, there exists a multilinear form $\delta_q^{(y_{[k+1] \setminus I})} : G_I \to \mathbb{F}_p$ of partition rank at most $r$, such that 
        \[\Big(\forall z_I \in (V_q)_{y_{[k] \setminus \{I\}}}\Big)\hspace{1cm} \sum_{i \in [2d]} (-1)^{i} \tilde{\phi}_{x_i}(y_{[k+1] \setminus I}, z_I) = \lambda \cdot \beta_q(z_I) + \delta_q^{(y_{[k+1] \setminus I})}(z_I).\]
    \end{claim}

    \begin{proof}
        Let $y_{[k+1] \setminus I}$ be such that $(\gamma_{q,j}(y_{[k+1] \setminus I}))_{j\in t'_q} = \lambda$ and $y_{[k+1] \setminus I} \in (V_q)^{([k] \setminus I)} \times G_{k+1}$. Note that~\eqref{controlledprankeqonvars} and~\eqref{controlledprankeqonvarsNewRHS} imply 
        \begin{align*}&\Big(\forall y_{I} \in (V_q)_{y_{[k] \setminus I}} \times G_{k+1} \Big)\\
        &\hspace{1cm}\sum_{i \in [2d]} (-1)^{i} \tilde{\phi}_{x_i}(y_{[k+1]}) = \sum_{j \in [t'_q]} \beta_{q, j}(y_{I}) \gamma_{q, j}(y_{[k+1] \setminus I}) + L_q(y_{[k+1]}) = \lambda \cdot \beta_q(y_I) + (L_q)_{y_{[k+1] \setminus I}}(y_I).\end{align*}
        The map $L_q$ does not involve $I$-dependent forms in its partition rank decomposition, so setting $\delta_q^{(y_{[k+1] \setminus I})}(z_I) = (L_q)_{y_{[k+1] \setminus I}}(z_I)$, we reach the desired conclusion.
        
        \indent On the other hand, the set of such $y_{[k+1] \setminus I}$ such that $(\gamma_{q,j}(y_{[k+1] \setminus I}))_{j\in t'_q} = \lambda$ and $y_{[k+1] \setminus I} \in (V_q)^{([k] \setminus I)} \times G_{k+1}$ is a non-empty multiaffine variety of codimension at most $t'_q + 2dr\leq (2d+1)r$, so it has size at least $p^{-k(2d+1)r}|G_{[k+1] \setminus I}|$.
    \end{proof}

    Let $s$ be a positive integer to be specified later. Take elements $e_{i, j} \in G_i$, for $i \in [k+1]\setminus I$, $j \in [s]$, uniformly and independently at random and define subspaces $E_i = \langle e_{i,1}, \dots, e_{i, s} \rangle$. By Lemma~\ref{randomproductintersection}, we have $\mathbb{P}(Y_{q,\lambda} \cap \prod_{i \in [k+1] \setminus I} E_i \not=\emptyset) \geq 1 - 8kp^{k(2d+1)r + 1 - s}$. Hence, taking $s = k(4d+2)r + 10k + r + \log_p \varepsilon^{-1}$, there exist such elements $e_{i, j}$ and subspaces $E_i$ and a collection of additive $2d$-tuples $\mathcal{Q}' \subseteq \mathcal{Q}$ of size $|\mathcal{Q}'| \geq |\mathcal{Q}| - 2^{-k}\varepsilon|G_0|^{2d-1}$ such that for each $q \in \mathcal{Q}'$, we have that $Y_{q,\lambda} \cap \prod_{i \in [k+1] \setminus I}E_i \not=\emptyset$ holds for all $\lambda \in \Lambda_q$. Now, define new varieties
    \[\tilde{V}_x = V_x \cap \Big(\bigcap_{j : I \to [s]} \Big\{z_I \in G_I : \tilde{\phi}_x\Big(z_I, (e_{i, j(i)} : i \in [k+1] \setminus I)\Big) = 0\Big\}\Big) \times G_{[k] \setminus I}.\]

    The codimension of $\tilde{V}_x$ is at most $r + s^k \leq (2r + \log \varepsilon^{-1})^{O(1)} \leq \log^{O(1)}(2\varepsilon^{-1}c^{-1})$. Finally, let us check that on these varieties we get desired identity. Fix any additive $2d$-tuple $q= x_{[2d]}$ in $\mathcal{Q}'$. By using a spanning set of $\Lambda_q$ and changing the basis, we may assume that vectors $(\gamma_{q, j}(y_{[k+1] \setminus I}))_{j\in t'_q}$ appearing in~\eqref{controlledprankeqonvars} actually attain all basis vectors of $\mathbb{F}_p^{t_q'}$.
    
    \indent Let $j \in [t'_q]$. Hence the standard basis vector $\lambda$ of $\mathbb{F}_p^{t_q'}$ having a single 1 at $j$\tss{th} coordinate appears in $\Lambda_q$, so by the definition of $\mathcal{Q}'$, $Y_{q,\lambda} \cap \prod_{i \in [k+1] \setminus I}E_i \not=\emptyset$. Let $y_{[k+1] \setminus I} = \Big(\sum_{j' \in [s]}\mu_{j, j'} e_{j ,j'} : j \in [k+1] \setminus I\Big)$ be a point in $Y_{q,\lambda} \cap \prod_{i \in [k+1] \setminus I}E_i$, for some scalars $\mu_{j, j'}$. Hence, by the conclusion of Claim~\ref{claimyqlrankdec}, we have
    \[\beta_{q,j}(z_I) = \sum_{i \in [2d]} (-1)^{i} \tilde{\phi}_{x_i}\Big(y_{[k+1]\setminus I} , z_I\Big) - \delta_q^{(y_{[k+1] \setminus I})}(z_I)\]
    for all $z_I \in (V_q)_{y_{[k] \setminus I}}$. In particular, when $z_{[k+1]} \in (\tilde{V}_{x_1} \cap \dots \cap \tilde{V}_{x_{2d}}) \times G_{k+1}$, then the $\tilde{\phi}_{x_i}$ terms in the identity above vanish and we have $\beta_{q,j}(z_I) = -\delta_q^{(y_{[k+1] \setminus I})}(z_I)$, completing the proof.
\end{proof}

\section{Multilinear Bogolyubov argument}

In this section we derive a multilinear Bogolyubov argument which is the third phase of our main proof. The goal is to ensure that the varieties in the domains of the system of MM-$r$-maps depend linearly on the indexing element. Recall that we say that $(V_x)_{x \in G_0}$ is a \textit{linear system} of mulitlinear varieties if each $V_x$ is a slice of a variety $V \subseteq G_{[0,k]}$ defined by a mixed-linear map whose all components depend linearly on $G_0$.

\begin{theorem}\label{multbogstep}
    Let $\varepsilon < \bsc$ and $d = 2^{2^{k+1}}$. Let $(\phi_x : V_x \to H)_{x \in G_0}$ be a system of MM-$r$-maps such that $(1-\varepsilon)|G_0|^{2d-1}$ additive $2d$-tuples are variety-respected. Then there exist a positive quantity $c \geq \exp(-(2r)^{O(1)})$ and a integer $s \leq (2r)^{O(1)}$, another system $(\psi_x : W\cap V'_x \to H)_{x \in X}$ of MM-$s$-maps such that $c|G_0|^3$ additive quadruples in $X$ are variety-respected, where $W \subseteq G_{[k]}$ is a multilinear variety, independent of $x$, and $V'_x\subseteq G_{[k]}$ depends linearly on $x$. Moreover, for each $x \in X$, we have 
    \[\bigg|Z\Big(\psi_x - \Big(\sum_{i \in [2d]} (-1)^i \phi_{y_i}\Big)\Big)\bigg| \geq c|G_{[k]}|\]
    for at least $c|G_0|^{2d-1}$ choices of $y_{[2d]}$ with $\sum_{i \in [2d]} (-1)^i y_i = x$.
\end{theorem}

Recall that for a collection of sets $\mathcal{D} \subseteq \mathcal{P}[k]$, we say that a multilinear variety is \textit{$\mathcal{D}$-dependent} if it is defined by forms whose coordinate sets belong to $\mathcal{D}$.

Our approach towards obtaining a linear behaviour in $G_0$ in the domain varieties of maps in the system will be based on convolutions in the following sense. At each step, for $a \in G_0$, we shall consider the map $\phi_{x+a} - \phi_x$ for $x\in G_0$. Note that $\phi_{x+a} - \phi_x$ is only defined on $V_{x + a} \cap V_x$. However, if we consider another choice $y$ instead of $x$, typically additive quadruple $(\upd{x+a}, x, y+a, \upd{y})$ is variety-respected so we have $\phi_{x+a} - \phi_x - \phi_{y+a} + \phi_y = 0$ on $V_{x+a} \cap V_x \cap V_{y+a} \cap V_y$. In other words, $\phi_{x+a} - \phi_x$ and $\phi_{y+a} - \phi_y$ coincide on the intersection of their domains, and we may try to take a simultaneous extension of $\phi_{x+a} - \phi_x$ and $\phi_{y+a} - \phi_y$ as a new map, whose domain is expected to behave linearly in $a$.

\indent Nevertheless, there is an obstacle to such a strategy coming from the general problem of extending multilinear maps, which is not present in $k = 1$ case, which was the subject of the previous paper~\cite{QPU4}. In that paper, we relied on the simple linear-algebraic fact that linear maps on subspaces $L_1 : U_1 \to H$ and $L_2 : U_2 \to H$, coinciding on the intersection of the domains, have a unique common extension on $U_1 + U_2$. This suggests the following approach when $U_1, U_2$ are varieties in $G_{[k]}$. Naively, we would need to consider the 'variety generated by $U_1$ and $U_2$' and extend the maps to that domain. However, such an approach is problematic for two reasons. Firstly, it is not obvious that multilinear maps can be extended directly and, even if they can, the best we can guarantee is extension to the smallest multilinear set (a set which is a subspace on each line in principal directions) containing $U_1 \cup U_2$. That would mean that we need to consider multilinear maps defined only on multilinear sets rather than varieties which is not a well-behaved category of maps, so we need a different approach. The proof requires some setting up, so we first introduce the relevant notation in Proposition~\ref{inductivemlbogstep}, and then explain the key ideas. Before proceeding with the proof, we derive an auxiliary result on a system of multilinear forms that is not quasirandom.

\subsection{Biased systems of multilinear forms}

In this subsection we consider a system of multilinear forms that is not quasirandom.

\begin{lemma}\label{biasedsystemmm}
    Let $X_1, \dots, X_4 \subseteq G_0$ be sets, and let $\alpha^{(i)}_{x} : G_{[k]} \to \mathbb{F}_p$ be a multilinear form for each $i \in [4]$ and $x \in X_i$. Let $\theta^{(1)}, \dots, \theta^{(3)} : G_{[0,k]} \to \mathbb{F}_p$ and $\gamma : G_{[k]} \to \mathbb{F}_p$ be multilinear forms. Suppose that
    \begin{equation}
        \on{bias}\Big(\alpha^{(1)}_{x+a} + \alpha^{(2)}_x + \alpha^{(3)}_{y+a} + \alpha^{(4)}_y + \theta^{(1)}_x + \theta^{(2)}_y + \theta^{(3)}_a + \gamma\Big) \geq c \label{biasedsystemeqn}
    \end{equation}
    holds for at least $c|G_0|^3$ additive quadruples $(x + a, x, y+a, y) \in X_1 \tdt X_4$. Then there exist further multilinear forms $\mu : G_{[0, k]} \to \mathbb{F}_p$ and $\mu' : G_{[k]} \to \mathbb{F}_p$ such that
    \[\on{bias}(\alpha^{(1)}_x - \mu_x - \mu') \geq \eplog{c}\]
    holds for at least $\eplog{c} |G_0|$ elements $x \in X_1$.
\end{lemma}

In the proposition above, the notation $\theta^{(i)}_x$ means the restricted multilinear form $z_{[k]} \mapsto \theta^{(i)}(x, z_{[k]})$.

\begin{proof}
    By defining $\tilde{\alpha}_x^{(1)} = \alpha_x^{(1)} + \theta^{(3)}_x + \gamma$, $\tilde{\alpha}_x^{(2)} = \alpha_x^{(2)} + \theta^{(1)}_x - \theta^{(3)}_x$, $\tilde{\alpha}_x^{(3)} = \alpha_x^{(3)}$ and $\tilde{\alpha}_x^{(4)} = \alpha_x^{(4)} + \theta^{(2)}_x$, condition~\eqref{biasedsystemeqn} simplifies to
    \[\on{bias}\Big(\tilde{\alpha}^{(1)}_{x+a} + \tilde{\alpha}^{(2)}_x + \tilde{\alpha}^{(3)}_{y+a} + \tilde{\alpha}^{(4)}_y \Big) \geq c.\]
    
    Furthermore, by triangle and Cauchy-Schwarz inequalities, we obtain
    \begin{align*}c^4 \leq &\Big|\exx_{x, y, a \in G_0} \id_{X_1}(x + a)\id_{X_2}(x) \id_{X_3}(y + a)\id_{X_4}(y) \exx_{x_{[k]} \in G_{[k]}} \\
    &\hspace{6cm}\omega\Big(\tilde{\alpha}^{(1)}_{x+a}(x_{[k]}) + \tilde{\alpha}^{(2)}_x(x_{[k]}) + \tilde{\alpha}^{(3)}_{y+a}(x_{[k]}) + \tilde{\alpha}^{(4)}_y(x_{[k]})\Big)\Big|^2\\
    \leq & \bigg(\exx_{a \in G_0, x_{[k]} \in G_{[k]}} \Big|\Big(\exx_{x \in G_0} \id_{X_1}(x + a)\id_{X_2}(x) \omega\Big(\tilde{\alpha}^{(1)}_{x+a}(x_{[k]}) + \tilde{\alpha}^{(2)}_x(x_{[k]})\Big) \Big)\\
    &\hspace{6cm}\Big(\exx_{y \in G_0} \id_{X_3}(y + a)\id_{X_4}(y) \omega\Big(\tilde{\alpha}^{(3)}_{y+a}(x_{[k]}) + \tilde{\alpha}^{(4)}_y(x_{[k]})\Big)\Big)\Big|\bigg)^2\\
    \leq & \exx_{a \in G_0, x_{[k]} \in G_{[k]}} \Big|\Big(\exx_{x \in G_0} \id_{X_1}(x + a)\id_{X_2}(x) \omega\Big(\tilde{\alpha}^{(1)}_{x+a}(x_{[k]}) + \tilde{\alpha}^{(2)}_x(x_{[k]})\Big) \Big) \Big|^2\\
    = & \exx_{x, y, a \in G_0} \id_{X_1}(x + a)\id_{X_2}(x) \id_{X_1}(y + a)\id_{X_2}(y)\\
    &\hspace{6cm}\exx_{x_{[k]} \in G_{[k]}} \omega\Big(\tilde{\alpha}^{(1)}_{x+a}(x_{[k]}) + \tilde{\alpha}^{(2)}_x(x_{[k]}) + \tilde{\alpha}^{(1)}_{y+a}(x_{[k]}) + \tilde{\alpha}^{(2)}_y(x_{[k]})\Big).
    \end{align*}

    Another round of such maneuvers leads to 
    \begin{equation}
        c^8 \leq \exx_{x, y, a \in G_0} \id_{X_1}(x + a)\id_{X_1}(x) \id_{X_1}(y + a)\id_{X_1}(y) \exx_{x_{[k]} \in G_{[k]}} \omega\Big(\tilde{\alpha}^{(1)}_{x+a}(x_{[k]}) + \tilde{\alpha}^{(1)}_x(x_{[k]}) + \tilde{\alpha}^{(1)}_{y+a}(x_{[k]}) + \tilde{\alpha}^{(1)}_y(x_{[k]})\Big).\label{diraddquadsmultbogaux}
    \end{equation}
    We simplify the notation and write $\alpha_x$ and $X$ instead of $\tilde{\alpha}^{(1)}_x$ and $X_1$.

    Define map $A : X \times G_{[k-1]} \to G_k$ so that for all $x \in X$ and $y_{[k]} \in G_{[k]}$ we have 
    \[A(x, y_{[k-1]}) \cdot y_k = \alpha_x(y_{[k]}).\]
    It is multilinear in coordinates $G_{[k-1]}$. By~\eqref{diraddquadsmultbogaux}, we have that $A$ respects at least $c^8 |G_0|^3 |G_{[k-1]}|$ additive quadruples in direction $G_0$. Averaging and applying Theorem~\ref{invhomm}, we may pass to a subset of domain $Y \subseteq X \times G_{[k-1]}$ of size $|Y| \geq \eplog{c} |G_{[0,k-1]}|$ where $A$ is a Freiman multi-homomorphism. Since we assume the main result (Theorem~\ref{strFmult}) for $k$ variables, there exists a global multiaffine map $\Psi : G_{[0, k-1]} \to G_k$ that coincides with $A$ at at least $\eplog{c} |G_{[0,k-1]}|$ points. By averaging, we find a subset $X' \subseteq X$ of size  $\eplog{c} |G_0|$ such that for each $x \in X'$, we have $A(x, y_{[k-1]}) = \Psi(x, y_{[k-1]})$ for at least $\eplog{c} |G_{[k-1]}|$ points $y_{[k-1]} \in G_{[k-1]}$. Note also that for such an $x$,
    \[A(x, y_1 - z_1, \dots, y_{k-1} - z_{k-1}) = \sum_{I \subseteq [k-1]} (-1)^{k-1-|I|} A(x, y_I, z_{[k-1] \setminus I}) = \sum_{I \subseteq [k-1]} (-1)^{k-1-|I|} \Psi(x, y_I, z_{[k-1] \setminus I})\]
    holds for $\eplog{c} |G_{[k-1]}|^2$ choices of $(y_{[k-1]}, z_{[k-1]})$ so we may assume that $\Psi$ is multilinear in coordinates $G_{[k-1]}$. The claim follows after averaging and setting $\mu(x, y_{[k]}) = \Psi(x, y_{[k-1]}) \cdot y_k - \Psi(0, y_{[k-1]}) \cdot y_k - \theta^{(3)}_x(y_{[k]})$ and $\mu'(y_{[k]}) = \Psi(0, y_{[k-1]}) \cdot y_k - \gamma(y_{[k]})$. 
\end{proof}

\subsection{Obtaining a linear system of varieties}

In this subsection we complete the multilinear Bogolyubov argument step. The proof will be inductive and we shall gradually linearize the behaviour of the domain varieties. The inductive step is formulated as the next proposition.

\begin{proposition}\label{inductivemlbogstep}
    Let $\mathcal{D} \subseteq \mathcal{P}[k]$ be a down-set and let $D_0 \leq O(1)$. Let $\theta_i : G_0 \times G_{I_i} \to \mathbb{F}_p$ and $\gamma_i : G_{J_i} \to \mathbb{F}_p$ be multilinear forms for $i \in [m]$, where $I_i, J_i \subseteq [k]$. For each $x \in G_0$, let $A_x, B_x \subseteq [m]$ be sets of size at most $r$ and let $V_x \subseteq G_{[k]}$ be a $\mathcal{D}$-dependent multilinear variety of codimension at most $r$. Write 
    \[C_x = \{y_{[k]} \in G_{[k]} : (\forall i \in A_x) \gamma_i(y_{J_i}) = 0\},\,\,L_x = \{y_{[k]} \in G_{[k]} : (\forall i \in B_x) \theta_i(x, y_{J_i}) = 0\},\,\, S_x = C_x \cap L_x \cap V_x\]
    and let $\phi_x : S_x \to H$ be a multilinear map for each $x \in X$, for some subset $X \subseteq G_0$. Let $I \subseteq [k]$ be a maximal set in $\mathcal{D}$.
    
    Suppose that, for each $d \in \{2, \dots, 2D_0\}$, $(1-\varepsilon)|G_0|^{2d-1}$ additive $2d$-tuples are variety-respected in the system $(\phi_x)_{x \in X}$. Then there exist positive integers $m' \leq \exp((2r \log(\varepsilon^{-1} m))^{O(1)})$ and $r' \leq (2r \log(\varepsilon^{-1}))^{O(1)}$, a subset $Y \subseteq G_0$ of size $|Y| \geq (1 - O(\varepsilon))|G_0|$ sets $A'_x, B'_x \subseteq [m']$ of size $r'$ for $x \in Y$, multilinear forms $\theta'_i : G_0 \times G_{I'_i} \to \mathbb{F}_p$ and $\gamma'_i : G_{J'_i} \to \mathbb{F}_p$ be multilinear forms for $i \in [m']$, $(\mathcal{D} \setminus \{I\})$-dependent multilinear varieties $V'_x$ of codimension at most $r'$, defining varieties
    \[S'_x = \{y_{[k]} \in G_{[k]} : (\forall i \in A'_x) \gamma'_i(y_{J'_i}) = 0\}\,\,\cap\,\, \{y_{[k]} \in G_{[k]} : (\forall i \in B'_x) \theta'_i(x, y_{J'_i}) = 0\} \cap V'_x\]
    and a system of MM-maps $(\psi_y : S'_y \to H)_{y \in Y}$ such that $Z|(\psi_y - (\phi_{x+y} - \phi_x))| \geq p^{-O(r')}$ on and for each $d \in \{2, \dots, D_0\}$, $(1-O(\varepsilon))|G_0|^{2d-1}$ additive $2d$-tuples are variety-respected in the system $(\psi_y : S'_y \to H)_{y \in Y}$.
\end{proposition}

\noindent\textbf{Remark on the notation.} During the proof, we maintain a list of multilinear forms depending on $G_0$, denoted $\theta_1, \dots, \theta_m$, and a list of multilinear forms independent of $G_0$, denoted $\gamma_1, \dots, \gamma_m$. Here, $m$ is somewhat large, but reasonable, namely we may think of it as comparable to $2^t$, where $t = O(r)$ is the codimension of the domains of the multilinear maps in the system. The current domains are intersections of three varieties. The first one is independent of $G_0$, denoted by $C_x$, where the letter $C$ is chosen to indicate the variety comes from \textit{common} forms, as they are shared between many $x$. The second one exhibits linear behaviour on $x$ and is (therefore) denoted $L_x$. The final one is $\mathcal{D}$-dependent, and this dependence decreases in each step. As the varieties $C_x$ and $L_x$ have codimension that is much smaller than the length lists, they are determined by the choice of forms in the list, which is the role of sets $A_x$ and $B_x$.

\noindent\textbf{Brief proof overview.} Out strategy to define $\psi_a$ is to find $x,y \in G_0$ such that $x+a,x, y+a, y$ is a variety-respected additive quadruple in $X$ and take $\psi_a$ as a common extension of $\phi_{x+a} - \phi_x$ and $\phi_{y+a} - \phi_y$. These two maps have domains
\[S_{x+a} \cap S_x = C_{x+a} \cap C_x \cap L_{x + a} \cap L_x \cap V_{x + a} \cap V_x\]
and
\[S_{y+a} \cap S_y = C_{y+a} \cap C_y \cap L_{y + a} \cap L_y \cap V_{y + a} \cap V_y.\]

Let $I \in \mathcal{D}$ be the given maximal set and let $\alpha_x : G_{I} \to \mathbb{F}_p^r$ be a multilinear map such that $V_x = V'_x \cap (\{\alpha_x = 0\} \times G_{[k] \setminus I})$ for a $(\mathcal{D} \setminus \{I\})$-supported multilinear variety $V'_x$. Write
\[D_{a, x} = (C_{x + a}\cap C_x) \cap (L_{x+a} \cap L_x) \cap ( V'_{x + a} \cap  V'_x) \subseteq G_{[k]}.\]

For each $x$, we shall find multilinear maps $\alpha^{\on{qr}}_x$, arising from the quasirandom linear combinations of components of $\alpha_x$, and $\alpha^{\on{str}}_x$, which is structured in the sense that its components are close to further forms depending linearly on $x$. Then, we extend $\phi_{x+a} - \phi_x$ and $\phi_{y+a} - \phi_y$ to a multilinear map $\tau_{a,x,y}$ defined on
\begin{align*}D_{a,x} \cap D_{a,y} \cap ((\{\alpha^{\on{str}}_{x +a} = 0\} \cap \{\alpha^{\on{str}}_{x} = 0\}\cap \{\alpha^{\on{str}}_{y +a} = 0\} \cap \{\alpha^{\on{str}}_{y} = 0\}) \times G_{[k] \setminus I}).\end{align*}

To obtain the final maps, we shall pick $O(1)$ choices of $x$ and $y$ and define a new map for element $a \in G_0$ as $\tau_{a,x,y}$ for some $x,y$ among the chosen ones, and we also need to remove $I$-dependent varieties $\{\alpha^{\on{str}}_{x+a} = 0\}\cap \dots \cap \{\alpha^{\on{str}}_{y} = 0\}$.

\indent Let us turn to details. 
    
\begin{proof}[Proof of Proposition~\ref{inductivemlbogstep}] Let $\eta > 0$ be a parameter to be chosen later. 

    \begin{claim}\label{mlbogregularization}
        We may find a positive integer $m' \leq \exp\Big((2r\log (\eta^{-1}m))^{O(1)}\Big)$, multilinear maps $\theta'_i : G_0 \times G_I \to \mathbb{F}_p$, $\gamma'_i : G_I \to \mathbb{F}_p$ for $i \in [m']$ and, for each $x \in G_0$, independent elements $\lambda_{x, 1}, \dots, \lambda_{x, s_x}$ of $\mathbb{F}_p^r$ such that, writing $\Lambda_x = \langle \lambda_{x, 1}, \dots, \lambda_{x, s_x} \rangle$, we have
        \begin{itemize}
            \item[\textbf{(i)}] for each $i \in [s_x]$ there exists $j \in [m']$ such that $\on{bias} \Big(\lambda_{x, i} \cdot \alpha_x - (\theta'_j)_x - \gamma'_j\Big) \geq \exp(-\log^{O(1)}(2\eta^{-1}))$,
            \item[\textbf{(ii)}] for all but at most $\eta|G_0|^3$ additive quadruples $(x+a,x,y+a,y)$ in $X$, when $\mu, \mu', \nu,\nu' \in \mathbb{F}_p^r$ are linear combinations such that $\mu\notin\Lambda_{x+a}$, $\mu'\notin\Lambda_{x}$, $\nu\notin\Lambda_{y+a}$ or $\nu'\notin\Lambda_{y}$, for all directions $\kappa \notin I$ the multilinear form $\mu  \cdot \alpha_{x+a} + \mu' \cdot \alpha_{x+a} + \nu \cdot \alpha_{y+a} +  \nu' \cdot \alpha_{y})$ is $\eta$-quasirandom with respect to $(D_{a,x} \cap D_{a,y})_{z_\kappa} \cap (D_{a,x} \cap D_{a,y})_{z'_\kappa}$ for all but at most $\eta |G_\kappa|^2$ pairs $(z_\kappa, z'_\kappa) \in G_\kappa^2$,
            \item[\textbf{(iii)}] for $d \in [2D_0$ and all but at most $\eta|G_0|^{6d-1}$ of $(a_{[2d]}, x_{[2d]}, y_{[2d]})$ with $\sum_{i \in [2d]} (-1)^i a_i = 0$ and $x_i + a_i, x_i, y_i +a_i, y_i \in X$, when $\mu^{(i)}, {\mu'}^{(i)}, {\nu}^{(i)}, {\nu'}^{(i)} \in \mathbb{F}_p^r$, $i \in [2d]$, are linear combinations such that for some $i \in [2d]$ at least one $\mu^{(i)}\notin\Lambda_{x_i+a_i}$, ${\mu'}^{(i)}\notin\Lambda_{x_i}$, ${\nu}^{(i)}\notin\Lambda_{y_i+a_i}$ or ${\nu'}^{(i)}\notin\Lambda_{y_i}$ holds, the form
            \[\sum_{i \in [2d]}\mu^{(i)} \cdot \alpha_{x_i + a_i} + {\mu'}^{(i)} \cdot \alpha_{x_i} + \nu^{(i)} \cdot \alpha_{y_i + a_i} + {\nu'}^{(i)} \cdot \alpha_{y_i}\]
            is $\eta$-quasirandom with respect to $D_{a_1, x_1} \cap D_{a_1, y_1} \cap D_{a_2, x_2} \cap \dots \cap D_{a_{2d}, y_{2d}}$.
        \end{itemize}
    \end{claim}

    We call an additive quadruple satisfying the condition \textbf{(ii)} and a $6d$-tuple satisfying the condition \textbf{(iii)} \textit{good}.

    \begin{proof}
        We iteratively add new multilinear forms until the conditions \textbf{(ii)} and \textbf{(iii)} holds. Suppose that one of them fails.

        \begin{claim}
            If condition \textbf{(ii)} or \textbf{(iii)} fails, then we have  multilinear forms
             $T, T', T'' : G_{\{0\} \cup I} \to \mathbb{F}_p$ and $\Gamma: G_I \to  \mathbb{F}_p$ and $\mu, \mu', \nu, \nu' \in \mathbb{F}_p^r$, such that 
            \[\on{bias} \Big(\mu  \cdot \alpha_{x+a} + \mu' \cdot \alpha_{x+a} + \nu \cdot \alpha_{y+a} +  \nu' \cdot \alpha_{y} + T_{a} + T'_{x} + T''_{y} + \Gamma\Big) \geq \eta^{O(1)}\]
            and $\mu \notin \Lambda_{x+a}$ hold for at least $(pm)^{-O(r)}\eta^{O(1)}|G_{0}|^3$ choices of $(a,x,y) \in G_0^3$.
        \end{claim}

        \begin{proof}
            Suppose first that the condition \textbf{(ii)} fails. Let us examine the structure of forms involved in $D_{a,x}$. We have $D_{a,x} = (C_{x + a}\cap C_x) \cap (L_{x+a} \cap L_x) \cap ( V'_{x + a} \cap  V'_x)$, so $D_{a,x}$ has codimension at most $6r$ and involves forms $\gamma_i$ for $i \in A_{x+a} \cup A_x$, $\theta_i$ for $i \in B_{x+a} \cup B_x$ and $(\mathcal{D}\setminus \{I\})$-supported forms defining $V_{x+a} \cap V_x$. Hence, for $a, x,y, z_\kappa, z'_\kappa$, if the given multilinear form $M = \mu  \cdot \alpha_{x+a} + \mu' \cdot \alpha_{x+a} + \nu \cdot \alpha_{y+a} +  \nu' \cdot \alpha_{y}$ is not $\eta$-quasirandom with respect $(D_{a,x} \cap D_{a,y})_{z_\kappa} \cap (D_{a,x} \cap D_{a,y})_{z'_\kappa}$, then, for some  we have a linear combination $\tilde{\gamma}$ of $(\gamma_i)_{z_\kappa}$, $(\gamma_i)_{z'_\kappa}$ for $i \in A_{x + a} \cup A_x \cup A_{y+a} \cup A_y$, a linear combination $\tilde{\theta}$ of $(\theta_i)_{z_\kappa}$, $(\theta_i)_{z'_\kappa}$ for $i \in B_{x + a} \cup B_x \cup B_{y+a} \cup B_y$, and a linear combination $\psi$ of $(\mathcal{D}\setminus \{I\}) \cap \mathcal{P}([k] \setminus \{\kappa\})$-supported forms such that
            \[\Big|\exx_{w_{[k] \setminus \{\kappa\}}} \omega\Big(M(w_I) + \tilde{\theta}(w_{[k] \setminus \{\kappa\}}) +\tilde{\gamma}(w_{[k] \setminus \{\kappa\}}) + \psi(w_{[k] \setminus \{\kappa\}})\Big)\Big| \geq \eta.\]

            By Gowers-Cauchy-Schwarz inequality, we have
            \[\exx_{w_{[k] \setminus \{\kappa\}}} \omega\Big(M(w_I) + \tilde{\theta}'(w_{[k] \setminus \{\kappa\}}) +\tilde{\gamma}'(w_{[k] \setminus \{\kappa\}}))\Big) \geq \eta^{2^k},\]
            where $\tilde{\theta}'$ arises from $\tilde{\theta}$ by removing multilinear parts that do not depend on all variables in $I$, and similarly $\tilde{\gamma}'$ arises from $\tilde{\gamma}$. In particular, $\tilde{\theta}'$ is a linear combination of at most $8r$ forms among $\theta_1, \dots, \theta_m$, restricted at $x+a,x,y+a$ or $y$ and $z_\kappa$ or $z_{\kappa}'$ and similarly for $\tilde{\gamma}'$, so there are at most $(pm)^{O(r)}$ choices for $\tilde{\gamma}'$ and $\tilde{\theta}'$. By averaging over such forms and $z_\kappa$ and $z'_\kappa$, we get forms $T, T', T'' : G_{[0, k]} \to \mathbb{F}_p$, linear in coordinates $G_{\{0\} \cup I}$ and $\Gamma: G_{[k]} \to  \mathbb{F}_p$, linear in coordinates $G_I$, such that 
            \[\on{bias} \Big(\mu  \cdot \alpha_{x+a} + \mu' \cdot \alpha_{x+a} + \nu \cdot \alpha_{y+a} +  \nu' \cdot \alpha_{y} + T_{a, t_{[k] \setminus I}} + T'_{x, t_{[k] \setminus I}} + T''_{y, t_{[k] \setminus I}} + \Gamma_{t_{[k] \setminus I}}\Big) \geq \eta^{2^k}\]
            holds for at least $(pm)^{-O(r)}\eta^{O(1)}|G_0|^3|G_{[k] \setminus I}|$ choices of $(a,x,y, t_{[k] \setminus I})$. Average over $t_{[k] \setminus I}$ to reach the desired conclusion.

            For the failure of condition \textbf{(iii)}, analogous arguments apply to give multilinear maps $T_i, T'_i, T''_i : G_{[0, k]} \to \mathbb{F}_p$ and $\Gamma_i: G_{I} \to  \mathbb{F}_p$, $i \in [2d]$ such that

            \[\on{bias}\Big(\sum_{i \in [2d]}\mu^{(i)} \cdot \alpha_{x_i + a_i} + {\mu'}^{(i)} \cdot \alpha_{x_i} + \nu^{(i)} \cdot \alpha_{y_i + a_i} + {\nu'}^{(i)} \cdot \alpha_{y_i} + (T_i)_{a_i} + (T''_i)_{x_i} + (T''_i)_{y_i} + \Gamma_i\Big) \geq \eta.\]
            
            Two applications of the Cauchy-Schwarz inequality allow us to obtain the desired conclusion.
        \end{proof}

        Lemma~\ref{biasedsystemmm} applies to complete the iteration step to give new maps $\theta'_{m'+1}$ and $\gamma'_{m'+1}$ so we move the appropriate $\mu$ into the sequence $\lambda_{x,1}, \dots, \lambda_{x, s_x}$ for at least $\exp\Big(-(2r\log (\eta^{-1}m))^{O(1)}\Big)|G_0|$ elements $x \in G_0$. The procedure thus terminates in $\exp\Big((2r\log (\eta^{-1}m))^{O(1)}\Big)$ steps.
    \end{proof}

    Extend the sequence  $ \lambda_{x, 1}, \dots, \lambda_{x, s_x}$ to a basis with additional vectors $ \mu_{x, 1}, \dots, \mu_{x,r- s_x}$ and define multilinear maps $\alpha^{\on{qr}}_x : G_I \to \mathbb{F}_p^{r - s_x}$ and $\alpha^{\on{str}}_x : G_I \to \mathbb{F}_p^{s_x}$ by $\alpha^{\on{qr}}_{x,i} = \mu_{x, i} \cdot \alpha_x$ and $\alpha^{\on{str}}_{x,i} = \lambda_{x, i} \cdot \alpha_x$. Define multilinear varieties in $G_I$ by $Q_{a,x} = \{\alpha^{\on{qr}}_{x+a} = 0\} \cap \{\alpha^{\on{qr}}_x = 0\}$ and $R_{a,x} = \{\alpha^{\on{str}}_{x+a} = 0\} \cap \{\alpha^{\on{str}}_x = 0\}$.

    We are now ready to carry out the extension procedure.

    \begin{claim}
        Suppose that $\eta \leq \varepsilon p^{-\blc r}$. For all but at most $O(\varepsilon|G_0|^3)$ choices of $(a,x,y) \in G_0^3$, we have $x+a,x,y+a, y \in X$ and there exist $\tau_{a,x,y} : D_{a,x} \cap D_{a,y} \cap ((R_{a, x} \cap R_{a, y}) \times G_{[k] \setminus I}) \to H$ such that $\tau_{a,x,y} = \phi_{x+a} - \phi_x$ on $D_{a,x} \cap D_{a,y} \cap ((R_{a, x} \cap R_{a, y} \cap Q_{a, x}) \times G_{[k] \setminus I})$ and $\tau_{a,x,y} = \phi_{y+a} - \phi_y$ on $D_{a,x} \cap D_{a,y} \cap ((R_{a, x} \cap R_{a, y} \cap Q_{a,y}) \times G_{[k] \setminus I})$.\\
        \indent Moreover, for $d \in \{2, \dots, 2D_0\}$, for all but at most $O(\varepsilon |G_0|^{6d - 1})$ choices of $(a_{[2d]}, x_{[2d]}, y_{[2d]}) \in G_0^{6d}$ such that $\sum_{i \in [2d]} (-1)^i a_i = 0$, we have the maps $\tau_{a_i,x_i,y_i}$ defined and 
        \[\sum_{i \in [2d]} (-1)^i\tau_{a_i,x_i,y_i} = 0\text{ on }\cap_{i \in [2d]} D_{a_i, x_i} \cap D_{a_i, y_i} \cap ((R_{a_i, x_i} \cap R_{a_i, y_i}) \times G_{[k] \setminus I}).\]
    \end{claim}

    \begin{proof}
        Define multilinear map $\delta_{a, x} = \phi_{x+a} - \phi_x$ on the domain $D_{a, x} \cap R_{a,x} \cap Q_{a, x}$.

        \indent Let $x, y \in G_0$ be such that the additive quadruple $(x+a, x, y+a, y)$ is good in the of Claim~\ref{mlbogregularization} and variety-respected. If $\eta \leq \varepsilon$, this holds for all but at most $O(\varepsilon|G_0|^3)$ additive quadruples in $G_0$. Since the quadruple is good, for all directions $\kappa \notin I$ map $(\alpha^{\on{qr}}_{x+a}, \alpha^{\on{qr}}_{x}, \alpha^{\on{qr}}_{y+a}, \alpha^{\on{qr}}_{y})$ is $\eta$-quasirandom with respect to $(D_{a,x} \cap D_{a,y})_{z_\kappa} \cap (D_{a,x} \cap D_{a,y})_{z'_\kappa} \cap ((R_{a,x} \cap R_{a,y}) \times G_{[k] \setminus I})$ for all but at most $\eta |G_\kappa|^2$ pairs $(z_\kappa, z'_\kappa) \in G_\kappa^2$. Then, provided $\eta \leq p^{-\blc r}$, Corollary~\ref{simextcorsinglevar} gives multilinear maps $\tau_{a, x,y} : D_{a,x} \cap D_{a, y} \cap ((R_{a,x} \cap R_{a,y}) \times G_{[k] \setminus I})\to H$ such that 
        \begin{align*}\tau_{a, x,y} = &\delta_{a, x}\text{ on }D_{a,x}\cap D_{a,y} \cap ((R_{a,x} \cap R_{a,y} \cap Q_{a, x}) \times G_{[k] \setminus I})\text{ and}\\
        \tau_{a, x,y} = &\delta_{a,y}\text{ on }D_{a,x} \cap D_{a, y} \cap ((R_{a,x} \cap R_{a,y} \cap Q_{a, y}) \times G_{[k] \setminus I}).\end{align*}

        Now consider $(a_{[2d]}, x_{[2d]}, y_{[2d]}) \in G_0^{6d-1}$ such that $\sum_{i \in [2d]} (-1)^i a_i = 0$ and such that $\tau_{a_i,x_i,y_i}$ are all defined. Assume also that additive $(4d)$-tuples $(x_1 + a_1,  x_1, \dots, x_{2d} + a_{2d}, x_{2d})$ and $(y_{2d} + a_{2d}, y_{2d}, \dots, y_{2d} + a_{2d}, y_{2d})$ are variety-respected. Then we have
        \[\sum_{i \in [2d]} \Big(\phi_{x_i + a_i} - \phi_{x_i}\Big) = 0\text{ on }\cap_{i \in [2d]} (D_{a_i, x_i} \cap ((R_{a_i, x_i} \cap Q_{a_i, x_i}) \times G_{[k] \setminus I})).\]
        Writing $D' = \cap_{i \in [2d]} D_{a_i, x_i} \cap (R_{a_i, x_i} \times G_{[k] \setminus I})$ and $D'' = \cap_{i \in [2d]} D_{a_i, y_i} \cap R_{a_i, y_i}$, we have
        \[\sum_{i \in [2d]}  (-1)^i \tau_{a_i,x_i,y_i} = 0\text{ on }D' \cap D'' \cap  \Big(\cap_{i \in [2d]} Q_{a_i, x_i}  \times G_{[k] \setminus I}\Big),\]
        and similarly
        \[\sum_{i \in [2d]}  (-1)^i \tau_{a_i,x_i,y_i} = 0\text{ on }D' \cap D'' \cap  \Big(\cap_{i \in [2d]} Q_{a_i, y_i}  \times G_{[k] \setminus I}\Big).\]
        
        By Theorem~\ref{vanishingmmmaps}, provided $\eta \leq p^{-\blc r}$ and $(a_{[2d]}, x_{[2d]}, y_{[2d]})$ is good, we have
        \[\sum_{i \in [2d]}  (-1)^i \tau_{a_i,x_i,y_i} = 0\text{ on }D' \cap D'',\]
        as desired.
    \end{proof}

    We need another claim in the similar spirit that tells us that the maps $\tau_{a,x,y}$ agree on the intersection of domains for most of choices of $(x,y)$. Write $\tilde{D}_{a,x,y} = D_{a,x} \cap D_{a,y} \cap ((R_{a, x} \cap R_{a, y}) \times G_{[k] \setminus I})$, which is the domain of $\tau_{a,x,y}$. 

    \begin{claim}
        Suppose that $\eta \leq \varepsilon p^{-\blc r}$. Then for all but at most $O(\varepsilon|G_0|^{4D_0+1})$ choices of $(a, x_{[2D_0]}, y_{[2D_0]})$ the maps $\tau_{a, x_i, y_i}$, $i \in [2D_0]$, are defined and they all agree on the intersection $\cap_{i \in [2D_0]}\tilde{D}_{a,x_i,y_i} $.
    \end{claim}

    \begin{proof}
        Provided $\eta \leq \varepsilon p^{-\blc r}$, by Claim~\ref{mlbogregularization} and Cauchy-Schwarz inequality, we have that, for all pairs $(i,j)$, $i \not = j$, in $[2D_0]$, the map $(\alpha^{\on{qr}}_{x_i+a},$ $\alpha^{\on{qr}}_{x_i},$ $ \alpha^{\on{qr}}_{x_j+a},$ $\alpha^{\on{qr}}_{x_j},$ $\alpha^{\on{qr}}_{y_i+a},$ $\alpha^{\on{qr}}_{y_i},$ $ \alpha^{\on{qr}}_{y_j+a},$ $\alpha^{\on{qr}}_{y_j})$ is $\eta$-quasirandom with respect to $\tilde{D}_{a,x_i,y_i} \cap \tilde{D}_{a,x_j,y_j}$. Moreover, for all but at most $O(\varepsilon|G_0|^{2D_0+1})$ choices of $(a, x_{[2D_0]}, y_{[2D_0]})$, the  additive quadruples $(x_i + a, x_i, x_j + a, x_j)$ and $(y_i + a, y_i, y_j + a, y_j)$ for $i, j \in [2D_0]$ are all variety-respected. Hence, $\tau_{a, x_i, y_i} = \phi_{x_i+a} - \phi_{x_i} = \phi_{x_j+a} - \phi_{x_j} = \tau_{a, x_j, y_j}$ on $\tilde{D}_{a,x_i,y_i} \cap \tilde{D}_{a,x_j,y_j} \cap ((Q_{a, x_i}\cap Q_{a, x_j}) \times G_{[k] \setminus I})$. Similarly, $\tau_{a, x_i, y_i} = \phi_{y_i+a} - \phi_{y_i} = \phi_{y_j+a} - \phi_{y_j} = \tau_{a, x_j, y_j}$ on $\tilde{D}_{a,x_i,y_i} \cap \tilde{D}_{a,x_j,y_j} \cap ((Q_{a, y_i}\cap Q_{a, y_j}) \times G_{[k] \setminus I})$. By Theorem~\ref{vanishingmmmaps}, provided $\eta \leq p^{-\blc r}$, we get $\tau_{a, x_i, y_i} = \tau_{a, x_j, y_j}$ on $\tilde{D}_{a,x_i,y_i} \cap \tilde{D}_{a,x_j,y_j}$ for all $i, j \in [2D_0]$, as desired.
    \end{proof}
    
    \indent By averaging, there exist $x_{[2D_0]}, y_{[2D_0]}$ and a set $A \subseteq G_0$ such that the we have 
    
    \begin{itemize}
        \item for each $a \in A$, the maps $\tau_{a, x_i, y_i}$ are defined and equal on $U_a := \cap_{i \in [2D_0]} \tilde{D}_{a,x_i,y_i} $, thus giving a well-defined map $\psi_a : U_a \to H$, 
        \item for each $d \in [2, D_0]$, all but at most $O(\varepsilon |G_0|^{2d-1})$ additive $(2d)$-tuples in $A$ are variety-respected in the system $(\psi_a: U_a \to H)_{a \in A}$. 
    \end{itemize}

    To finish the proof, we need to modify the varieties in the desired shape. We take $\eta \geq \varepsilon^{O(1)} p^{-O(r)}$ so that required bounds hold. For each $i \in [s_x]$, apply Theorem~\ref{arankprank} to find a $(\mathcal{P}(I) \setminus \{I\})$-supported multilinear variety $V''_x\subseteq G_I$ of codimension at most $\log^{O(1)}(2\eta^{-1})$ such that for each $i \in [r]$ there is $j = j(x,i) \in [m']$ $(\alpha^{\on{str}}_x)_i(y_I) = (\theta'_j)_x(y_I) + \gamma'_j(y_I)$ when $y_I \in V''_x$. Note that $(\mathcal{P}(I) \setminus \{I\}) \subseteq \mathcal{D} \setminus \{I\}$.

    \indent Finally, we pass to subvarieties in the domains. Note that $R_{a,x}= \{\alpha^{\on{str}}_{x+a}\} \cap \{\alpha^{\on{str}}_x = 0\}$ contains the variety
    \begin{align*}V''_{a+x} \cap V''_x \cap &\bigg(\Big(\bigcap_{i \in [s_{x+a}]} \{\theta'_{j(x+a, i)} = 0\}_{x+a}\, \cap\, \bigcap_{i \in [s_x]} \{\theta'_{j(x, i)} = 0\}_x\Big) \times G_{[k] \setminus I}\bigg) \\ 
    \cap &\bigg(\Big(\bigcap_{i \in [s_{x+a}]} \{\gamma'_{j(x+a, i)} = 0\} \,\cap\, \bigcap_{i \in [s_x]} \{\theta'_{j(x, i)} = 0\}\Big)\times G_{[k] \setminus I}\bigg).\end{align*}
    Define the set of indices in $[m']$
     \begin{align*}A'_a = &\{j(a + x_i, i') : i \in [2D_0D_0] , i' \in [s_{a + x_i}]\} \cup \{j(x_i, i') : i \in [2D_0D_0] , i' \in [s_{x_i}]\}\, \cup\\
     &\{j(a + y_i, i') : i \in [2D_0D_0] , i' \in [s_{a + y_i}]\} \cup \{j(y_i, i') : i \in [2D_0D_0] , i' \in [s_{y_i}]\}.\end{align*}
    Thus, $\psi_a$ may be restricted to the intersection of 
    \begin{align*}&\cap_{i \in [2D_0D_0]} \Big(C_{a+x_i} \cap C_{x_i} \cap C_{a+y_i} \cap C_{y_i}\Big) \\
    &\hspace{1cm}\cap \Big(\{z_{[k]} \in G_{[k]} : (\forall j \in B_{x_i + a} \cup B_{x_i}) \theta_j(x_i, z_{J_i}) = 0\} \Big)\\
    &\hspace{1cm}\cap  \Big(\{z_{[k]} \in G_{[k]} : (\forall j \in B_{y_i + a} \cup B_{y_i}) \theta_j(y_i, z_{J_i}) = 0\} \Big)\\
    &\hspace{1cm}\cap\Big(\cap_{j \in A'_a}\{\gamma'_j = 0\} \times G_{I^c}\Big) \cap \Big(\cap_{j \in A'_a}\{(\theta'_j)_{x_i} = 0\} \times G_{I^c}\Big)\cap \Big(\cap_{j \in A'_a}\{(\theta'_j)_{y_i} = 0\} \times G_{I^c}\Big)\end{align*}
    giving the new common variety,
    \[\Big(\cap_{i \in [2d]} \{z_{[k]} \in G_{[k]} : (\forall j \in B_{x_i + a} \cup B_{y_i + a}) \theta_i(a, z_{J_i}) = 0\} \Big) \cap \Big(\cap_{j \in A'_a}\{(\theta'_j)_{a} = 0\} \times G_{I^c}\Big)\]
    giving the new linearly-varying variety, and
    \[\cap_{i \in [2d]} \Big(V'_{a+x_i} \cap V'_{x_i} \cap V'_{a+y_i} \cap V'_{y_i} \cap V''_{a+x_i} \cap V''_{x_i} \cap V''_{a+y_i} \cap V''_{y_i}\Big),\]
    which is $\mathcal{D}\setminus \{I\}$-dependent, as required.
\end{proof}

We may now easily deduce Theorem~\ref{multbogstep}.

\begin{proof}[Proof of Theorem~\ref{multbogstep}]
    Apply Proposition~\ref{inductivemlbogstep} iteratively until the varieties $V_x$ become $\emptyset$-supported and therefore $G_{[k]}$. We begin with some $\varepsilon \geq \Omega(1)$ that results the procedure results in a system where at least $\frac{1}{2}|G_0|^3$ additive quadruples are variety respected. Thus, at the end of the procedure, we have $(\psi_x)_{x \in X}$, with domains $C_x \cap L_x$ defined by sets of indices $A_x, B_x$ of size $r' \leq (2r)^{O(1)}$. Take random sets $A_0, B_0 \subseteq [m]$ of size $4r'$ and let $X' \subseteq X$ be the set of all $x$ such that $A_x \subseteq A_0, B_x \subseteq B_0$. The probability that a respected additive quadruple $(x + a, x, y+a, y)$ remains in $X'$ is the probability that $A_{x+a} \cup A_x \cup A_{y+a} \cup A_y \subseteq A_0$ and $B_{x+a} \cup B_x \cup B_{y+a} \cup B_y \subseteq B_0$, which is at least $m^{-O(r')}$, from which the theorem follows.
\end{proof}

\section{Obtaining a multilinear map on a variety in $G_{[0,k]}$}

In this section, we carry out the final phase of the proof. The work here can be summarized as the following theorem, stating that the system of MM-$r$-maps obtained so far can be strengthened to give a multilinear map on a multilinear variety inside $G_{[0,k]}$.

Recall that by a \textit{linear system of varieties of codimension $r$}, we mean a collection of forms $\alpha_i : G_{\{0\} \cup I_i} \to \mathbb{F}_p$ where $I_i \subseteq [k]$, $i \in [r]$, and varieties $V_x = \{y_{[k]} \in G_{[k]} : (\forall i \in [r])\,\, \alpha_i(x, y_{I_i}) = 0\}$.

\begin{theorem}\label{phase4mainres}
    Let $(\phi_x : W \cap V_x \to H)_{x \in X}$ be a system of MM-$r$-maps, where $W \subseteq G_{[k]}$ is a multilinear variety and $(V_x)_{x \in G_0}$ is linear system of varieties, such that $c|G_0|^{3}$ additive quadruples are variety-respected. Then there exist a multilinear variety $U \subseteq G_{[0, k]}$ of codimension $(r + \log c^{-1})^{O(1)}$ and a map $\Phi : U \to H$, affine in direction $G_0$ and multilinear in directions $G_{[k]}$, such that for each $x \in U^{(\{0\})}$, there exists at least $\exp(-(r + \log c^{-1})^{O(1)})|G_0|^{15}$ 16-tuples $y_{[16]}$ in $X$ such that $\sum_{i \in [16]} (-1)^i y_i = x$ and
    \[\Phi_x = \sum_{i \in [16]} (-1)^i \phi_{y_i}\]
    holds on the variety $U^{[k]} \cap U_x \cap (\cap_{i \in [16]} U_{y_i})$.
\end{theorem}

Each of the four steps has its own subsection.

\subsection{Abstract Balog-Szemer\'edi--Gowers step 2}

%  Already proved above!
%
%\begin{lemma}
%     Let $(V_x)_{x \in G_0}$ be a linear system of multilinear varieties of codimension $r$. Let $s,d, \varepsilon$ be given. Then there exists a variety $W \subseteq G_{[k]}$ of codimension at most ?? such that for any MM-$s$-map $\psi : U \to H$, $U \subseteq G_{[k]}$ such that
%     \[\psi|_{U \cap V_{x_1} \cap \dots \cap V_{x_d}} = 0\]
%     holds for at least $\varepsilon|G_0|^d$ choices of $x_1, \dots, x_d \in G_0$, we have $\psi = 0$ on $U \cap W$.
% \end{lemma}

% \begin{proof}
%     When maps are already defined, things are simpler:
%     \begin{itemize}
%         \item Regularize the forms $\alpha_i$, by marking the non-qr linear combinations.
%         \item Quasirandom linear combinations can be removed in the simultaneous extension using Corollary~\ref{vanishingmmmaps}.
%         \item Biased contribute to $W$.
%     \end{itemize}

%     Let $\alpha_i : G_{\{0\} \cup I_i} \to \mathbb{F}_p$, $i \in [r]$, be the multilinear forms defining the linear system of varieties. Let $\beta_i : G_{J_i} \to \mathbb{F}_p$ be multilinear forms for $j \in [s]$, defining the variety $U$.\\

%     \noindent\textbf{Regularization.} We mark the linear combinations $\lambda \cdot \alpha$ of large bias, as long as there are such.\\

%     \noindent\textbf{Removal of quasirandom components.}\\

%     \noindent\textbf{Passing the $G_0$-independent variety.}
% \end{proof}

We begin with the second Balog-Szemer\'edi--Gowers step which again serves to ensure that, after passing to a large subset of indexing elements, all additive $O(1)$-tuples are respected up to an error function in a reasonably short list.

\begin{proposition}\label{ph4step1}
    Let $d = O(1)$. Let $(\phi_x : W \cap V_x \to H)_{x \in X}$ be a system of MM-$r$-maps, where $W \subseteq G_{[k]}$ is a multilinear variety and $V_x$ depend linearly on $x$, such that $c|G_0|^{3}$ additive quadruples are variety-respected. Then there exist an integer $r' \leq (r  + \log(2c^{-1}))^{O(1)}$, a further variety $W' \subseteq G_{[k]}$ of codimension $r'$, a subset $X' \subseteq X$ of density $\exp(-r')$, a linear system of varieties $(V'_a)_{a \in G_0}$ of codimension $r'$ and multilinear maps $\psi_1, \dots, \psi_m : W' \to H$ such that for each additive $2d$-tuple $x_{[2d]}$ in $X'$ we have some $j \in [m]$ 
    \[\sum_{i \in [2d]} (-1)^i \phi_{x_i} - \psi_j\Big|_{W' \cap V'_{x_1} \cap \dots \cap V'_{x_{2d}}} = 0.\]
\end{proposition}

We shall apply the proposition with $d= 32$, giving respectedness of all additive 64-tuples.

\begin{proof}
    We shall find suitable multilinear varieties $W_1, \dots, W_{36} \subseteq G_{[k]}$ and
    apply the abstract Balog-Szemer\'edi--Gowers theorem (Theorem~\ref{absg}), for the choice $A_{\on{aBSG}} = X$ to the sets $\mathcal{Q}_i$ consisting of additive quadruples $(\upd{a}_1, a_2, a_3, \upd{a}_4)$ in $X$ such that 
    \[\phi_{a_1} - \phi_{a_2} - \phi_{a_3} + \phi_{a_4} = 0\text{ on }W_i\cap V_{a_1} \cap V_{a_2} \cap V_{a_3} \cap V_{a_4}.\]
    Define $c' = (c/2)^{O(1)}$ as the sufficiently small parameter so that Theorem~\ref{absg} holds.

    \noindent\textbf{Defining varieties $W_1, \dots, W_{36}$.} We set $W_1 = W$, of codimension $s_1 = r$, and define other varieties $W_i$ of codimension $s_i$ inductively, using Theorem~\ref{lin-sys-vanishing}. The codimension bounds will form an increasing sequence. Suppose that $W_1, \dots, W_i$ have been defined thus far. Apply Theorem~\ref{lin-sys-vanishing} to the linear system of varieties $(V_x)_{x \in G_0}$, codimension parameter $4r + 2s_i$ and density parameter $c'$ to get a further multilinear variety $W_{i+1}$ of codimension $s_{i + 1} \leq (r + s_i + \log {c'}^{-1})^{O(1)} \leq (r + \log c^{-1})^{O(1)}$.

    \noindent\textbf{Checking conditions of the abstract Balog-Szemer\'edi--Gowers theorem.} Setting $W_1 = W$ ensures the largeness condition. Symmetry is obvious. It remains to check weak-transitivity. Suppose that for an additive quadruple $(\upd{x+a}, x, z+a,\upd{z})$ in $X$, there exist at least $c' |G_0|$ choices of $y$ such that $y, y+a \in X$ and $(\upd{x+a}, x, y + a, \upd{y}) \in \mathcal{Q}_i$ and $(\upd{y+a}, y, z+a, \upd{z}) \in \mathcal{Q}_j$. Thus,
    \[\phi_{x+a} - \phi_x + \phi_z - \phi_{z + a} = \Big(\phi_{x+a} - \phi_x + \phi_y - \phi_{y + a}\Big) + \Big(\phi_{y+a} - \phi_y + \phi_z - \phi_{z + a}\Big)\]
    vanishes on $W_i \cap W_j \cap V_{x+a} \cap V_x \cap V_{y+a} \cap V_y \cap V_{z+a} \cap V_z$. But $V_x$ is linear in $x$, so this variety is 
    \[\Big(W_i \cap W_j \cap V_{x+a} \cap V_x \cap V_{z+a} \cap V_z\Big) \cap V_y.\]
    
    This holds for $c'|G_0|$ choices of $y$, so by the application of Theorem~\ref{lin-sys-vanishing} in the definition of variety $W_{i + j}$, we have $\phi_{x+a} - \phi_x + \phi_z - \phi_{z + a} = 0$ on $W_i \cap W_j \cap V_{x+a} \cap V_x \cap V_{z+a} \cap V_z$, so the quadruple lies in $\mathcal{Q}_{i + j}$, as required.

    \noindent\textbf{Finding maps $\psi_1, \dots, \psi_m$.} Apply Theorem~\ref{absg} to find a subset $X' \subseteq X$ of size $(c/2)^{O(1)}|G_0|$ such that for each additive $6d$-tuple $q = a_{[6d]} \in {X'}^{12}$ there exists a  collection $\mathcal{Z}_{6d}^{(q)} \subseteq X^{18d}$ of additive $18d$-tuples in $X$ of size at least $c_1|G_0|^{18d-1}$, where $c_1 \geq (c/2)^{O(1)}$, satisfying the conditions in the conclusion of Theorem~\ref{absg}. 

    \begin{claim}
        When $b_{[18d]} \in \mathcal{Z}_{6d}^{(q)}$, then we have 
        \[\sum_{i \in [6d]} (-1)^i \phi_{a_i} = \sum_{i \in [18d]} (-1)^i \phi_{b_i}\text{ on }W_{36} \cap V_{a_1} \cap V_{a_2} \cap \dots \cap V_{a_{6d-1}} \cap V_{b_1} \cap \dots \cap V_{b_{18d-1}}.\]
    \end{claim}

    \begin{proof}
        Unpacking the definitions of $\mathcal{Z}^{(a_{[\ell]})}_{\ell}$ for $\ell \in [6d]$, the second and the third properties imply that when $y_{[3\ell]} \in \mathcal{Z}^{(a_{[\ell]})}_{\ell}$, we have, for $\tilde{y}_{3\ell - 3} = y_{3\ell-3} - y_{3\ell -2} + y_{3\ell -1} - y_{3\ell} + a_{\ell}$, 
        \[\phi_{\tilde{y}_{3\ell - 3}} = \phi_{y_{3\ell-3}} + \phi_{y_{3\ell - 1} - y_{3\ell} + a_{\ell }} - \phi_{y_{3\ell-2}} = \phi_{y_{3\ell-3}} - \phi_{y_{3\ell-2}} + \phi_{a_\ell} + \phi_{y_{3\ell - 1}} - \phi_{y_{3\ell}}\]
        on $W_{36} \cap V_{a_{\ell}} \cap V_{y_{3\ell}} \cap V_{y_{3\ell - 1}} \cap  V_{y_{3\ell - 2}}\cap  V_{y_{3\ell - 3}}$. Note that the last property implies that $(y_1, \dots, y_{3\ell - 4}, \tilde{y}_{3\ell - 3}) \in \mathcal{Z}_{\ell - 1}^{(a_{[\ell -1]})}$, so we may proceed and obtain the claimed equality. Finally, the first property implies that $b_{[18d]}$ is an additive $18d$-tuple.
    \end{proof}

    Now, apply the same argument with $X'$ instead of $X$ as the starting set. Let us first show that there are still many variety-respected additive quadruples in $X'$. Theorem~\ref{absg} gives us the same conclusion above for additive quadruples in $X'$ rather than additive $6d$-tuples, which we need for the rest of the proof. By averaging, we get a map $\psi : U \to H$ defined on a multilinear variety of codimension at most $(r + \log c^{-1})^{O(1)}$ such that for $(c/2)^{O(1)}|G_0|^3$ additive quadruples $(x+a, x, y+a, y)$ in $X'$ we have
    \[\phi_{x+a} - \phi_x - \phi_{y+a} + \phi_y = \psi\text{ on }U \cap V_x \cap V_y \cap V_a.\]
    After averaging over $y$ and an application of the  Cauchy-Schwarz inequality, we get
    \[\phi_{x+a} - \phi_x - \phi_{z+a} + \phi_z = 0\text{ on }(U \cap V_y) \cap V_x \cap V_z \cap V_a\]
    for $(c/2)^{O(1)}|G_0|^3$ additive quadruples $(x+a, x, z+a, z)$ in $X'$.

    Hence, we may apply the same argument and get a subset $X'' \subseteq X'$ of size $(c/2)^{O(1)}|G_0|$ and a variety $W' \subseteq W$ of codimension $(r + \log c^{-1})^{O(1)}$ such that for each additive $2d$-tuple $q = a_{[2d]} \in {X''}^{2d}$ there exists a  collection $\mathcal{Z}_{2d}^{(q)} \subseteq {X'}^{6d}$ of additive $6d$-tuples in $X'$ of size at least $c_2|G_0|^{6d-1}$, where $c_2 \geq (c/2)^{O(1)}$, satisfying the conditions in the conclusion of Theorem~\ref{absg}. Analogous argument to the one above shows that, when $b_{[6d]} \in \mathcal{Z}_{2d}^{(q)}$, then we have 
        \[\sum_{i \in [2d]} (-1)^i \phi_{a_i} = \sum_{i \in [6d]} (-1)^i \phi_{b_i}\text{ on }W' \cap V_{a_1} \cap V_{a_2} \cap \dots \cap V_{a_{2d-1}} \cap V_{b_1} \cap \dots \cap V_{b_{6d-1}}.\]

    \begin{claim}
        Let $\eta > 0$. There exists a multilinear variety $\tilde{W} \subseteq W'$ and a linear system of varieties $(\tilde{V}_a)_{a \in G_0}$, $\tilde{V}_a \subseteq V_a$, of codimension $(r + \log(2c^{-1}\eta^{-1}))^{O(1)}$ such that the following holds. Let $Z$ be a collection of additive $6d$-tuples in $X'$ of size $\eta |G_0|^{6d-1}$. Then there exists a multilinear map $\psi : \tilde{W} \to H$ such that $\psi = \sum_{i \in [6d]} (-1)^i\phi_{a_i}$ on $\tilde{W} \cap \tilde{V}_{a_1} \cap \dots \cap \tilde{V}_{a_{6d}}$ for at least $(c\eta/2)^{O(1)}|G_0|^{6d-1}$ additive $6d$-tuples in $Z$.
    \end{claim}

    \begin{proof}
        By the property of the set $X'$, for any $a_{[6d]} \in Z$ we have at least $(c \eta)^{O(1)}|G_{0}|^{18d-1}$ additive $18d$-tuples $b_{[18d]}$ in $X$ such that 
        \[ \sum_{i \in [6d]} (-1)^i\phi_{a_i} =  \sum_{i \in [18d]} (-1)^i\phi_{b_i}\text{ on }W' \cap (\cap_{i \in [6d]} V_{a_i}) \cap (\cap_{i \in [18d]} V_{b_i}).\]
        Define a bipartite graph whose vertex classes are $Z$ and the set of additive $18d$-tuples in $X$ and we put an edge between $a_{[6d]}$ and $b_{[18d]}$ if the above equality holds. It has density $(c \eta)^{O(1)}$, so there are $(c \eta)^{O(1)}|G_{0}|^{12d-2}$ pairs of additive $6d$-tuples $(a_{[6d]}, a'_{[6d]})$ in $Z$ that
        \[ \sum_{i \in [6d]} (-1)^i\phi_{a_i} =  \sum_{i \in [6d]} (-1)^i\phi_{a'_i}\text{ on }W' \cap (\cap_{i \in [6d]} V_{a_i})\cap (\cap_{i \in [6d]} V_{a'_i}) \cap (\cap_{i \in [18d]} V_{b_i})\]
        for  $(c \eta)^{O(1)}|G_{0}|^{18d-1}$ additive $18d$-tuples $b_{[18d]}$. Apply Theorem~\ref{lin-sys-vanishing} to pass to a subvariety $W'' \subseteq W'$ of codimension $(r  + \log(2c^{-1}\eta^{-1}))^{O(1)}$, independent of $a_{[6d]}, a'_{[6d]}$, such that 
        \[ \sum_{i \in [6d]} (-1)^i\phi_{a_i} =  \sum_{i \in [6d]} (-1)^i\phi_{a'_i}\text{ on }W'' \cap (\cap_{i \in [6d]} V_{a_i})\cap (\cap_{i \in [6d]} V_{a'_i}).\]
        Apply Theorem~\ref{linsystextns} to get the desired variety $\tilde{W} \subseteq W''$ and a linear system of varieties $(\tilde{V}_a)_{a \in G_0}$ of codimension $ (r  + \log(2c^{-1}\eta^{-1}))^{O(1)}$. The properties of that variety allow us to define the desired $\psi$.
    \end{proof}

    Apply the claim with $\eta = c_2$ to find a variety $\tilde{W}$ and a linear system of varieties $(\tilde{V}_a)_{a \in G_0}$ of codimension at most $\tilde{s} \leq (r  + \log(2c^{-1}))^{O(1)}$. For each additive $2d$-tuple $q = (a_1, \dots, a_{2d})$ in $X''$, the conclusion of the claim applies to the set of additive $6d$-tuples $\mathcal{Z}_{2d}^{(q)} \subseteq {X'}^{6d}$, giving a further subset $\tilde{\mathcal{Z}}^{(q)} \subseteq \mathcal{Z}_{2d}^{(q)}$ of size $|\tilde{\mathcal{Z}}^{(q)}| \geq c_3 |G_0|^{6d-1}$ for $c_3 \geq (c/2)^{O(1)}$ and a map $\psi^{(q)} : \tilde{W} \to H$ such that for each additive $6d$-tuple $b_{[6d]}$ in $\tilde{\mathcal{Z}}^{(q)}$ we have $\psi^{(q)} = \sum_{i \in [6d]} (-1)^i\phi_{b_i}$ on $\tilde{W} \cap \tilde{V}_{b_1} \cap \dots \cap \tilde{V}_{b_{6d}}$.

    Take a maximal set of additive $2d$-tuples $q_1, \dots, q_m$ such that $|\tilde{\mathcal{Z}}^{(q_i)} \cap \tilde{\mathcal{Z}}^{(q_j)}| \leq \frac{c_3^2}{2}|G_0|^{6d-1}$, whenever $i \not = j$. Clearly, $m \leq (c/2)^{-O(1)}$. Write $\psi_i = \psi^{(q_i)}$. Apply Theorem~\ref{lin-sys-vanishing} another time to the system $(\tilde{V}_x)_{x \in G_0}$, codimension parameter at most $5\tilde{s}$ and density parameter at least $c_3^2/2$  to find a multilinear variety $\tilde{W}' \subseteq \tilde{W}$ of codimension $(\tilde{s} + r + \log c_3^{-1})^{O(1)} \leq (r  + \log(2c^{-1}))^{O(1)}$ with properties in the conclusion of the theorem.

    \indent Let $a_{[2d]}$ be an additive $2d$-tuple in $X''$. Then, for some $j \in [m]$, we have $|\tilde{\mathcal{Z}}^{(q)} \cap \tilde{\mathcal{Z}}^{(q_j)}| \geq \frac{c_3^2}{2}|G_0|^{6d-1}$. For each additive $6d$-tuple $b_{[6d]} \in \tilde{\mathcal{Z}}^{(q)} \cap \tilde{\mathcal{Z}}^{(q_j)}$, we have
    \[\sum_{i \in [2d]} (-1)^i \phi_{a_i} = \sum_{i \in [6d]} (-1)^i \phi_{b_i} = \psi_j\]
    on $\tilde{W} \cap \tilde{V}_{a_1} \cap \dots \cap \tilde{V}_{a_{2d}} \cap \tilde{V}_{b_1} \cap \dots \cap \tilde{V}_{b_{6d}}$. By the property of $\tilde{W}'$ guaranteed by Theorem~\ref{lin-sys-vanishing}, we have that $\sum_{i \in [2d]} (-1)^i \phi_{a_i} =  \psi_j$ on $\tilde{W}' \cap \tilde{V}_{a_1} \cap \dots \cap \tilde{V}_{a_{2d}}$, as claimed.
\end{proof}

\subsection{Robust Bogolyubov-Ruzsa step 2}

In the next step, we apply the robust Bogolyubov-Ruzsa theorem to ensure that we have maps $\phi_a$ defined for all elements of a subspace of $G_0$ and all additive $O(1)$-tuples variety-respected, up to an error function.

\begin{proposition}\label{rob-bog-ruzsa-step2}
    Let $d = O(1)$. Let $(\phi_x : W \cap V_x \to H)_{x \in X}$ be a system of MM-$r$-maps, where $W \subseteq G_{[k]}$ is a multilinear variety, $V_x$ depend linearly on $x$ and $|X| \geq c|G_0|$. Let $\psi_1, \dots, \psi_m : W \to H$ be multilinear maps. Suppose that for each additive $8d$-tuple $x_{[8d]}$ in $X$ we have some $j \in [m]$ such that
    \begin{equation}\sum_{i \in [8d]} (-1)^i \phi_{x_i} - \psi_j\Big|_{W \cap V_{x_1} \cap \dots  \cap V_{x_{8d}}} = 0.\label{functionalrespcon-br2}\end{equation}
    Then there exist a subspace $U \leq G_0$ of codimension $\log^{O(1)}c^{-1}$, a multilinear variety $W'$, a linear system $(V_x')_{x\in G_0}$ of codimension $r' \leq (r + \log (2c^{-1}m))^{O(1)}$ and a system of MM-$r'$-maps $(\theta_a : W' \cap V'_x \to H)_{a \in U}$ such that for each $a \in U$ we have $\theta_a = \phi_{x_1} + \phi_{x_2} - \phi_{x_3} - \phi_{x_1 + x_2 - x_3 - a}$ on $W' \cap V'_a \cap V'_{x_1} \cap  V'_{x_2} \cap  V'_{x_3}$ for at least $(c/m)^{O(1)}|G_0|^3$ choices of $(x_1, x_2, x_3)$  
    and for every additive $2d$-tuple $x_{[2d]}$ in $U$ we have some $j \in [m]$ with
    \[\sum_{i \in [2d]} (-1)^i \theta_{x_i} - \psi_j\Big|_{W' \cap V'_{x_1} \cap \dots  \cap V'_{x_{2d}}} = 0.\]
\end{proposition}

\begin{rem*} The error functions $\psi$ are unchanged in the proof, the only difference is that for additive $2d$-tuples and maps $(\theta_a)_{a \in U}$, we use the error functions for the case of additive $8d$-tuples and maps $(\phi_x)_{x \in X}$.\end{rem*}

\indent Eventually, the proposition will be applied with $d = 8$.

\begin{proof}
By the robust Bogolyubov-Ruzsa theorem, there exists a subspace $U \leq G_0$ of codimension at most $\log^{O(1)}c^{-1}$ such that for each $a \in U$, there are at least $(c/2)^{O(1)}|G_0|^3$ triples $(x_1, x_2, x_3)$ such that $x_1, x_2, x_3, x_1 + x_2 - x_3 - a \in X$. Denote the set of such triples by $T_a$.

Fix an element $a \in U$. Note that, when $(x_1, x_2, x_3), (y_1, y_2, y_3) \in T_a$, then 8 elements $x_1, x_2, x_3, x_1 + x_2 - x_3 - a, y_1, y_2, y_3, y_1 + y_2 - y_3 - a$ form an additive 8-tuple. Hence, there exists $j \in [m]$ such that~\eqref{functionalrespcon-br2} holds. In particular, by averaging, there exists some $j \in [m]$ such that this equality holds for at least $\frac{(c/2)^{O(1)}}{m} |G_0|^6$ pairs in $T_a$.

For the fixed element $a$, consider the bipartite graph whose vertex classes are copies of $G_0^3$, and triples $(x_1, x_2, x_3)$ and  $(y_1, y_2, y_3)$ are joined by an edge if the above equality holds with $\psi_j$. This is a graph of density $(c/m)^{O(1)}$, so we get $(c/m)^{O(1)}|G_0|^9$ of 9-tuples $(x_{[3]}, x'_{[3]}, y_{[3]}) \in T_a^3$ such that $x_{[3]}y_{[3]}$ and $x'_{[3]}y_{[3]}$ are both edges. In particular, we have a set $P_a \subseteq T_a \times T_a$ of pairs of triples $(x_{[3]}, x'_{[3]})$ of size $c_1|G_0|^6$, where $c_1 \geq (c/m)^{O(1)}$, such that there are at least $c_1|G_0|^3$ triples $y_{[3]} \in T_a$ with $x_{[3]}y_{[3]}$ and $x'_{[3]}y_{[3]}$ both being edges.

\begin{claim}
    There exist a multilinear variety $W'$ and a linear system $(V_x')_{x\in G_0}$ of codimension  $(r + \log (2c^{-1}m))^{O(1)}$, independent of $a$, such that for each $a \in U$, we have a multilinear map $\theta_a : W' \cap V'_a \to H$ such that
    \[\theta_a = \phi_{x_1} + \phi_{x_2} - \phi_{x_3} - \phi_{x_1 + x_2 - x_3 - a}\text{ on }W' \cap V_a \cap V'_{x_1} \cap V'_{x_2} \cap V'_{x_3}\]
    for  $(c/m)^{O(1)}|G_0|^3$ choices of $x_{[3]} \in T_a$.
\end{claim}

\begin{proof}
    For a given $a \in U$, take any $(x_{[3]}, x'_{[3]}) \in P_a$. Note that
    \[\phi_{x_1} + \phi_{x_2} - \phi_{x_3} - \phi_{x_1 + x_2 - x_3 - a} = \phi_{y_1} + \phi_{y_2} - \phi_{y_3} - \phi_{y_1 + y_2 - y_3 - a}  + \psi_j\text{ on }W \cap V_a \cap \Big(\cap_{i \in [3]} V_{x_i}\Big)\cap \Big(\cap_{i \in [3]} V_{y_i}\Big)\]
    and similarly
    \[\phi_{x'_1} + \phi_{x'_2} - \phi_{x'_3} - \phi_{x'_1 + x'_2 - x'_3 - a} = \phi_{y_1} + \phi_{y_2} - \phi_{y_3} - \phi_{y_1 + y_2 - y_3 - a}  + \psi_j\text{ on }W \cap V_a \cap \Big(\cap_{i \in [3]} V_{x'_i}\Big)\cap \Big(\cap_{i \in [3]} V_{y_i}\Big).\]
    Hence
    \begin{align*}\phi_{x_1} + \phi_{x_2} - \phi_{x_3} - \phi_{x_1 + x_2 - x_3 - a} = &\phi_{x'_1} + \phi_{x'_2} - \phi_{x'_3} - \phi_{x'_1 + x'_2 - x'_3 - a} \\
    &\text{ on }W \cap V_a \cap \Big(\cap_{i \in [3]} V_{x_i}\Big) \cap \Big(\cap_{i \in [3]} V_{x'_i}\Big)\cap \Big(\cap_{i \in [3]} V_{y_i}\Big)\end{align*}
    for $c_1|G_0|^3$ choices of $y_{[3]}$. By Theorem~\ref{lin-sys-vanishing} for the system $(V_x)_{x \in G_0}$, the density parameter $c_1$ and the codimension parameter $8r$, we have a variety $W'$ of codimension $r' \leq (r + \log (2c^{-1}m))^{O(1)}$, independent of $a$, such that 
    \[\phi_{x_1} + \phi_{x_2} - \phi_{x_3} - \phi_{x_1 + x_2 - x_3 - a} = \phi_{x'_1} + \phi_{x'_2} - \phi_{x'_3} - \phi_{x'_1 + x'_2 - x'_3 - a} \text{ on }W' \cap V_a \cap \Big(\cap_{i \in [3]} V_{x_i}\Big) \cap \Big(\cap_{i \in [3]} V_{x'_i}\Big)\]
    holds for each pair $(x_{[3]}, x'_{[3]}) \in P_a$.

    Apply Theorem~\ref{linsystextns} to the linear system $(V_x)$, codimension parameter $r' + 3r$ and density parameter $c_1$, to get further multilinear variety $Z$ and a further linear system of subvarieties $(\tilde{V}_x)_{x \in G_0}$ of codimension $(r + \log (2c^{-1}m))^{O(1)}$, with properties described in the theorem. Apply it to common variety $W' \cap V_a$ and maps $\tilde{\phi}_{x_{[3]}} = \phi_{x_1} + \phi_{x_2} - \phi_{x_3} - \phi_{x_1 + x_2 - x_3 - a} : W' \cap V_a \cap \Big(\cap_{i \in [3]} V_{x_i}\Big) \to H$. Properties in the conclusion of Theorem~\ref{linsystextns} imply that there exists $\theta_a : Z \cap V_a \to H$ such that $\theta_a = \phi_{x_1} + \phi_{x_2} - \phi_{x_3} - \phi_{x_1 + x_2 - x_3 - a}$ on $Z \cap V_a \cap V'_{x_1} \cap  V'_{x_2} \cap  V'_{x_3}$ for at least $(c/m)^{O(1)}|G_0|^3$ choices of $(x_1, x_2, x_3) \in T_a$.
\end{proof}

Finally, given any additive $2d$-tuple $a_{[2d]}$ in $U$, we have $(c/m)^{O(1)}|G_0|^{6d}$ many $x_{[2d] \times [3]}$ in $G_0$ such that 
\[\sum_{i \in [2d]} (-1)^i \theta_{a_i} = \sum_{i \in [2d]} (-1)^i\Big(\phi_{x_{i, 1}} + \phi_{x_{i, 2}} - \phi_{x_{i, 3}} - \phi_{x_{i, 1} + x_{i, 2} - x_{i, 3} - a_i}\Big) = \psi_{j'}\]
on $Z \cap V'_{a_1} \cap V'_{a_2} \cap \dots \cap V'_{a_{2d}} \cap \Big(\cap_{i \in [2d], j \in [3]} V'_{x_{i, j}}\Big)$, for some $j' \in [m]$, as additive $8d$-tuples are respected in the system $(\phi_x)_{x \in X}$, up to error functions. Apply Theorem~\ref{lin-sys-vanishing} another time to complete the proof.
\end{proof}

\subsection{Getting variety-respected additive quadruples}

We now apply an algebraic dependent random choice argument, similar in spirit to the \textbf{Step 1.2} in Subsection~\ref{ensuringrespStep}. In order to carry out the argument, we need to find points where the error functions $\psi_1, \dots, \psi_m$ do not vanish. However, this is an important difficulty, as MM-$r$-maps can have zero sets covering almost the whole of the domain. To overcome this issue, we rely on the dichotomy in Proposition~\ref{nonvanishing-or-variety-dichotomy}.

\begin{proposition}\label{ph4step3}
    Let $d = O(1)$. Let $(V_x)_{x \in G_0}$ be a linear system of multilinear varieties of codimension $r$, let $W$ be a multilinear variety of codimension $r$ in $G_{[k]}$ and let $\phi_x : W \cap V_x \to H$ be a multilinear map for each $x \in G_0$. Suppose that $\psi_1, \dots, \psi_m : W \to H$ are multilinear maps such that for each additive $2d$-tuple $x_{[2d]}$ in $G_0$, we have
    \[\sum_{i \in [2d]} (-1)^i \phi_{x_i} - \psi_j\Big|_{W \cap V_{x_1} \cap \dots  \cap V_{x_{2d}}} = 0.\]
    for some $j$. Then there exist a set $X \subseteq G_0$ of size $|X| \geq p^{-(\log m + r)^{O(1)}}|G_0|$, such that all additive $2d$-tuples in $X$ are variety-respected.
\end{proposition}

This proposition will be applied with $d= 8$.

\begin{proof}
    Apply Proposition~\ref{nonvanishing-or-variety-dichotomy} iteratively to the sequence of maps $\psi_1, \dots, \psi_m$. Each time, we get one of the two conclusions for at least $\Omega(1)$ proportion of maps, which we then remove, and apply the proposition again. We exhaust the list of maps after $O(\log m)$ steps. Thus, for some $\ell \leq O(\log m)$, we obtain points $x^{(1)}_{[k]}, \dots, x^{(\ell)}_{[k]} \in W$ and multilinear varieties $Z_1, \dots, Z_\ell \subseteq W$ of codimension $r^{O(1)}$ such that for each $\psi_j$ we either have some $x^{(i)}_{[k]}$ where it is non-zero, or $Z_i \subseteq \on{Z}(\psi_j)$. Write $Z = Z_1 \cap \dots \cap Z_\ell$ which is a multilinear variety of codimension at most $\ell r^{O(1)}$. Let $J$ be the indices $j \in [m]$ such that $\psi_j(x^{(i)}_{[k]})\not=0$ for some $i \in [\ell]$.

    Let $h_1, \dots, h_n \in H$ be elements such that for each $j \in J$, taking some $i \in [\ell]$ such that $\psi_j(x^{(i)}_{[k]})\not= 0$, we have  $\psi_j(x^{(i)}_{[k]}) \cdot h_{i'} \not= 0$, for some $i' \in [n]$. Clearly, we may find such elements with $n = O(\log m)$.

    Let $U$ be the subspace of all $a \in G_0$ such that $x^{(1)}_{[k]}, \dots, x^{(\ell)}_{[k]} \in V_a$. Take $\mu \in \mathbb{F}_p^{[\ell] \times [n]}$ uniformly at random and define $X = \{a \in U: (\forall i \in [\ell], j \in[n])\,\, \phi_a(x^{(i)}_{[k]}) \cdot h_j = \mu_{i,j}\}$.

    \begin{claim}
        For any choice of $\mu$, we have that whenever $a_1, \dots, a_{2d}$ is an additive $2d$-tuple in $X$, then
        \[\sum_{i \in [2d]} (-1)^i \phi_{a_i} = 0\text{ on }Z \cap V_{a_1} \cap \dots \cap V_{a_{2d}}.\]
    \end{claim}

    \begin{proof}
        We have $\sum_{i \in [2d]} (-1)^i \phi_{a_i} = \psi_j$ for some $j \in [m]$ on $W \cap V_{a_1} \cap \dots  \cap V_{a_{2d}}$. If $j \notin J$, then $Z \subseteq \on{Z}(\psi_j)$, so we are done. Otherwise, we have some $\iota$ such that $\psi_j(x^{(\iota)}_{[k]}) \not=0$. But $x^{(\iota)}_{[k]} \in V_{a_1} \cap \dots \cap V_{a_{2d}}$, so we have
        \[\sum_{i \in [2d]} (-1)^{i} \phi_{a_i}(x^{(\iota)}_{[k]}) = \psi_j(x^{(\iota)}_{[k]}).\]
        Taking dot product with $h_{i'}$ such that $\psi_j(x^{(\iota)}_{[k]}) \cdot h_{i'} \not = 0$ gives
        \[0 = \sum_{i \in [2d]} (-1)^{i} \mu_{\iota i'} = \sum_{i \in [2d]} (-1)^{i} \phi_{a_i}(x^{(\iota)}_{[k]}) \cdot h_{i'} = \psi_j(x^{(\iota)}_{[k]}) \cdot h_{i'} \not = 0,\]
        which is a contradiction.
    \end{proof}

    Now pick $\mu$ so that $X$ is large; we may ensure that $|X| \geq p^{-\ell n}|U| \geq p^{-\ell n - \ell r} |G_0| \geq p^{-(\log m + r)^{O(1)}}|G_0|$.
\end{proof}

\subsection{Final Bogolyubov-Ruzsa argument}

We apply the Bogolyubov-Ruzsa argument once again the obtain the desired completion of the system of MM-$r$-maps. We simply need the special case of Proposition~\ref{rob-bog-ruzsa-step2} when all $\psi_i$ are zero maps, which we record here.

\begin{proposition}\label{ph4step4}
    Let $(V_x)_{x \in G_0}$ be a linear system of multilinear varieties of codimension $r$, let $W$ be a multilinear variety of codimension $r$ in $G_{[k]}$ and let $\phi_x : W \cap V_x \to H$ be a multilinear map for each $x \in X$. Suppose that all additive 16-tuples in $X$ are variety-respected. Then there exist a subspace $U \leq G_0$ of codimension $\log^{O(1)}c^{-1}$, a multilinear variety $W'$, a linear system $(V_x')_{x\in G_0}$ of codimension $r' \leq (r + \log (2c^{-1}))^{O(1)}$, a system of MM-$r'$-maps $(\theta_a : W' \cap V'_x \to H)_{a \in U}$ such that for each $a \in U$ we have $\theta_a = \phi_{x_1} + \phi_{x_2} - \phi_{x_3} - \phi_{x_1 + x_2 - x_3 - a}$ on $W' \cap V'_a \cap V'_{x_1} \cap  V'_{x_2} \cap  V'_{x_3}$ for at least $(c/2)^{O(1)}|G_0|^3$ choices of $(x_1, x_2, x_3)$  
    and for every additive quadruple $a_1, \dots, a_4$ in $U$
    \[\sum_{i \in [4]} (-1)^i \theta_{a_i} \Big|_{W' \cap V'_{a_1} \cap \dots  \cap V'_{a_{4}}} = 0.\]
\end{proposition}

Combining all steps in this section, we deduce Theorem~\ref{phase4mainres}.

\begin{proof}[Proof of Theorem~\ref{phase4mainres}]
    Apply Proposition~\ref{ph4step1} with $d = 32$ to get all additive 64-tuples variety-respected up to error functions, then Proposition~\ref{rob-bog-ruzsa-step2} with its parameter $d$ chosen to be 8, resulting in additive 16-tuples being variety-respected up to error functions in a system of maps indexed by a full subspace, then Proposition~\ref{ph4step3} with its parameter $d$ also being 8 to remove the error functions, and finally Proposition~\ref{ph4step4}. Thus, there exist a positive integer $s \leq (r + \log c^{-1})^{O(1)}$, a subspace $S \leq G_0$, a multilinear variety $W'$, a linear system of varieties $(V_x')_{x\in G_0}$ of codimension $s$, a system of MM-maps $(\theta_a : W' \cap V'_x \to H)_{a \in S}$, such that for each $a \in S$ we have at least $\exp(-s)|G_0|^{15}$ choices of $y_{[16]}$ in $X$ such that $\sum_{i \in [16]} (-1)^i y_i = a$ and $\theta_a = \sum_{i \in [16]} (-1)^i \phi_{y_i}$ on $W' \cap V'_a \cap (\cap_{i \in [16]}V'_{y_i})$. Recall that  a linear system of varieties $(V_x')_{x\in G_0}$ is in fact a collection of slices of a multilinear variety $V'$. Thus, we may define a multilinear variety $U = (S \times G_{[k]}) \cap (G_0 \times W') \cap V'$. Define map $\Phi : U \to H$ by $\Phi(a, x_{[k]}) = \theta_a(x_{[k]})$. Clearly, it is multilinear in $G_{[k]}$. For the direction $G_0$, take an additive quadruple in $U$ in direction $G_0$, namely points $(a_i, y_{[k]}) \in U$, $i \in [4]$, with $a_1 +a_2 = a_3 + a_4$. Then $y_{[k]} \in U_{a_1} \cap \dots \cap U_{a_4} \subseteq W' \cap V'_{a_1} \cap \dots \cap V'_{a_4}$ and the additive quadruple $(\upd{a_1}, \upd{a_2}, a_3, a_4)$ is variety-respected by $\theta$, so we obtain
    \[\Phi(a_1, y_{[k]}) + \Phi(a_2, y_{[k]}) - \Phi(a_3, y_{[k]}) - \Phi(a_4, y_{[k]}) = \theta_{a_1}(y_{[k]}) + \theta_{a_2}(y_{[k]}) - \theta_{a_3}(y_{[k]}) - \theta_{a_4}(y_{[k]}) = 0,\]
    showing that $\Phi$ is affine in direction $G_0$, completing the proof.
\end{proof}

\section{Putting everything together}

In this section we prove the main result of the paper, the quasipolynomial structure theorem for Freiman multi-homomorphisms, which was stated as Theorem~\ref{strFmult}. Before combining all steps together, we need a lemma concerning MM-$r$-maps with dense zero sets.

\begin{lemma}\label{mmrmapszerosetslemma}
    Let $V\subseteq G_{[k]}$ be a multilinear variety of codimension $r$ and let $\phi : V \to H$ be an MM-$r$-map such that $|Z(\phi)| \geq c|V|$. Then $Z(\phi)$ contains a multilinear variety of codimension at most $(r + \log c^{-1})^{O(1)}$.
\end{lemma}

\begin{proof}
    By Theorem~\ref{basicextensionstheory}, there exists a global multilinear map $\Phi : G_{[k]} \to H$ and a multilinear variety $W \subseteq V$ of codimension $s \leq (2r)^{O(1)}$ such that $\phi = \Phi$ on $W$. Let $\alpha : G_{[k]} \to \mathbb{F}_p^s$ be the mixed-linear map defining $W$. By the pigeonhole principle, for some $\lambda \in \mathbb{F}_p^s$, we have $|Z(\phi) \cap \{\alpha = \lambda\}| \geq c|\{\alpha = \lambda\}| \geq c p^{-ks} |G_{[k]}|$. Iterating the Cauchy-Schwarz inequality, we may find at least $(c p^{-ks})^{O(1)}$ pairs $(y_{[k]}, z_{[k]}) \in G_{[k]} \times G_{[k]}$ such that $(y_I, z_{[k] \setminus I}) \in Z(\phi) \cap \{\alpha = \lambda\}$ for all subsets $I \subseteq [k]$. Since $\phi$ is defined on the multilinear variety $V$, we get that point $(y_i-z_i)_{i \in [k]}$ belongs to $V$ and $\phi((y_i-z_i)_{i \in [k]}) = 0$. Additionally, $\alpha((y_i-z_i)_{i \in [k]}) = 0$, so it belongs to $W$. Hence, $|Z(\phi) \cap W| \geq (c p^{-ks})^{O(1)}|G_{[k]}|$.

    \indent Since $\phi = \Phi$ on $W$, we have $|Z(\Phi)| \geq (c p^{-ks})^{O(1)}|G_{[k]}|$. As $\Phi$ is a global multilinear map, $Z(\Phi)$ is a multilinear variety so by Theorem~\ref{densetolowcodim}, it contains a multilinear variety $W'$ of codimension $(s + \log c^{-1})^{O(1)} \leq (r + \log c^{-1})^{O(1)}$. Hence, $W \cap W' \subseteq Z(\phi)$, as desired.
\end{proof}

We may now proceed to prove Theorem~\ref{strFmult}.

\begin{proof}[Theorem~\ref{strFmult}] We prove the theorem by induction on $k$. The base case follows from Theorem~\ref{invhomm}. Suppose now that the theorem holds for some $k \geq 1$ and let $\Phi: A \to H$ be a Freiman multi-homomorphism on a set $A \subseteq G_{[0,k]}$. By Proposition~\ref{mlmaps-pass-step-proposition} we obtain a quantity $c_0 \geq \eplog{c}$, a set $X_0 \subseteq G_0$ and a collection of global multilinear map $\phi^{(0)}_x : G_{[k]} \to H$ such that $|Z(\phi^{(0)}_{x_1} - \phi^{(0)}_{x_2} + \phi^{(0)}_{x_3} - \phi^{(0)}_{x_4})| \geq c_0|G_{[k]}|$ holds for at least $c_0|G_0|^3$ additive quadruples $(\upd{x}_1, x_2, \upd{x}_3, x_4)$ in $X_0$. Moreover, there exist a global multiaffine map $\psi^{(0)} : G_{[0, k]} \to H$ and a point $t_{[k]} \in G_{[k]}$ such that for each $x \in X_0$, there exists a set $Y^{(0)}_x \subseteq G_{[k]}$ of size $c_0|G_{[k]}|$ such that
    \begin{equation}\label{PET:eqn1}(\forall y_{[k]} \in Y^{(0)}_x)\,\,\phi^{(0)}_x(y_{[k]}) = \Phi(x, (y + t)_{[k]}) + \psi^{(0)}(x, y_{[k]}).\end{equation}

Apply Theorem~\ref{phase1mainres}. We obtain a quantity $r_1 \leq \log^{O(1)}(2c_0^{-1}) \leq \log^{O(1)}(2c^{-1})$, subspace $U_1 \leq G_0$ of codimension $r_1$, a system of global multilinear maps  $(\phi^{(1)}_x : G_{[k]} \to H)_{x \in U_1}$ such that every additive quadruple in $(\phi^{(1)}_x)_{x \in U_1}$ is $\exp(-r_1)$-respected and for each $x \in U_1$ we have  at least  $\exp(-r_1)|G_0|^3$ quadruples $(x_1 ,x_2 , x_3 , x_4) \in X_0^4$ such that $x_1 + x_2 - x_3 - x_4 = x$ and 
    \begin{equation}\label{PET:eqn2}|Z(\phi^{(1)}_x - \phi^{(0)}_{x_1} - \phi^{(0)}_{x_2} + \phi^{(0)}_{x_3} + \phi^{(0)}_{x_4})| \geq \exp(-r_1)|G_{[k]}|.\end{equation}

Let $d_{\on{bog}} = O(1)$ and $\varepsilon_{\on{bog}} = \Omega(1)$ be the quantities in Theorem~\ref{multbogstep}. Apply Theorem~\ref{changecategorystep} to get a quantity $r_2 \leq \log^{O(1)}(c^{-1})$, multilinear varieties $V^{(2)}_x$ for $x \in G_0$ of codimension at most $r_2$, such that $(1-\varepsilon_{\on{bog}})|G_0|^{2d_{\on{bog}}-1}$ additive $2d_{\on{bog}}$-tuples in $G_0$ are variety-respected in the system $(\phi^{(1)}_x : V^{(2)}_x \to H)_{x \in U_1}$.

Apply Theorem~\ref{multbogstep} to find a positive quantity $c_3 \geq \exp(-(2r_2)^{O(1)}) \geq \exp(-\log^{O(1)}(c^{-1}))$ and a integer $s_3 \leq (2r_2)^{O(1)} \leq \log^{O(1)}(c^{-1})$, another system $(\phi^{(3)}_x : W^{(3)}\cap V^{(3)}_x \to H)_{x \in X_3}$ of MM-$s_3$-maps such that $c_3|U_1|^3 \geq \exp(-\log^{O(1)}(c^{-1}))|G_0|^3$ additive quadruples in $X_3$ are variety-respected, where $W^{(3)} \subseteq G_{[k]}$ is a multilinear variety, independent of $x$, and $V^{(3)}_x\subseteq G_{[k]}$ depending linearly on $x$. Moreover, for each $x \in X_3$, we have 
    \begin{equation}\label{PET:eqn3}\Big|Z\Big(\phi^{(3)}_x - \sum_{i \in [2d_{\on{bog}}]} (-1)^i \phi^{(1)}_{y_i}\Big)\Big|\geq c_3|G_{[k]}|\end{equation}
for at least $\exp(-s_3)|U_1|^{2d_{\on{bog}}-1} \geq \exp(-\log^{O(1)}(c^{-1}))|G_0|^{2d_{\on{bog}}-1}$ choices of $y_{[2d_{\on{bog}}]}$ in $U_1$ with the property $\sum_{i \in [2d_{\on{bog}}]} (-1)^i y_i = x$.

Apply Theorem~\ref{phase4mainres} to find a quantity $s_4 \leq (s_3 + \log c_3^{-1})^{O(1)} \leq \log^{O(1)}(2c^{-1})$, variety $U^{(4)} \subseteq G_{[0, k]}$ of codimension $s_4$ and a map $\Psi : U^{(4)} \to H$, affine in direction $G_0$ and multilinear in directions $G_{[k]}$, such that for each $x \in (U^{(4)})^{(\{0\})}$, there exists at least $\exp(-s_4)|G_0|^{15}$ 16-tuples $y_{[16]}$ in $X_3$ such that $\sum_{i \in [16]} (-1)^i y_i = x$ and
    \begin{equation}\label{PET:eqn4}\Psi_x = \sum_{i \in [16]} (-1)^i \phi^{(3)}_{y_i}\end{equation}
holds on the variety $(U^{(4)})^{[k]} \cap U^{(4)}_x \cap (\cap_{i \in [16]} U^{(4)}_{y_i})$. By Theorem~\ref{basicextensionstheory}, at the cost of passing to a slightly smaller multilinear variety inside $U^{(4)}$, we may without loss of generality assume that $\Psi$ is a global multiaffine map.

Observe that if $\gamma_1 : S_1 \to H$ and $\gamma_2 : S_2 \to H$ are two MM-$t$-maps with $|Z(\gamma_1)|, |Z(\gamma_2)|\geq \varepsilon |G_{[k]}|$, then Lemmas~\ref{mmrmapszerosetslemma} and~\ref{poscorvars} imply that $|Z(\gamma_1 + \gamma_2)| \geq \exp(-(t + \log \varepsilon^{-1})^{O(1)})|G_{[k]}|$. This observation allows us to combine equalities~\eqref{PET:eqn2},~\eqref{PET:eqn3} and~\eqref{PET:eqn4} and thus obtain some $\ell = O(1)$, a set $\mathcal{A} \subseteq G_0 \times X_0^{2\ell}$ of size $\exp(-\log^{O(1)}(2c^{-1}))|G_0|^{2\ell}$ whose members $(a, x_1, \dots, x_{2\ell})$ satisfy $a = \sum_{i \in [2\ell]} (-1)^i x_i$ and
\[\Psi(a, y_{[k]}) = \sum_{i \in [2\ell]}  (-1)^i \phi^{(0)}_{x_i}(y_{[k]})\]
for at least $\exp(-\log^{O(1)}(2c^{-1}))|G_{[k]}$ choices of $y_{[k]} \in G_{[k]}$. By averaging, there exist $x_1, \dots, x_{2\ell - 1}$ such that, writing $a_0 = \sum_{i \in [2\ell-1]} (-1)^ix_i$ and $\tilde{\phi} =  \sum_{i \in [2\ell-1]}  (-1)^i \phi^{(0)}_{x_i}$,
\[\Psi(a, y_{[k]}) = \phi^{(0)}_{a_0 + a}(y_{[k]}) + \tilde{\phi}(y_{[k]})\]
and $a_0 + a \in X_0$ hold for $\exp(-\log^{O(1)}(2c^{-1}))|G_{[0,k]}|$ choices of $(a, y_{[k]}) \in G_{[0,k]}$. Let $Y'_{a}$ be the set of all $y_{[k]}$ such that the equality above holds for the given $a$. As $a_0 + a \in X_0$, recalling~\eqref{PET:eqn1}, we have $|Y^{(0)}_{a_0 + a}| \geq \exp(-\log^{O(1)}(2c^{-1}))|G_{[k]}|$. By averaging, there are at least $\exp(-\log^{O(1)}(2c^{-1}))|G_{[k]}|$ choices of $b_{[k]} \in G_{[k]}$ such that $|(b_{[k]} + Y^{(0)}_{a_0 + a}) \cap Y'_a| \geq \exp(-\log^{O(1)}(2c^{-1}))|G_{[k]}|$. By~\eqref{PET:eqn1}, we get
\begin{align*}\Psi(a, (y+b)_{[k]}) - \tilde{\phi}((y+b)_{[k]}) =& \phi^{(0)}_{a_0 + a}((y+b)_{[k]}) = \sum_{I \subseteq [k]} \phi^{(0)}_{a_0 + a}(y_I, b_{I^c})\\
= &\phi^{(0)}_{a_0 + a}(y_{[k]}) + \sum_{I \subsetneq [k]} \phi^{(0)}_{a_0 + a}(y_I, b_{I^c}) \\
=& \Phi(a_0 + a, (y + t)_{[k]}) + \psi^{(0)}(a_0 + a, y_{[k]}) + \sum_{I \subsetneq [k]} \phi^{(0)}_{a_0 + a}(y_I, b_{I^c}).\end{align*}

Fix $b_{[k]}$ such that the above equality holds for $\exp(-\log^{O(1)}(2c^{-1}))|G_{[0, k]}|$ choices of $(a,y_{[k]})$. Gowers-Cauchy-Schwarz argument in coordinates  $G_{[k]}$ implies that, for the multilinearizations in coordinates $G_{[k]}$ denoted by $\Psi^{\on{ml}}$, $\tilde{\phi}^{\on{ml}}$ and $\tilde{\phi}^{\on{ml}}$,
\[\Psi^{\on{ml}}(a, (y-z)_{[k]}) - \tilde{\phi}^{\on{ml}}((y-z)_{[k]}) - (\psi^{(0)})^{\on{ml}}(a_0 + a, (y-z)_{[k]}) = \sum_{I \subseteq [k]} \Phi(a_0 + a, (y + t)_{I}, (z + t)_{I^c})\]
$\exp(-\log^{O(1)}(2c^{-1}))|G_0||G_{[k]}|^2$ choices of $(a,y_{[k]}, z_{[k]})$. The theorem follows by averaging over $z_{[k]}$ and inductive hypothesis applied to replace the terms $\Phi(a_0 + a, (y + t)_{I}, (z + t)_{I^c})$ by multiaffine maps on $G_{\{0\} \times I}$ for $I \subsetneq [k]$.\end{proof}

\printbibliography

\end{document}